\documentclass[11pt]{amsart}

\usepackage{amsmath,amsthm,amssymb,mathrsfs,mathtools,color}
\usepackage{bm} 
\usepackage{stmaryrd}

\usepackage[top=1in, bottom=1in, left=1in, right=1in]{geometry}

\usepackage{float}
\usepackage{graphicx}
\usepackage{etoolbox}

\usepackage[utf8]{inputenc}
\usepackage[T1]{fontenc}

\usepackage{enumitem}
\setlist[enumerate]{itemsep=2pt,parsep=2pt,before={\parskip=2pt}}
\setlist[itemize]{itemsep=2pt,parsep=2pt,before={\parskip=2pt}}
\usepackage{longtable}

\usepackage[colorlinks=true,hyperindex, linkcolor=red!60, pagebackref=false, citecolor=cyan, pdfpagelabels, hypertexnames=false]{hyperref}
\usepackage[capitalize]{cleveref}
\crefname{equation}{}{}
\crefformat{enumi}{#2\textup{\textcolor{black}{(}#1\textcolor{black}{)}}#3}

\usepackage{appendix}
\AtBeginEnvironment{appendices}{\crefalias{section}{appendix}}

\usepackage{url}

\usepackage{tikz}
\usetikzlibrary{calc}
\usetikzlibrary{math}
\usetikzlibrary{arrows.meta}

\usepackage{pdfpages}

\usepackage{thmtools}
\usepackage{thm-restate}

\newtheorem{theorem}{Theorem}[section]
\newtheorem*{theorem*}{Theorem}
\newtheorem*{definition*}{Definition}
\newtheorem{proposition}[theorem]{Proposition}
\newtheorem{lemma}[theorem]{Lemma}

\theoremstyle{definition}
\newtheorem{definition}[theorem]{Definition}

\newtheorem{remark}[theorem]{Remark}

\newtheorem{notation}[theorem]{Notation}

\AtEndEnvironment{definition}{\hfill$\triangleleft$}

\renewcommand{\leq}{\leqslant}

\renewcommand{\geq}{\geqslant}

\newcommand{\interval}[2]{\llbracket#1,#2\rrbracket}
\newcommand{\sinterval}[1]{\llbracket#1\rrbracket}
\newcommand{\norm}[1]{\left\|#1\right\|}

\newcommand{\abs}[1]{\left|#1\right|}
\newcommand{\ind}[1]{\mathbf{1}_{#1}}

\newcommand{\eps}{\varepsilon}
\newcommand{\R}{\mathbb{R}}
\newcommand{\Z}{\mathbb{Z}}
\newcommand{\C}{\mathbb{C}}
\newcommand{\E}{\mathbb{E}}

\newcommand{\Tr}{\mathrm{Tr}}
\newcommand{\rank}{\mathrm{rank}_{\d}}
\newcommand{\spmod}[1]{\, (\text{mod}\, #1)}
\newcommand{\bigmid}{\  \Big\vert}

\newcommand{\N}{I_{\mkern -1mu N}}
\newcommand{\primes}{\mathcal{P}}
\newcommand{\D}{\mathcal{D}}

\renewcommand{\d}{\bm{d}}
\newcommand{\dsub}[2]{d_{#1#2}}
\newcommand{\pall}{\d_{*}}

\newcommand{\epsone}{\eps_1}
\newcommand{\V}{V}
\newcommand{\Vtopower}[1]{V^{#1}}
\renewcommand{\b}{b}

\newcommand{\J}{J}
\newcommand{\K}{K}
\newcommand{\intK}{\sinterval{2K}}
\newcommand{\intJ}{\sinterval{J}}

\newcommand{\DD}{\mathbf{D}}
\newcommand{\pred}{\mathbf{Prd}_{\s,\lit,r}}
\newcommand{\unpred}{\mathbf{Unp}_{\s,\lit,r}}
\newcommand{\red}[1]{\mathbf{Red}_{#1}}

\newcommand{\upletters}[1]{{\normalfont{\texttt{#1}}}}
\newcommand{\makesubscript}[1]{\textsubscript{$#1$}}
\newcommand{\progression}{A_{\lit,\d}}

\newcommand{\word}{\mathcal{W}}
\newcommand{\letters}[1]{#1[*]}
\newcommand{\pos}[2]{\mathrm{Pos}(#2; #1)}
\newcommand{\posbis}[2]{\mathrm{Pos}(#2; #1)}

\newcommand{\lit}{\mathcal{L}}
\newcommand{\unlit}{{\kern0.05em}\mathcal{U}}
\newcommand{\s}{{\kern0.05em}\mathcal{S}}

\renewcommand{\L}{L}
\newcommand{\Y}{Y_{\!\L}}
\newcommand{\prohibprog}{\mathcal{Y}}
\newcommand{\shiftedprog}{\prohibprog-\bm{\b}}
\newcommand{\interprog}{(\shiftedprog)^{\cap}}

\newcommand{\oE}{\vec{E}}
\newcommand{\NB}[1]{\mathrm{nbw}(#1)}
\newcommand{\modulus}{q}
\newcommand{\indexset}{\mathcal{I}}
\newcommand{\otherindexset}{\mathcal{I'}}

\newcommand{\XX}{\mathbf{T}}
\newcommand{\Xd}{X_{\d}}
\newcommand{\calX}{\mathcal{X}}
\newcommand{\template}{\theta}
\newcommand{\T}{T}
\newcommand{\Mvar}{M}

\makeatletter

\renewcommand{\tocsection}[3]{%
  \indentlabel{\@ifnotempty{#2}{\bfseries\makebox[1.75em][l]{#2}}}\bfseries#3}

\renewcommand{\tocsubsection}[3]{%
  \indentlabel{\@ifnotempty{#2}{\makebox[3em][l]{#1 #2}}}#3}

\newcommand\@dotsep{4.5}
\def\@tocline#1#2#3#4#5#6#7{\relax
  \ifnum #1>\c@tocdepth \else
    \par \addpenalty\@secpenalty\addvspace{#2}%
    \begingroup \hyphenpenalty\@M
    \@ifempty{#4}{%
      \@tempdima\csname r@tocindent\number#1\endcsname\relax
    }{%
      \@tempdima#4\relax
    }%
    \parindent\z@ \leftskip#3\relax \advance\leftskip\@tempdima\relax
    \rightskip\@pnumwidth plus1em \parfillskip-\@pnumwidth
    #5\leavevmode\hskip-\@tempdima{#6}\nobreak
    \leaders\hbox{$\m@th\mkern \@dotsep mu\hbox{.}\mkern \@dotsep mu$}\hfill
    \nobreak
    \hbox to\@pnumwidth{\@tocpagenum{#7}}\par
    \nobreak
    \endgroup
  \fi}

\def\l@subsection{\@tocline{2}{0pt}{2.5pc}{5pc}{}}
\makeatother

\newcommand{\tabularsize}{0.92}

\begin{document}

\title[Improved bounds for the two-point logarithmic {C}howla conjecture]{Improved bounds for \\the two-point logarithmic Chowla conjecture}

\author{C\'edric Pilatte}
\address{Mathematical Institute, University of Oxford.}
\email{cedric.pilatte@maths.ox.ac.uk}

\begin{abstract}
  Let $\lambda$ be the Liouville function, defined as $\lambda(n) := (-1)^{\Omega(n)}$ where $\Omega(n)$ is the number of prime factors of $n$ with multiplicity. In 2021, Helfgott and Radziwi{\l}{\l} proved that
  \begin{equation*}
    \sum_{n\leq x} \frac{1}{n}\lambda(n) \lambda(n+1) \ll \frac{\log x}{(\log \log x)^{1/2}},
  \end{equation*}
  improving earlier results by Tao and Teräväinen. We prove that
  \begin{equation*}
    \sum_{n\leq x} \frac 1n \lambda(n) \lambda(n+1) \ll (\log x)^{1-c}
  \end{equation*}
  for some absolute constant $c>0$. This appears to be best possible with current methods.
\end{abstract}

\maketitle

{
  \vspace{1.6cm} 
  \hypersetup{linkcolor=black}
  \tableofcontents
}

\parindent 0pt
\parskip 5pt

\section{Introduction}
\subsection{Background}
Let $\lambda : \mathbb{N} \to \{-1, +1\}$ be the Liouville function, defined by $\lambda(n) := (-1)^{\Omega(n)}$ where $\Omega(n)$ is the number of prime factors of $n$, counted with multiplicity. Its statistical properties are closely connected with the distribution of primes. Indeed, the bounds $\frac{1}{x}\sum_{n\leq x} \lambda(n) = o_{x\to \infty}(1)$ and
\begin{equation*}
    \sum_{n\leq x} \lambda(n) \ll_{\eps}x^{1/2+\eps}
\end{equation*}
are equivalent to the Prime Number Theorem and the Riemann Hypothesis respectively, by elementary arguments. These two examples are consistent with the Liouville pseudorandomness principle, a heuristic which suggests that $\lambda$ should statistically behave like a sequence of independent random variables taking the values $-1$ and $+1$ with probability $1/2$.

For higher-degree correlations, a well-known conjecture of Chowla~\cite{chowla} asserts that, for any $k\geq 1$ and distinct integers $h_1, \ldots, h_{k}$, one has
\begin{equation}
    \label{conj:chowla}
    \frac{1}{x}\sum_{n\leq x} \lambda(n+h_1) \lambda(n+h_2) \cdots \lambda(n+h_k) = o_{x\to \infty}(1).
\end{equation}
This can be regarded as a multiplicative analogue of the Hardy-Littlewood prime $k$-tuple conjecture, which predicts an asymptotic formula for correlations of the von Mangoldt function $\Lambda$. Chowla's conjecture is subject to the \emph{parity problem}, a major obstacle in analytic number theory (see \cite[Section~16.4]{opera} for more details). It is open for all $k\geq 2$.

Yet, in recent years, remarkable progress has been made on weaker variants of Chowla's conjecture.

In 2015, Matomäki, Radziwi\l\l~and Tao proved that \cref{conj:chowla} holds \emph{on average} over $h_1, \ldots, h_k$, for every fixed $k\geq 2$ \cite{MRT}. A crucial ingredient in their proof was the groundbreaking work of Matomäki and Radziwi\l\l~\cite{MR} on sums of multiplicative functions over short intervals.

One year later, Tao proved a \emph{logarithmic} version of Chowla's conjecture for $k=2$ \cite{T}. This means that the regular average $\tfrac1x \sum_{n\leq x} f(n)$ is replaced by the logarithmic average $\frac{1}{\log x} \sum_{n\leq x}\frac{1}{n} f(n)$. Fixing $(h_0, h_1) = (0, 1)$ for simplicity, Tao's result thus reads
\begin{equation}
    \label{eq:logchowla}
    \frac{1}{\log x}\sum_{n\leq x} \frac{1}{n}\lambda(n) \lambda(n+1) = o_{x\to \infty}(1).
\end{equation}
Tao's proof \cite{T}, which used a novel \emph{entropy decrement argument}, was a key step in his resolution of the Erdős discrepancy problem \cite{erdos}. From his paper~\cite{T}, it is possible (see \cite{HU}) to extract the explicit bound
\begin{equation}
    \label{eq:quantit1}
    \sum_{n\leq x} \frac{1}{n}\lambda(n) \lambda(n+1) \ll \frac{\log x}{(\log \log \log \log x)^{1/5}}.
\end{equation}

The logarithmic version of Chowla's conjecture \cref{conj:chowla} was later proved for all \emph{odd} $k\geq 3$ by Tao and Teräväinen~\cite{TTduke}. The two authors gave a different proof of that result in \cite{taoODD}. For \emph{even} $k\geq 4$, the logarithmically averaged Chowla conjecture is still open. The methods of their paper~\cite{taoODD} can be used to obtain the following quantitative refinement of \cref{eq:quantit1}: for some small absolute constant $c>0$,
\begin{equation}
    \label{eq:quantit2}
    \sum_{n\leq x} \frac{1}{n}\lambda(n) \lambda(n+1) \ll \frac{\log x}{(\log \log \log x)^{c}}.
\end{equation}

In 2021, Helfgott and Radziwi\l\l~\cite{HR} proved the substantial quantitative improvement
\begin{equation}
    \label{eq:HRchowla}
    \sum_{n\leq x} \frac{1}{n}\lambda(n) \lambda(n+1) \ll \frac{\log x}{(\log \log x)^{1/2}}.
\end{equation}
They used a very different combinatorial approach, studying the eigenvalues of a certain weighted graph defined in terms of divisibility by small primes. A high-level exposition of their proof is given by Helfgott~\cite{explanation}.

In this paper, we improve the approach of Helfgott and Radziwi\l\l~\cite{HR} to prove the following bound.

\begin{theorem}[Logarithmic two-point Chowla correlations]
    \label{thm:chowla}
    For some absolute constant $c>0$,
    \begin{equation*}
        \sum_{n\leq x} \frac 1n \lambda(n) \lambda(n+1) \ll (\log x)^{1-c}.
    \end{equation*}
\end{theorem}

It appears that saving a fixed power of the logarithm is the best that is achievable with current techniques. Ultimately, our proof relies on the work of Matomäki and Radziwi\l\l~\cite{MR} on multiplicative functions in short intervals, where the current state of the art only allows one to save a small power of $\log x$. The exploitation of multiplicativity using an idea of Tao~\cite{T} also separately appears to limit our saving to a small power of $\log x$, because a typical integer has $O(\log x)$ divisors.

The methods of this paper should generalise to a wider class of multiplicative functions through appropriate modifications. The complete multiplicativity of $\lambda$ is only used in \cref{prop:differencesum} and in the proof that \cref{thm:main} implies \cref{thm:chowla}. The only other property of $\lambda$ we use is that it is $1$-bounded (for \cref{sec:thesecondsection,sec:alterations}), but a weaker $\ell^p$ bound would suffice.

Our proof also yields an improved bound for the \emph{unweighted} two-point correlations (i.e.~without logarithmic averaging) at \emph{almost all scales}; see \cref{rem:almostallscales}.

\subsection{Proof outline}
In this section, we give a very short description of the overall strategy. Fuller explanations are given along the way, at various points in the paper.

Using the multiplicativity of $\lambda$, Tao~\cite{T} showed that the problem of bounding $\sum_{n\leq x} \frac 1n \lambda(n) \lambda(n+1)$ reduces to bounding
\begin{equation*}
    \E_{p\in P} \sum_{n\leq x} \lambda(n) \lambda(n+p) \left(\ind{p\mid n}-\frac{1}{p}\right),
\end{equation*}
where $P \subset [1, \exp(\sqrt{\log x})]$ is a set of primes.

Helfgott and Radziwi\l\l~\cite{HR} interpreted the above expression as the matrix product $\bm{\lambda}^{\!\top}\!A\bm{\lambda}$ where $\bm{\lambda} := (\lambda(1), \ldots, \lambda(x))^{\top}$ and $A$ is the matrix with entries
\begin{equation*}
    A_{mn} := \begin{cases} \ind{p\mid n}-\frac{1}{p} &\text{if }\abs{m-n}=p\in P,\\ 0&\text{otherwise.}\end{cases}
\end{equation*}
Hence, it is sufficient to bound the eigenvalues of the matrix $A$, or the eigenvalues of its restriction $A|_{X}$ to some very dense subset $X \subset \{1, \ldots, x\}$. Using a high trace method, Helfgott and Radziwi\l\l~\cite{HR} managed to obtain the bound $\big(\hspace{-1pt}\sum_{p\in P}1/p \big)^{1/2+o(1)}$ for the largest eigenvalue of such a matrix, which is essentially the best possible. Since $\sum_{p\in P}1/p \ll \log \log x$, this approach cannot yield a saving better than a power of $\log \log x$ over the trivial bound for two-point Chowla correlations.

In our new approach, we replace the average over primes $p\in P$ with an average over integers $d = p_1\cdots p_k$ that are products of $k$ primes, where $k\asymp \log \log x$. By Tao's argument, we need to bound
\begin{equation*}
    \E_{d\in D_k} \sum_{n\leq x} \lambda(n) \lambda(n+d) \prod_{p\mid d} \left(\ind{p\mid n} - \frac{1}{p}\right)
\end{equation*}
where $D_k$ is a set of integers with $k$ prime factors. Following the strategy of Helfgott and Radziwi\l\l~\cite{HR}, it is sufficient to bound the eigenvalues of the matrix $\widetilde{A}|_X$ where $\widetilde{A}$ is the matrix defined by
\begin{equation*}
    \widetilde{A}_{mn} := \begin{cases} \prod_{p\mid d}\left(\ind{p\mid n}-\frac{1}{p}\right) &\text{if }\abs{m-n}=d\in D_k,\\ 0&\text{otherwise;}\end{cases}
\end{equation*}
and $X$ is a large subset of $\{1, \ldots, x\}$. We prove that all eigenvalues of $\widetilde{A}|_X$ are bounded by
\begin{equation*}
    \bigg(\sum_{d\in D_k} \frac{1}{d} \bigg)^{1/2+o(1)}
\end{equation*}
in absolute value. Since $k\gg \log\log x$, we can choose $D_k$ so that $\sum_{d\in D_k}1/d \gg (\log x)^{c}$, which produces the exponential improvement in \cref{thm:chowla}.

Unfortunately, working with products of multiple primes rather than single primes introduces new difficulties throughout the argument. It is handling all of these new difficulties which is the key new contribution of our work. We are forced to rework and generalise all the arguments of \cite{HR} with the result that our paper is essentially self-contained.

One particular new difficulty is in \cref{sec:unpredictable} where we wish to bound the number of solutions to systems of divisibility constraints. In the prior work this was a linear system, and so could be bounded by a simple lattice point argument. In our situation this now becomes a polynomial system, and to handle it we require a much more involved argument based on the structure of what we call \emph{unpredictable words}. Another novel feature of our work is the use of \emph{non-backtracking operators}; these considerably simplify certain technical aspects of the proof.

\subsection{Structure of the paper}
We now give a broad overview of the structure of the paper. The reader may wish to refer to \cref{fig:dependency}, which depicts the main propositions of the paper along with their logical dependencies. The paper is designed to be as self-contained as possible. In particular, no prior familiarity with \cite{HR} is needed.

\begin{figure}[H]
    \centering
    \scalebox{1}{
        \begin{tikzpicture}[scale=1]
            \node [draw] (0) at (0, 4.5) {\cref{thm:chowla}}; 
            \node [draw] (1) at (0, 3.25) {\cref{thm:main}}; 
            \node [draw] (3) at (-4, 2) {\cref{prop:differencesum}}; 
            \node [draw] (4) at (0, 2) {\cref{prop:OGmatrixeigenbound}}; 
            \node [draw] (5) at (0, 0.75) {\cref{prop:nonbacktrackingeigenbound}}; 
            \node [draw] (6) at (0, -0.5) {\cref{prop:hightracebound}}; 
            \node [draw] (7) at (-4, 0.75) {\cite[Theorem~1.3]{MRT}};
            \node [draw] (8) at (0, -1.75) {\cref{prop:predictablevector}}; 
            \node [draw] (9) at (4, -1.75) {\cref{prop:unpredictable}};  
            \node [draw] (10) at (4, 0.75) {\cref{prop:iharabass}}; 
            \node [draw] (11) at (-4, -1.75) {\cref{prop:cancellationoverY}}; 

            \draw [-{Latex[length=2.5mm]}] (1) to (0);
            \draw [-{Latex[length=2.5mm]}] (3) to (1);
            \draw [-{Latex[length=2.5mm]}] (4) to (1);
            \draw [-{Latex[length=2.5mm]}] (7) to (3);
            \draw [-{Latex[length=2.5mm]}] (5) to (4);
            \draw [-{Latex[length=2.5mm]}] (10) to (4);
            \draw [-{Latex[length=2.5mm]}] (6) to (5);
            \draw [-{Latex[length=2.5mm]}] (8) to (6);
            \draw [-{Latex[length=2.5mm]}] (9) to (6);
            \draw [-{Latex[length=2.5mm]}] (11) to (6);
        \end{tikzpicture}}
    \caption{Dependency graph for the proof of \cref{thm:chowla} (main propositions only).}
    \label{fig:dependency}
\end{figure}
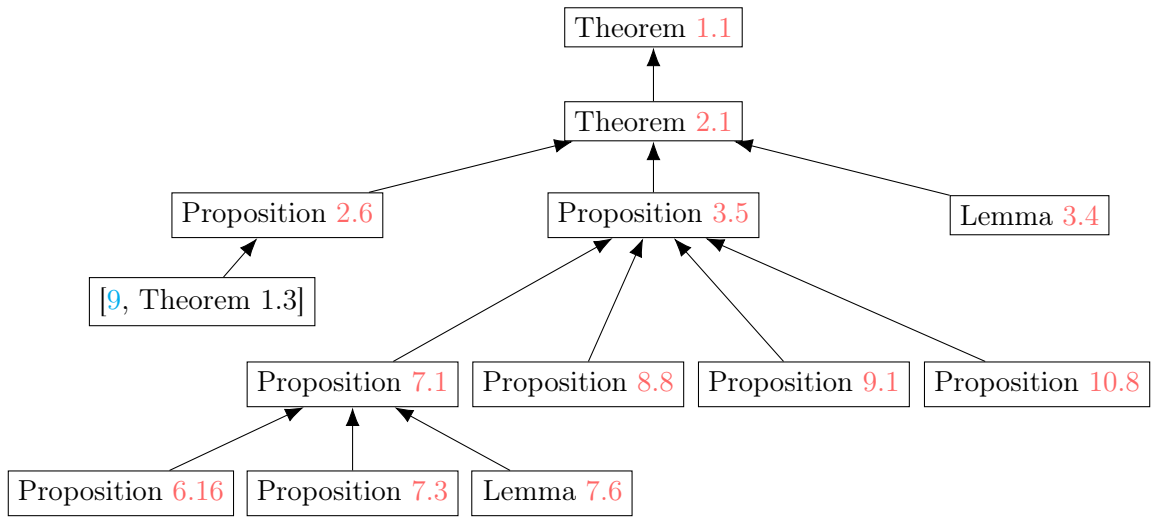

We begin by reducing \cref{thm:chowla}, our bound for two-point logarithmic Chowla correlations, to \cref{thm:main}, an estimate that involves an additional average over products of many primes. This reduction, carried out in \cref{sec:thesecondsection}, follows the clever manipulations of Tao~\cite{T}.

The proof of \cref{thm:main} begins with \cref{prop:differencesum}, which replaces the double sum in \cref{thm:main} with a more convenient `centred' version. The proof of \cref{prop:differencesum} relies on an exponential sum estimate of Matomäki, Radziwi\l\l~and Tao~\cite{MRT}.

Following the approach of Helfgott and Radziwi\l\l~\cite{HR}, our task reduces to proving \cref{prop:OGmatrixeigenbound}, a bound on the eigenvalues of a certain symmetric matrix (denoted $\widetilde{A}|_X$ in the preceding section).

The high trace method cannot be applied directly to this matrix, due to the excessive contribution of `backtracking walks' that repeatedly visit integers with a very large number of prime factors. Helfgott and Radziwi\l\l~\cite{HR} overcame this obstacle by introducing smooth weights to discard such integers. Unfortunately, this method creates considerable technical difficulties, which would be even more severe in our setting.

Instead, we replace our matrix with a related \emph{non-backtracking matrix}, which is no longer symmetric. We relate its spectral properties to those of the original matrix using a fundamental identity for non-backtracking operators, known as the \emph{Ihara-Bass formula} (\cref{prop:iharabass}).

As a result, it suffices to prove \cref{prop:nonbacktrackingeigenbound}, a tail estimate on the eigenvalues of the non-backtracking matrix. In turn, this spectral tail estimate follows from \cref{prop:hightracebound}, which bounds the trace of a high power of the non-backtracking matrix.

The proof of this high trace bound is technically involved, occupying \cref{sec:thesetY,sec:tracecomputations,sec:single,sec:predictable,sec:systems,sec:unpredictable,sec:combinatorial}. To guide the reader, we provide a detailed overview of this part of the argument in \cref{sec:heuristics}.

\addtocontents{toc}{\protect\setcounter{tocdepth}{-1}}
\section*{Acknowledgements}
\addcontentsline{toc}{section}{Acknowledgements}
\addtocontents{toc}{\protect\setcounter{tocdepth}{2}}

The author is supported by the Oxford Mathematical Institute and a Saven European Scholarship. I wish to express my deep gratitude to James Maynard, whose guidance, encouragement, and continuous discussions have played a vital role in the successful completion of this work. Furthermore, I am grateful to Ben Green, Jad Hamdan, Harald Helfgott, Mayank Pandey, Joni Teräväinen, and the PhD students and postdocs of the Oxford analytic number theory group for enriching discussions. I would like to thank Jon Keating for encouraging me to persevere with the non-backtracking approach, and for providing useful references on the Ihara zeta function.

\subsection{Symbols and notation}
For ease of reading, we provide a table showing the main parameters, their size and a reference to where they are introduced.
\begin{center}
    \scalebox{\tabularsize}{
        \begin{tabular}{c|c|c}
            Parameter & Size properties                              & First appearance     \\
            \hline
            $\epsone$ & $\epsone>0$ sufficiently small               & \cref{not:mainparam} \\
            $H$       & $H = H(\epsone)$ sufficiently large          & \cref{not:mainparam} \\
            $H_0$     & $H_0 = \exp\!\big((\log H)^{1-\epsone}\big)$ & \cref{not:mainparam} \\
            $\J$      & $1\leq \J \leq \epsone^2 \log \log H $       & \cref{def:setprimes} \\
            $\V_j$    & $\V_j = \sum_{p\in \primes_j} 1/p\leq \V$    & \cref{def:setprimes} \\
            $\V$      & $\epsone^{-1}\ll \V\ll \epsone \log \log H$  & \cref{def:setprimes} \\
            $N$       & $N \geq \exp \!\big((\log H)^{2.001}\big)$   & \cref{thm:main}      \\
            $\K$      & $\K = \lfloor \log H \rfloor$                & \cref{not:K}         \\
            $\L$      & $\L=\K^{1-10\epsone}$                        & \cref{not:L}
        \end{tabular}}
\end{center}

The following table contains most of the other symbols used repeatedly in the paper.

\begin{center}
    \scalebox{\tabularsize}{
        \begin{tabular}{c|c|c}
            Notation                         & Description                                                               & Reference                              \\
            \hline
            $\primes_j$                      & disjoint sets of primes $\subset (H_0, H)$                                & \cref{def:setprimes}                   \\
            $\primes$                        & $\primes := \bigcup_j \primes_j$                                          & \cref{not:bigP}                        \\
            $\D$                             & set of all products $p_1p_2\cdots p_{\J}$ with each $p_j\in \primes_j$    & \cref{not:bigP}                        \\
            $\N$                             & $\N := \mathbb{N} \cap (N, 2N]$                                           & \cref{not:bigP}                        \\
            $w_Y$                            & explicit function $\N\times \N \to \mathbb{R}$ depending on $Y\subset \Z$ & \cref{def:AYmatrix}                    \\
            $G_Y$                            & graph on $\N$ with edge weights given by $w_Y$                            & \cref{def:specificYgraph}              \\
            $A_Y$                            & weighted adjacency matrix of $G_Y$                                        & \cref{def:AYmatrix,def:specificYgraph} \\
            $M_Y$                            & non-backtracking matrix associated to $G_Y$                               & \cref{def:specificYgraph}              \\
            $\N^{\mathrm{typ}}$              & set of $n\in \N$ with a moderate number of prime factors in $\primes$     & \cref{def:highlydivisiblenumbers}      \\
            $b_i$                            & partial sums $b_i := \sum_{1\leq k\leq i }d_{k}$                          & \cref{lem:tracesimplebound}            \\
            $\prohibprog$                    & set of all prohibited progressions                                        & \cref{def:prohibitedprogression}       \\
            $\Y$                             & integers not contained in a prohibited progression                        & \cref{def:prohibitedprogression}       \\
            $\dsub{i}{j}$                    & unique prime in $\primes_j$ dividing $d_i$                                & \cref{def:dij}                         \\
            $\pall$                          & product of all $\dsub{i}{j}$                                              & \cref{def:dij}                         \\
            $\d_{I}$                         & product of all $\dsub{i}{j}$ with $(i,j)\in I$                            & \cref{def:dij}                         \\
            $\s$                             & set of single indices                                                     & \cref{def:single}                      \\
            $\lit,\unlit$                    & sets of lit and unlit indices                                             & \cref{lem:threeindicestypes}           \\
            $\DD_{\s,\lit,r}$                & set of vectors $\d\in (\pm \D)^{2\K}$ of interest                         & \cref{def:D}                           \\
            $E_{\s,\lit,\unlit,r}$           & quantity to be bounded                                                    & \cref{def:Esum}                        \\
            $\progression$                   & arithmetic progression determined by the lit indices                      & \cref{lem:existencerank}               \\
            $\d^{(k)}$                       & subsequence of $\d$, for $k=1,2,3$                                        & \cref{def:D}                           \\
            $v_{j,\d}^{(k)}, w_{j,\d}^{(k)}$ & words associated to $\d$ with letters in $\primes_j$                      & \cref{def:predunpred}                  \\
            $\pred$                          & vectors $\d$ such that all $w_{j,\d}^{(k)}$ are predictable               & \cref{def:predunpred}                  \\
            $\unpred$                        & vectors $\d$ such that some $w_{j,\d}^{(k)}$ is unpredictable             & \cref{def:predunpred}                  \\
            $\red{R,\s,\lit}$                & reduced vectors $\d\in (\pm \D)^{R}$ compatible with $\s$ and $\lit$      & \cref{def:reducedset}
        \end{tabular}
    }
\end{center}

We write $f\ll g$ or $f = O(g)$ if $|f| \leq Cg$ for some absolute constant $C>0$. The notation $f\asymp g$ means that $f\ll g$ and $g\ll f$.

If $a,b\in \Z$, we write $\interval{a}{b} := \Z\cap [a,b]$ ($=\emptyset$ if $a>b$), and we call a set of this form a \emph{discrete interval}. Its length, or size, is its cardinality ($b-a+1$ if $a\leq b$). For $n\in \mathbb{N}$, we write $\sinterval{n} := \{1,2,\ldots, n\}$.

If $A\subset \R$, we write $\pm A$ for $\{\sigma a: \sigma \in \{\pm1\},\, a\in A\}$.

In this paper, the term \emph{arithmetic progression} always refers to a `two-sided infinite' arithmetic progression of the form $a+q\Z$ for some $a\in \Z$ and $q\in \mathbb{N}$. The modulus $q$ of an arithmetic progression $R = a+q\Z$ is denoted by $\modulus_R$.

For $n\in \Z$, we write $\omega(n)$ for the number of distinct prime factors of $n$. If $\primes$ is a set of primes, we let $\omega_{\primes}(n)$ be the number of primes in $\primes$ that divide $n$. Euler's totient function and the divisor sum function are denoted by $\varphi$ and $\sigma_1$, respectively. The $p$-adic valuation of $n\in \mathbb{N}$ is denoted by $\nu_p(n)$.

If $A = (a_{m,n})_{m,n\in I}$ is a square matrix with complex entries indexed by a set $I$, and $J\subset I$, we write $A|_J$ for the restriction of $A$ to the rows and columns indexed by $J$. The conjugate transpose of $A$ is denoted by $A^*$. The operator norm of $A$ is $\norm{A}_{\mathrm{op}} := \sup_{\norm{x}_2 = 1} \norm{Ax}_2$. The spectral radius of $A$ is the largest absolute value of an eigenvalue of $A$. The eigenvalues of a matrix are always counted with multiplicity.

\section{Main theorem and consequences}
\label{sec:thesecondsection}

\subsection{Main parameters}

To state our main theorem, we need to introduce some parameters. We will use the following notation throughout the paper.

\begin{notation}[$\epsone$, $H$, $H_0$]
    \label{not:mainparam}
    Let $\epsone>0$ be a sufficiently small constant. Let $H\in \mathbb{N}$ be sufficiently large in terms of $\epsone$. Let $H_0 := \exp\!\big((\log H)^{1-\epsone}\big)$.
\end{notation}

\begin{definition}[$\J$, $\primes_i$, $\V_i$, $\V$]
    \label{def:setprimes}
    Let $\J$ be an integer satisfying $1\leq \J \leq \epsone^2 \log \log H$.

    For $1\leq j\leq \J$, let $\primes_j$ be the set of all primes $p$ such that
    \begin{equation*}
        \rho^{2j-2} < \frac{\log p}{\log H_0} < \rho^{2j-1},
    \end{equation*}
    where $\rho := \exp\!\big(\epsone (\log \log H)/(2\J)\big)$.

    For each $1\leq j\leq \J$, we define $\V_j := \sum_{p\in \primes_j} 1/p$, and we let $\V := \max_{1\leq j\leq \J} \V_j$.
\end{definition}

We establish simple estimates concerning the sets $\primes_j$ and the associated parameters $\V_j$, which will be used repeatedly in the sequel.

\begin{lemma}
    \label{lem:primesets}
    The sets $\primes_j$ and parameters $\V_j$ satisfy the following properties:
    \begin{enumerate}[label=(\alph*), ref=\alph*]
        \item \label{item:primesets1} $\primes_1, \ldots, \primes_{\J}$ are disjoint subsets of $(H_0, H)$;
        \item \label{item:primesets2} if $(p_1, \ldots, p_{\J}) \in \primes_1\times \cdots \times \primes_{\J}$, then $p_1p_2\cdots p_{\J}<H$ and $p_1p_2\cdots p_i < p_{i+1}^{1/10}$ for all $1\leq i <\J$;
        \item \label{item:primesets3} $\epsone^{-1}\ll \V \asymp \displaystyle\frac{\epsone \log \log H}{\J}\ll \epsone \log \log H$;
        \item \label{item:primesets4} $\V_1\V_2\cdots \V_{\J} \asymp \V^{\J} \ll (\log H)^{c_0(\epsone)}$ where $c_0(\epsone) := \epsone^2 \log (\frac{1}{2\epsone})$. Moreover, in the special case where ${\J = \lfloor \epsone^2 \log \log H\rfloor}$, we have $\V^{\J} \asymp_{\epsone} (\log H)^{c_0(\epsone)}$.
    \end{enumerate}
\end{lemma}
\begin{proof}
    Let $\rho := \exp\!\big(\epsone (\log \log H)/(2\J)\big)$; note that $\rho\geq 20$. Since $\primes_j$ is the set of all primes in the interval
    \begin{equation*}
        \Big(\exp\!\left(\rho^{2j-2} \log H_0\right),\  \exp\!\left(\rho^{2j-1} \log H_0\right) \Big),
    \end{equation*}
    property \cref{item:primesets1} is clear. If $(p_1, \ldots, p_{\J}) \in \primes_1\times \cdots \times \primes_{\J}$, then for all $i\in \intJ$ we have
    \begin{equation*}
        p_1p_2\cdots p_i < \exp\! \bigg(\sum_{1\leq j\leq i} \rho^{2j-1} \log H_0\bigg) \leq \exp \!\Big( 2\rho^{2i-1} \log H_0\Big) \leq \exp\!\bigg( \frac{\rho^{2i}}{10} \log H_0\bigg).
    \end{equation*}
    The right-hand side is $\leq p_{i+1}^{1/10}$ if $i<\J$, and equals $H^{1/10}$ if $i=\J$. This proves \cref{item:primesets2}.

    By Mertens' second estimate, we have
    \begin{equation*}
        \V_j = \frac{\epsone \log \log H}{2\J} + O\!\left(\frac{1}{\log H_0}\right),
    \end{equation*}
    which implies property \cref{item:primesets3} and the bounds
    \begin{equation*}
        \V_1\V_2\cdots \V_{\J} \asymp \V^{\J} \asymp \bigg(\frac{\epsone \log \log H}{2\J}\bigg)^{\J}.
    \end{equation*}
    Finally, since $\J \leq \epsone^2 \log \log H$ and the function $x\mapsto (A/x)^x$ is increasing on $[0, A/e]$, property~\cref{item:primesets4} follows.
\end{proof}

\subsection{Statement of the main theorem}
Our bound for the two-point logarithmic Chowla correlations is a consequence of the following key estimate.

\begin{theorem}
    \label{thm:main}
    Let $N$ be an integer such that $\log N \geq (\log H)^{2.001}$. Then
    \begin{equation}
        \label{eq:summainthm}
        \bigg\lvert\sum_{(p_1, \ldots, p_{\J})\in \primes_1\times \cdots \times \primes_{\J}}\,  \sum_{\substack{n\in (N, 2N] \\ p_1\cdots p_{\J} \mid n}}\lambda(n) \lambda(n+p_1\cdots p_{\J}) \bigg\rvert \leq e^{O(\J)} (\V_1\cdots \V_{\J})^{1/2} N.
    \end{equation}
\end{theorem}

\begin{remark}
    \label{rem:L}
    \Cref{thm:main} should be compared with the trivial bound $\ll \V_1\cdots \V_{\J} N$.

    For $\J = 1$, we recover the result of Helfgott and Radziwi{\l}{\l}~\cite{HR}. Crucially, however, our approach allows $\J$ to be as large as a small multiple of $\log \log H$. By \cref{lem:primesets}, this flexibility enables us to save a power of $\log H$.
\end{remark}

\begin{remark}
    \label{rem:withabsval}
    In fact, our methods yield a slightly stronger estimate with absolute values:
    \begin{equation}
        \label{eq:summainthmABSVAL}
        \sum_{(p_1, \ldots, p_{\J})\in \primes_1\times \cdots \times \primes_{\J}} \Bigg\lvert \sum_{\substack{n\in (N, 2N] \\ p_1\cdots p_{\J} \mid n}}\lambda(n) \lambda(n+p_1\cdots p_{\J})  \Bigg\rvert \leq e^{O(\J)} (\V_1\cdots \V_{\J})^{1/2} N.
    \end{equation}
    This follows by repeating the proof with arbitrary signs $c_{p_1, \ldots, p_{\J}} \in \{\pm 1\}$, but for clarity we focus on the version in \cref{thm:main}, which suffices for our main application.
\end{remark}

\subsection{Proof of the two-point logarithmic Chowla bound}

\begin{notation}[$\D$, $\N$, $\primes$]
    \label{not:bigP}
    To shorten the expressions, we define $\D\subset (H_0, H)$ to be the set of all products $p_1p_2\cdots p_{\J}$ with $p_j\in \primes_j$ for all $j\in \intJ$. We also write $\N := \mathbb{N} \cap (N, 2N]$ and $\primes := \bigsqcup_{j=1}^{\J} \primes_j$.
\end{notation}

With this notation, we can rewrite the double sum in \cref{eq:summainthm} as
\begin{equation*}
    \sum_{d\in \D} \sum_{\substack{n\in \N\\ d\mid n}} \lambda(n) \lambda(n+d).
\end{equation*}
We now show how a bound on this expression implies a bound on the two-point logarithmic Chowla conjecture. This step is due to Tao \cite{T}, and crucially relies on the multiplicativity of $\lambda$. Together with the proof of \cref{prop:differencesum}, this is the only place where the multiplicativity of $\lambda$ is used -- the rest of the paper will only use that $\lambda$ is $1$-bounded.

\begin{proof}[Proof of \cref{thm:chowla}, assuming \cref{thm:main}]
    Let $\epsone$, $H$ and $H_0$ be as in \cref{not:mainparam}. In particular, $\epsone > 0$ is a sufficiently small absolute constant, and $H$ is chosen sufficiently large in terms of~$\epsone$. Choose $\J:= \lfloor \epsone^2 \log \log H \rfloor$. Let $(\primes_j)$, $(\V_j)$ and $\V$ be as in \cref{def:setprimes}. Thus, by \cref{lem:primesets},
    \begin{equation}
        \label{eq:VJbound}
        V^{\J} \asymp_{\epsone} (\log H)^{\epsone^2 \log (\frac{1}{2\epsone})}.
    \end{equation}
    Finally, let $x := \exp\!\big((\log H)^{6}\big)$.

    For any $d\in \D$, since $\lambda$ is completely multiplicative and $\lambda^2=1$, we may rewrite the logarithmic Chowla average as
    \begin{equation*}
        \sum_{n\leq x} \frac 1n \lambda(n) \lambda(n+1) =  \sum_{n\leq x} \frac 1n \lambda(dn) \lambda(dn+d) = d \sum_{\substack{m\leq dx\\ d\mid m}} \frac 1m \lambda(m) \lambda(m+d).
    \end{equation*}
    Dividing by $d$ and summing over all $d\in \D$ yields
    \begin{equation}
        \label{eq:manip1}
        \V_1\cdots \V_{\J} \sum_{n\leq x} \frac 1n \lambda(n) \lambda(n+1) = \sum_{d\in \D}\, \sum_{\substack{m\leq dx\\ d\mid m}} \frac 1m \lambda(m) \lambda(m+d).
    \end{equation}

    We now apply \cref{thm:main} in a dyadic fashion to bound the right-hand side of \cref{eq:manip1}.

    For every $N$ such that $\log N \geq (\log H)^3 = \sqrt{\log x}$, by \cref{thm:main}, we know that
    \begin{equation}
        \label{eq:largerangemainthm}
        \bigg\lvert\sum_{d\in \D} \sum_{\substack{n\in \N\\ d\mid n}} \lambda(n) \lambda(n+d)\bigg\rvert \leq e^{O(\J)} \V^{\J/2} N.
    \end{equation}
    Moreover, when ${\log N \leq \sqrt{\log x}}$, trivially bounding $\abs{\lambda} \leq 1$ we have
    \begin{equation}
        \label{eq:smallrangemainthm}
        \sum_{d\in \D} \sum_{\substack{n\in \N\\ d\mid n}} \lambda(n) \lambda(n+d) \ll \V^{\J} N.
    \end{equation}
    By a suitable dyadic decomposition, \cref{eq:largerangemainthm,eq:smallrangemainthm} together give, for all $M\geq 1$, the bound
    \begin{equation}
        \label{eq:allrangesmainthm}
        \bigg\lvert\sum_{d\in \D} \sum_{\substack{m\leq M \\ d\mid m}} \lambda(m) \lambda(m+d)\bigg\rvert \ll \V^{\J} \min\!\big(M, e^{\sqrt{\log x}}\big) + e^{O(\J)}\V^{\J/2} M.
    \end{equation}
    By partial summation (and using \cref{eq:VJbound}), the estimate \cref{eq:allrangesmainthm} implies that
    \begin{equation}
        \label{eq:mainthmpartialsumm}
        \bigg\lvert\sum_{d\in \D}  \sum_{\substack{m\leq x \\ d \mid m}}\frac{1}{m} \lambda(m) \lambda(m+d)\bigg\rvert \leq e^{O(\J)}\V^{\J/2} \log x.
    \end{equation}
    This is almost the expression in \cref{eq:manip1}, up to an error
    \begin{equation}
        \label{eq:manip2}
        \Bigg|\sum_{d\in \D}\, \sum_{\substack{x<m\leq dx \\ d\mid m}} \frac 1m \lambda(m) \lambda(m+d)\Bigg| \leq \sum_{d\in \D}\, \sum_{\substack{x<m\leq dx\\ d\mid m}} \frac 1m
        \ll \sum_{d\in \D} \frac{\log(d)}{d} \ll \V^{\J} \log H.
    \end{equation}
    Hence, by \cref{eq:mainthmpartialsumm,eq:manip1,eq:manip2}, we conclude that
    \begin{equation*}
        \sum_{n\leq x} \frac 1n \lambda(n) \lambda(n+1) \ll \frac{1}{\V^{\J}} \left(e^{O(\J)}\V^{\J/2} \log x +\V^{\J} \log H \right).
    \end{equation*}
    By \cref{eq:VJbound}, this is $\ll_{\epsone} (\log x)^{1-c}$ for some constant $c>0$ depending only on $\epsone$, as desired.
\end{proof}

\begin{remark}[Two-point Chowla at almost all scales]
    \label{rem:almostallscales}
    Our main result also implies an improved quantitative version of Chowla's conjecture for two-point correlations at \emph{almost all scales}. Namely, for all $2<w\leq X$, we have
    \begin{equation}
        \label{eq:almostallscales}
        \frac{1}{\log w} \int_{X/w}^X \abs{\frac{1}{x} \sum_{n\leq x} \lambda(n) \lambda(n+1)} \frac{dx}{x} \ll \frac{1}{(\log w)^c},
    \end{equation}
    where $c>0$ is an absolute constant. In particular, we get
    \begin{equation*}
        \frac{1}{x} \sum_{n\leq x} \lambda(n) \lambda(n+1) \ll \frac{1}{(\log X)^{c/2}}
    \end{equation*}
    for all $x\in [1, X]$ outside of a set $E_X$ of logarithmic density $O((\log X)^{-c/2})$.\footnote{This means that $\displaystyle{\frac{1}{\log X}\int_1^X \ind{E_X}(x) \frac{dx}{x} \ll (\log X)^{-c/2}}$.}

    This almost all scales result \cref{eq:almostallscales} follows from \cref{eq:summainthmABSVAL} by a straightforward adaptation of the proof in \cite[Section~8]{HR} (which in turn was inspired by \cite{TT}) to our setting.
\end{remark}

\section{Reformulations of the main theorem}
\label{sec:alterations}

\subsection{Centred sum}

The bound in \cref{thm:main} turns out to be easier to prove if the divisibility term $\ind{p_1\mid n}\cdots \ind{p_{\J}\mid n}$ is replaced by the `centred' product $\big(\ind{p_1\mid n} - \frac{1}{p_1}\big)\cdots \big(\ind{p_{\J}\mid n} - \frac{1}{p_{\J}}\big)$. The following proposition shows that working with either of these two expressions is equivalent up to a small error.

\begin{proposition}
    \label{prop:differencesum}
    Let $N$ be an integer such that $\log N \geq (\log H)^{2}$. Let
    \begin{equation*}
        S_1 := \sum_{n\in \N} \sum_{\substack{d\in \D\\ d\mid n}} \lambda(n) \lambda(n+d)\quad \text{and}\quad S_2 := \sum_{n\in \N} \sum_{d\in \D} \lambda(n) \lambda(n+d) \prod_{p\mid d} \bigg(\ind{p\mid n} - \frac{1}{p} \bigg).
    \end{equation*}
    Then
    \begin{equation*}
        \left| S_1 - S_2 \right| \ll \frac{N}{(\log H)^{1/2500}}.
    \end{equation*}
\end{proposition}

\Cref{prop:differencesum} is proved in \cref{appendix:probmodel}, using the circle method and crucially relying on an estimate of Matomäki, Radziwiłł and Tao~\cite{MRT}. We will now focus on the centred expression $S_2$.

\subsection{Reduction to an eigenvalue bound} To prove our main theorem, we analyse the spectral properties of a matrix whose entries correspond to the coefficients $\prod_{p\mid d} \big(\ind{p\mid n} - \frac{1}{p} \big)$ appearing in $S_2$. This key step effectively suppresses the role of the Liouville function in the problem.

\begin{definition}[$w_Y$, $A_Y$]
    \label{def:AYmatrix}
    For any set $Y\subset \Z$, let $w_Y : \N \times \N \to \R$ be the function defined by
    \begin{equation*}
        w_Y(m,n) :=
        \begin{cases}
            \prod_{p \mid d}\left(\ind{p\mid n}-\frac{1}{p}\right) & \text{if }|m-n| = d\in \D\text{ and }m,n\in Y, \\
            \, {0}                                                 & \text{otherwise.}
        \end{cases}
    \end{equation*}
    Define $A_Y$ to be the real symmetric matrix\footnote{Note that the entries of $A_Y$ are indexed by the elements of $\N$, rather than $\{1, \ldots, N\}$.} $A_Y := (w_Y(m,n))_{m,n\in \N}$.
\end{definition}

In this section, we take $Y = \Z$ (so the condition $m,n \in Y$ can be omitted). The general definition of $A_Y$ will be useful later: by analysing $A_Y$ for a different, more technical choice of $Y$ (introduced in \cref{def:prohibitedprogression}), we will obtain spectral information that ultimately applies to $A_{\Z}$.

As a consequence of this analysis, almost all eigenvalues of $A_{\Z}$ are bounded by $e^{O(\J)} \V^{\J/2}$ in absolute value.\footnote{This follows from \cref{prop:OGmatrixeigenbound} via Cauchy’s interlacing theorem.} In fact, we prove a stronger result—\cref{prop:OGmatrixeigenbound}—which, combined with \cref{prop:differencesum}, yields our main theorem. To state it, we introduce the following notation.

\begin{notation}
    \label{not:restriction}
    Given a matrix $A = (a_{m,n})_{m,n\in \N}$ and a subset $X\subset \N$, we write $A|_{X}$ for the principal submatrix $(a_{m,n})_{m,n\in X}$ obtained by deleting all rows and columns at indices not in $X$.
\end{notation}

\begin{proposition}[Spectral radius of restriction]
    \label{prop:OGmatrixeigenbound}
    Let $N$ be an integer such that $\log N \geq (\log H)^{2.001}$.

    Then, there is a subset $X \subset \N$ with $\abs{\N \setminus X} \ll e^{-\J\V/3} N$ such that all eigenvalues of $(A_{\Z})|_{X}$  are bounded in absolute value by $ e^{O(\J)}\V^{\J/2}$.
\end{proposition}

\begin{remark}
    \label{rem:highlydivisiblenumbers}
    The restriction to a dense subset $X \subset \N$ is essential. Indeed, the full matrix $A_{\Z}$ does have very large eigenvalues, which arise from integers $n \in \N$ having an unusually large number of prime factors from $\primes$. For instance, if $n\in \N$ is such that $\omega_{\primes_j}(n) \geq C\V$ for every $1\leq j\leq \J$, and $\bm{v}$ is the standard basis vector corresponding to $n$, then
    \begin{equation*}
        \norm{A_{\Z} \bm{v}}_2 = \Bigg(\sum_m w_{\Z}(m,n)^2\Bigg)^{1/2} \geq \Bigg( \sum_{\substack{d\in \D\\ d\mid n}} \prod_{p\mid d} \bigg( 1 - \frac{1}{p}\bigg)^2 \Bigg)^{1/2} \gg \Bigg( \prod_{j=1}^{\J} \omega_{\primes_j}(n)\Bigg)^{1/2} \gg C^{\J/2} \V^{\J/2}.
    \end{equation*}
    Consequently, $A_{\Z}$ possesses an eigenvalue of size at least $\gg C^{\J/2} \V^{\J/2}$ in absolute value, exceeding the desired bound if $C$ is too large. This large eigenvalue directly reflects the exceptional factorisation of $n$. The previous computation also shows that the eigenvalue bound in \cref{prop:OGmatrixeigenbound} is optimal up to the $e^{O(\J)}$ factor.
\end{remark}

\begin{proof}[Proof of \cref{thm:main}, assuming \cref{prop:OGmatrixeigenbound}]
    Let $S_1$ and $S_2$ be as in \cref{prop:differencesum}. Our goal is to show that $|S_1| \leq e^{O(\J)} \V^{\J/2} N$. By \cref{prop:differencesum}, it suffices to prove the same bound for $S_2$.

    Let $\bm{\lambda} := (\lambda(n))_{n\in \N}$ be the vector of Liouville function values. Up to a negligible error, we can express $S_2$ in terms of the matrix $A_{\Z}$ as:
    \begin{equation*}
        S_2 = \sum_{\substack{m,n\in \N\\ m<n}} \lambda(m) \lambda(n) w_{\Z}(m,n) + O(H^2) = \frac{1}{2} \bm{\lambda}^{\!\top}\!A_{\Z}\,\bm{\lambda} + O(H^2).
    \end{equation*}

    To exploit \cref{prop:OGmatrixeigenbound}, we restrict $A_{\Z}$ to the subset $X \subset \N$ and write $\bm{\lambda}|_X := (\lambda(n))_{n\in X}$. Since $\abs{\lambda} \leq 1$, we have
    \begin{equation}
        \label{eq:S2intermediate}
        S_2 = \frac{1}{2} (\bm{\lambda}|_{X})^{\!\top}\!(A_{\Z})|_{X} (\bm{\lambda}|_{X}) + O\Bigg(H^2 + \sum_{n\in \N \setminus X} \sum_{d\in \D} \prod_{p\mid d} \bigg(\ind{p\mid n} + \frac{1}{p} \bigg) \Bigg).
    \end{equation}

    By \cref{prop:OGmatrixeigenbound}, the operator norm of $(A_{\Z})|_{X}$ is bounded by $e^{O(\J)}\V^{\J/2}$, as this is a real symmetric matrix. Therefore, by Cauchy-Schwarz,
    \begin{equation*}
        \big\lvert (\bm{\lambda}|_{X})^{\!\top}\!(A_{\Z})|_{X} (\bm{\lambda}|_{X}) \big\rvert
        \leq \norm{\bm{\lambda}|_{X}}_2 \norm{(A_{\Z})|_{X} (\bm{\lambda}|_{X})}_2
        \leq e^{O(\J)}\V^{\J/2} \norm{\bm{\lambda}|_{X}}_2^2 \leq e^{O(\J)}\V^{\J/2} N,
    \end{equation*}
    which gives the desired bound for the main term in \cref{eq:S2intermediate}.

    It remains to treat the error term. By the AM-GM inequality,
    \begin{equation*}
        \sum_{d\in \D} \prod_{p\mid d} \bigg(\ind{p\mid n} + \frac{1}{p} \bigg) = \prod_{j=1}^{\J} \big(\omega_{\primes_j}(n) + \V_j\big) \leq \left(\frac{\omega_{\primes}(n)}{\J} + \V\right)^{\J}.
    \end{equation*}
    Hence, by Cauchy-Schwarz,
    \begin{equation*}
        \sum_{n\in \N \setminus X} \sum_{d\in \D} \prod_{p\mid d} \bigg(\ind{p\mid n} + \frac{1}{p} \bigg) \leq \abs{\N \setminus X}^{1/2}  \Bigg( \V^{2\J} \sum_{n\in \N} \left(\frac{\omega_{\primes}(n)/\V + \J}{\J}\right)^{2\J} \Bigg)^{1/2}.
    \end{equation*}
    Since $(a/k)^k \leq e^a$ for any $a\geq 0$ and $k\in \mathbb{N}$, we have
    \begin{equation*}
        \sum_{n\in \N} \left(\frac{\omega_{\primes}(n)/\V + \J}{\J}\right)^{2\J} \leq e^{O(\J)}\sum_{n\in \N} e^{2\omega_{\primes}(n)/\V} \leq e^{O(\J)} N
    \end{equation*}
    using \cite[Lemma~(3.10)]{restrictedprimes} for the last inequality. Recalling that $\abs{\N \setminus X} \ll e^{-\J\V/3} N$, we conclude that
    \begin{equation*}
        \sum_{n\in \N \setminus X} \sum_{d\in \D} \prod_{p\mid d} \bigg(\ind{p\mid n} + \frac{1}{p} \bigg) \ll e^{O(\J)}\V^{\J} e^{-\J\V/6}  N.
    \end{equation*}
    By part \cref{item:primesets3} of \cref{lem:primesets}, and since $\epsone$ is assumed to be sufficiently small, this is $\ll e^{-\J\V/8} N$. Thus, the error term in \cref{eq:S2intermediate} is negligible compared to our main term bound.
\end{proof}

\section{Approach via non-backtracking operators}
\label{sec:linalg}

In the previous section, we reduced the proof of our main theorem to understanding the spectral properties of the matrix $A_{\mathbb{Z}}$ from \cref{def:AYmatrix}. This matrix can be viewed as the weighted adjacency matrix of a certain weighted graph. To make the analysis more accessible to the trace method, we instead consider a related object: the \emph{non-backtracking matrix} of this graph. We then use the \emph{Ihara-Bass formula} to transfer spectral information back to the original adjacency matrix.

\begin{remark}
    \label{rem:nonbacktrackingtracemethod}
    The use of non-backtracking operators in conjunction with the high trace method has found various applications in other areas of mathematics and computer science, see e.g.~\cite{usewatanabe,bordenaveAlon,fan,friedman}. In particular, Friedman~\cite{friedman} and Bordenave~\cite{bordenaveAlon} employed this approach in their proofs of Alon's second eigenvalue conjecture, which states that random $d$-regular graphs are almost Ramanujan (i.e., near-optimal spectral expanders).
\end{remark}

\subsection{Adjacency and non-backtracking operators}
We now introduce the necessary definitions and notation for working with weighted graphs and their associated linear operators.

\begin{definition}
    \label{def:graphs}
    A \emph{weighted graph} $G$ is a triple $(V, E, w)$ where
    \begin{itemize}
        \item $V$ is a finite set of vertices,
        \item $E \subseteq \binom{V}{2}$ is the set of edges,
        \item $w : E \to \R$ is a real-valued weight function.
    \end{itemize}
    The corresponding set of \emph{directed edges} is defined as
    \begin{equation*}
        \oE := \{(x,y) \in V^2 : \{x,y\} \in E\}.
    \end{equation*}
    For any $(x,y) \in \oE$, we define its weight as $w(x,y) := w(\{x,y\})$. For pairs $(x,y) \notin \oE$, we set $w(x,y) := 0$. Note that this implies $w(x,y) = w(y,x)$ and $w(x,x)=0$.
\end{definition}

\begin{notation}
    \label{not:vectorspaces}
    For any set $S$, we write $L(S)$ for the $\C$-vector space freely generated by $S$. Thus, $L(V)$ and $L(\oE)$ are $\C$-vector spaces with bases $V$ and $\oE$ respectively, and we denote their corresponding basis vectors by $\texttt{x}$ for $x \in V$ and $\overrightarrow{\texttt{xy}}$ for $(x,y) \in \oE$.
\end{notation}

\begin{definition}
    \label{def:adjacency}
    Let $G = (V, E, w)$ be a weighted graph. The \emph{adjacency operator} $A$ associated to~$G$ is the linear map $A : L(V) \to L(V)$ defined for each basis vector $\texttt{x}$ of $L(V)$ by
    \begin{equation*}
        A(\texttt{x}) := \sum_{y\in V} w(x,y)\, \texttt{y}.
    \end{equation*}
    The \emph{non-backtracking operator} $M$ associated to $G$ is the linear map $M : L(\oE)\to L(\oE)$ acting on a basis vector $\vec{\texttt{xy}}$ (where $(x,y)\in \oE$) as
    \begin{equation*}
        M(\overrightarrow{\texttt{xy}}):= \sum_{\substack{(y,z)\in \oE\\ z\neq x}} w(y,z)\, \overrightarrow{\texttt{yz}}
    \end{equation*}
    and extended linearly to $L(\oE)$.
\end{definition}

We will sometimes identify these operators with their matrices in the natural bases of $L(V)$ and $L(\oE)$, respectively. Note that $A$ is symmetric, while $M$ is usually not.

\subsection{The Ihara-Bass formula}
The \emph{Ihara-Bass formula} establishes a remarkable connection between the characteristic polynomial of the non-backtracking operator and a deformed characteristic polynomial of the adjacency operator. This formula originates from the work of Ihara~\cite{ihara}, who proved it in an algebraic setting corresponding to regular graphs, and was later extended by Bass~\cite{bass} to the general unweighted case. The paper by Stark and Terras~\cite{stark} is a recommended source for a transparent and elegant proof.

In our setting, we require a version of the Ihara-Bass formula applicable to weighted graphs. We have identified three papers that seem to have derived such an extension independently~\cite{deitmar,fan,watanabe}. We state below a formulation adapted to our notation.

\begin{proposition}[Ihara-Bass formula]
    \label{prop:iharabass}
    Let $G = (V, E, w)$  be a weighted graph, with adjacency and non-backtracking operators $A$ and $M$. Define the linear maps ${D_u, O_u : L(V) \to L(V)}$ for each basis vector $\text{\upshape{\texttt{x}}}$ by
    \begin{equation}
        \label{eq:DuOu}
        D_{u} (\text{\upshape{\texttt{x}}}) := \sum_{y\in V} \frac{w(x,y)^2}{1-u^2 w(x,y)^2} \text{\upshape{\texttt{x}}} \quad \text{and}\quad O_{u} (\text{\upshape{\texttt{x}}}) := u \sum_{y\in V} \frac{w(x,y)^3}{1-u^2 w(x,y)^2} \text{\upshape{\texttt{y}}},
    \end{equation}
    where $u\in \C$ is a parameter chosen such that all denominators are non-zero. Then,\footnote{In the formula \cref{eq:IharaBass}, the same letter $I$ denotes the identity operator on $L(\oE)$ on the left-hand side, and the identity operator on $L(V)$ on the right-hand side.}
    \begin{equation}
        \label{eq:IharaBass}
        \det(I-uM) =  \det\!\big(I-uA+u^2 (D_u-O_u)\big) \prod_{\{x,y\}\in E}\big(1-u^2w(x,y)^2\big).
    \end{equation}
\end{proposition}

We detail in \cref{appendix:nonbacktracking} how \cref{prop:iharabass} can be derived from the results in~\cite{deitmar}, where the notation differs slightly from ours.

The Ihara-Bass formula will be applied to the graph whose edges are weighted by the function $w_Y$ from \cref{def:AYmatrix}.

\begin{definition}[$G_Y$, $M_Y$, $D_{Y,u}$, $O_{Y,u}$]
    \label{def:specificYgraph}
    For any set $Y\subset \Z$, we define $G_Y$ to be the weighted graph with vertex set $\N$ and weight function $w_Y$.\footnote{The edge set of $G_Y$ is defined to be the set of all $\{m,n\}\subset \N$ with $w_Y(m,n)\neq 0$.}

    We write $A_Y$ and $M_Y$ for the adjacency and non-backtracking operators associated to $G_Y$.\footnote{Note that this definition of $A_Y$ is consistent with the one in \cref{def:AYmatrix}.} In addition, we denote by $D_{Y,u}$ and $O_{Y,u}$ the operators $D_u$ and $O_u$ from \cref{prop:iharabass}, specialised to the weighted graph~$G_Y$.
\end{definition}

We now state a central proposition, which provides a bound for almost all eigenvalues of the non-backtracking operator $M_Y$, for some well-chosen set $Y$. Recall that eigenvalues are always counted with multiplicity.

\begin{proposition}[Spectral tail bound for $M_Y$]
    \label{prop:nonbacktrackingeigenbound}
    Let $N$ be an integer such that $\log N \geq (\log H)^{2.001}$.

    Then, there exists a set $Y\subset \Z$ with the following properties.
    \begin{enumerate}
        \item  $\abs{\N \setminus Y} \ll H_0^{-1/3}N$.
        \item  All but $\leq N/H^3$ eigenvalues of $M_Y$ have absolute value at most $e^{O(\J)}\V^{\J/2}$.
    \end{enumerate}
\end{proposition}

The plan for the rest of this section is to use the Ihara-Bass formula to relate the spectrum of $M_Y$ to that of the adjacency operator $A_Y$. This will allow us to prove \cref{prop:OGmatrixeigenbound} assuming \cref{prop:nonbacktrackingeigenbound}. The proof of \cref{prop:nonbacktrackingeigenbound} will occupy the rest of the paper.

\begin{remark}
    \label{rem:iharabasscomparison}
    Several previous works have also used the Ihara-Bass formula to relate the spectra of the adjacency and non-backtracking operators~\cite{usewatanabe,bordenaveAlon,friedman}.

    For $d$-regular graphs (the setting of~\cite{bordenaveAlon,friedman}), this relation is immediate: the nontrivial eigenvalues $\lambda$ of the non-backtracking operator are obtained explicitly as the roots of the quadratic equations $\lambda^2 - \mu \lambda + (d-1) = 0$, where $\mu$ ranges over the eigenvalues of the adjacency operator.

    The random graphs considered in \cite{usewatanabe} have a maximum degree that is, with high probability, very close to the average degree (a consequence of concentration of measure). This near-regularity allows the perturbative term in the Ihara-Bass formula to be bounded uniformly.

    In this paper, the weighted graph $G_Y$ is highly irregular, with some vertices having very large weighted degrees. Moreover, we require a quantitative analysis that controls not only the magnitude, but also the number of large eigenvalues.
\end{remark}

\subsection{Integers with many prime factors}

As explained in \cref{rem:highlydivisiblenumbers}, the integers $n\in \N$ with an unusually large number of prime factors in $\primes$ are responsible for large eigenvalues of the adjacency operator $A_Y$ (which disappear when passing to the non-backtracking operator $M_Y$). To handle these large eigenvalues, we study the restriction of $A_Y$, $D_{Y,u}$, and $O_{Y,u}$ to the integers not having much more than the typical number of prime factors from $\primes$.

\begin{definition}
    \label{def:highlydivisiblenumbers}
    Let $\N^{\mathrm{typ}}$ be the set of integers $n\in \N$ such that $\omega_{\primes}(n) \leq 2\J\V$.
\end{definition}

\Cref{lem:operatorbound} shows that the matrices $D_{Y,u}$ and $O_{Y,u}$, when restricted to $\N^{\mathrm{typ}}$, have small operator norms, uniformly in $u\in [-1/2, 1/2]$. It relies on the following elementary bound.

\begin{lemma}[$\ell_1-\ell_{\infty}$ bound]
    \label{lem:l2l1linfty}
    Let $A \in M_n(\C)$ be a matrix with entries $(a_{ij})_{1\leq i,j\leq n}$. Then
    \begin{equation*}
        \norm{A}_{\mathrm{op}}^2 \leq \bigg(\!\max_{i} \sum_{j=1}^n |a_{ij}|\bigg)\bigg(\!\max_{j} \sum_{i=1}^n |a_{ij}| \bigg).
    \end{equation*}
\end{lemma}
\begin{proof}
    By Cauchy-Schwarz, for any $v\in \C^n$ we have
    \begin{equation*}
        \norm{Av}_2^2 \leq \sum_{i=1}^n \bigg(\sum_{j=1}^n |a_{ij}|\bigg)\bigg(\sum_{j=1}^n |a_{ij}| |v_j|^2\bigg) \leq  \bigg(\!\max_{i} \sum_{j=1}^n |a_{ij}|\bigg)\bigg(\!\max_{j} \sum_{i=1}^n |a_{ij}| \bigg) \norm{v}_2^2,
    \end{equation*}
    which gives the claimed operator norm bound.
\end{proof}

\begin{lemma}
    \label{lem:operatorbound}
    Let $Y\subset \Z$ be any set. Then, for any $u\in [-1/2, 1/2]$,
    \begin{equation*}
        \big\|{D_{Y,u}|_{\N^{\mathrm{typ}}} \big\|}_{\mathrm{op}},  \big\|{O_{Y,u}|_{\N^{\mathrm{typ}}} \big\|}_{\mathrm{op}} \leq e^{O(\J)} \V^{\J}.
    \end{equation*}
\end{lemma}

\begin{proof}
    We apply the simple $\ell_1-\ell_{\infty}$ bound given by \cref{lem:l2l1linfty}. Uniformly in $u\in [-1/2, 1/2]$, using that $\norm{w_Y}_{\infty}\leq 1$, we get
    \begin{equation*}
        \big\|{D_{Y,u}|_{\N^{\mathrm{typ}}} \big\|}_{\mathrm{op}},  \big\|{O_{Y,u}|_{\N^{\mathrm{typ}}} \big\|}_{\mathrm{op}} \ll \max_{m\in \N^{\mathrm{typ}}} \sum_{n\in \N^{\mathrm{typ}}} \abs{w_Y(m,n)}.
    \end{equation*}
    By the AM-GM inequality, for any $m\in \N^{\mathrm{typ}}$, we have
    \begin{equation*}
        \sum_{n\in \N^{\mathrm{typ}}} \abs{w_Y(m,n)} \ll \sum_{d\in \D}\prod_{p \mid d}\left(\ind{p\mid m}+\frac{1}{p}\right) = \prod_{j=1}^{\J} \big(\omega_{\primes_j}(m) + \V_j\big) \leq \left(\frac{\omega_{\primes}(m)}{\J} + \V\right)^{\J}.
    \end{equation*}
    Since $\omega_{\primes}(m) \leq 2\J\V$ by definition of $\N^{\mathrm{typ}}$, the claimed bound follows.
\end{proof}

Intuitively, \cref{lem:operatorbound} indicates that the terms $D_{Y,u}$ and $O_{Y,u}$ in the Ihara-Bass formula \cref{eq:IharaBass} should constitute only a small perturbation to the characteristic polynomial of~$A_Y$, once we restrict to $\N^{\mathrm{typ}}$. Moreover, this restriction should only affect a few eigenvalues, since the condition $\omega_{\primes}(n) \leq 2\J\V$ is satisfied by most integers $n\in \N$.

To make these ideas precise, we employ a topological argument together with Cauchy's interlacing theorem. The following lemma will be used to handle eigenvalues occurring with multiplicity greater than one.

\begin{lemma}
    \label{lem:topology}
    Let $P\in \C[X]$ be a non-zero polynomial. Suppose that, for $x\in [0, 1]$,
    \begin{equation*}
        P(x) = f_1(x) f_2(x) \cdots f_n(x) g(x)
    \end{equation*}
    where $f_i:[0, 1]\to \C$ are Lipschitz continuous functions and $g:[0, 1]\to \C$ is a continuous function. If each $f_i$ has a root in $(0, 1)$, then $P$ has at least $n$ roots in $(0, 1)$, counting multiplicities.
\end{lemma}

\begin{proof}
    Let $x_0\in (0, 1)$ be a point such that $f_i(x_0) = 0$ for all $i$ in a subset $I\subset \{1, \ldots, n\}$. It suffices to show that $P^{(k)}(x_0) = 0$ for all $0\leq k< \abs{I}$. We prove this by induction on $k$. Suppose that $P(x_0) = P'(x_0) = \cdots = P^{(k-1)}(x_0) = 0$ for some $0\leq k< |I|$. By Taylor expansion,
    \begin{equation*}
        P^{(k)}(x_0) = k! \lim_{x\to x_0} \frac{P(x)}{(x-x_0)^k} = k! g(x_0)\prod_{i\notin I} f_i(x_0)\cdot  \lim_{x\to x_0} (x-x_0)^{|I|-k}  \prod_{i\in I} \frac{f_i(x)}{x-x_0}.
    \end{equation*}
    Since each $f_i$ is Lipschitz continuous, the last product remains bounded as $x\to x_0$, and thus $P^{(k)}(x_0) = 0$ as desired.
\end{proof}

Using \cref{lem:operatorbound,lem:topology}, we can now relate the eigenvalues of the restricted adjacency matrix $A_Y|_{\N^{\mathrm{typ}}}$ to those of the non-backtracking matrix $M_Y$ via the Ihara-Bass formula.

To do this, we require two classical facts on eigenvalues of Hermitian matrices. For any Hermitian matrix $M\in M_n(\C)$, write ${\lambda_1(M) \geq \ldots \geq \lambda_n(M)}$ for the eigenvalues of $M$ in descending order.

\begin{lemma}[Weyl inequalities]
    \label{lem:weylineq}
    Let $A,B\in M_n(\C)$ be Hermitian matrices. Then
    \begin{equation*}
        \lambda_{i+h}(A) + \lambda_{n-h}(B) \leq \lambda_i(A+B) \leq \lambda_{i-h}(A) + \lambda_{h+1}(B)
    \end{equation*}
    for all $1\leq i\leq n$ and $0\leq h\leq \min(i-1, n-i)$.
\end{lemma}
\begin{proof}
    This is \cite[Theorem~III.2.1]{matrixanalysis}.
\end{proof}

\begin{lemma}[Cauchy's interlacing theorem]
    \label{lem:interlacing}
    Let $A \in M_n(\C)$ be a Hermitian matrix and let $B$ be the matrix obtained by deleting $r$ rows of $A$ and the $r$ corresponding columns, for some $r< n$. Then
    \begin{equation*}
        \lambda_{i+r}(A) \leq \lambda_i(B) \leq \lambda_i(A)
    \end{equation*}
    for all $1\leq i\leq n-r$.
\end{lemma}
\begin{proof}
    This is \cite[Corollary~III.1.5]{matrixanalysis}.
\end{proof}

\begin{lemma}
    \label{lem:adjmatrixeigenbound}
    Let $C\geq 1$ be a sufficiently large constant. Let $Y\subset \Z$ be any set and let $k\geq 1$.

    If the restricted adjacency matrix $A_Y|_{\N^{\mathrm{typ}}}$ has at least~$k$ eigenvalues larger than $2C^{\J}\V^{\J/2}$ in absolute value, then the non-backtracking matrix $M_Y$ must have at least $k$ eigenvalues larger than $C^{\J}\V^{\J/2}$ in absolute value.
\end{lemma}

\begin{proof}
    Since $A_Y|_{\N^{\mathrm{typ}}}$ is real and symmetric, its eigenvalues are real. We treat the case of positive eigenvalues, the negative case being similar.

    Suppose that $A_Y|_{\N^{\mathrm{typ}}}$ has at least $k$ positive eigenvalues larger than $2C^{\J}\V^{\J/2}$. By the Weyl inequalities (\cref{lem:weylineq}), the $k$ largest eigenvalues of the perturbed matrix
    \begin{equation*}
        \big(A_Y - u (D_{Y,u}-O_{Y,u})\big)\big|_{\N^{\mathrm{typ}}}
    \end{equation*}
    are bounded below by
    \begin{equation}
        \label{eq:lowerboundperturbeigen}
        2 C^{\J}\V^{\J/2} - u \big\|{D_{Y,u}|_{\N^{\mathrm{typ}}} \big\|}_{\mathrm{op}} -u \big\|{O_{Y,u}|_{\N^{\mathrm{typ}}} \big\|}_{\mathrm{op}}.
    \end{equation}
    By \cref{lem:operatorbound}, the expression \cref{eq:lowerboundperturbeigen} is greater than $C^\J\V^{\J/2}$ for every $u \in [-\V^{-\J/2}, \V^{-\J/2}]$, provided $C$ is sufficiently large. Hence, by Cauchy's interlacing theorem (\cref{lem:interlacing}), the unrestricted matrix
    \begin{equation*}
        A_Y - u (D_{Y,u}-O_{Y,u})
    \end{equation*}
    has at least $k$ eigenvalues greater than $C^{\J}\V^{\J/2}$, for every $u \in [-\V^{-\J/2}, \V^{-\J/2}]$.

    Let $\mathcal{H}_N$ denote the space of $N \times N$ Hermitian matrices, equipped with the operator norm. For $A\in \mathcal{H}_N$, let $\lambda_1(A) \geq \ldots \geq \lambda_N(A)$ be the eigenvalues of $A$. The Weyl inequalities (\cref{lem:weylineq}) imply that each $\lambda_i$ is a $1$-Lipschitz function on $\mathcal{H}_N$. Since the map $u\mapsto A_Y - u (D_{Y,u}-O_{Y,u}) \in \mathcal{H}_N$ is smooth over $-1<u<1$ (as $|w_Y| \leq 1$), the function
    \begin{equation*}
        f_i(u) := 1 - u \lambda_i\big(A_Y - u (D_{Y,u}-O_{Y,u}) \big)
    \end{equation*}
    is Lipschitz on $[-1/2, 1/2]$, for every $1\leq i\leq N$.

    For each $i$, we have $f_i(0) = 1$. Moreover, for $1\leq i \leq k$, we have $f_i(C^{-\J}\V^{-\J/2}) < 0$ by our lower bound on the $k$ largest eigenvalues of $A_Y - u (D_{Y,u}-O_{Y,u})$. Therefore, the functions $f_1, \ldots, f_k$ each have a root in the interval $(0, C^{-\J}\V^{-\J/2})$.

    By the Ihara-Bass formula (\cref{prop:iharabass}), for $|u| < 1$ we have the identity
    \begin{equation*}
        \det(I-uM_Y) = \prod_{i=1}^N f_i(u)\cdot  g(u),
    \end{equation*}
    where $g$ is a continuous function that can be written down explicitly. Applying \cref{lem:topology}, we conclude that $\det(I-uM_Y)$ has at least $k$ roots in $(0, C^{-\J}\V^{-\J/2})$, counting multiplicities. In other words, $M_Y$ has at least $k$ eigenvalues greater than $C^{\J} \V^{\J/2}$, as desired.
\end{proof}

\subsection{Eigenvalues of localised matrices}
A key feature of our adjacency matrix $A_Y$ is that its non-zero entries lie close to the diagonal. The next lemma states that, if such a Hermitian matrix has a very small number of large eigenvalues, it is possible to delete a small number of rows and columns to obtain a matrix without any large eigenvalue.

\begin{lemma}
    \label{lem:localised}
    Let $N \geq 10 H^3$. Let $A = (a_{m,n})_{m, n\in \N}$ be a Hermitian matrix such that $a_{m,n} = 0$ whenever ${|m-n| \geq H}$.

    Let $\alpha>0$, and suppose that $A$ has $\leq N/H^3$ eigenvalues with absolute value $\geq \alpha$. Then, there is a subset $E \subset \N$ with $\abs{\N\setminus E} \ll N/H$ such that every eigenvalue of $A|_E$ has absolute value $\leq \alpha$.
\end{lemma}

\begin{proof}
    Consider a sequence of disjoint discrete intervals
    \begin{equation*}
        B_1, \ldots, B_q \subset \N,
    \end{equation*}
    each of length $|B_i| = H^2$. We define this sequence such that for each $i$, $B_i$ is separated from $B_{i+1}$ by a gap of size exactly $H$. Choosing $q$ to be maximal with these properties, we have $q \asymp N/H^2$ and
    \begin{equation}
        \label{eq:boundNminusB}
        \abs{\N \setminus \bigsqcup_{i=1}^q B_i} \ll qH + H^2 \ll N/H.
    \end{equation}

    Since $B_1, \ldots, B_q$ have pairwise distance $\geq H$, the property of $A$ in the statement implies that the principal submatrix $A' := A|_{\bigsqcup_{1\leq i\leq q} B_i}$ is block-diagonal, with blocks $(A|_{B_i})_{1\leq i \leq q}$.

    Let $\indexset$ be the set of all indices $i$ such that $A|_{B_i}$ has an eigenvalue whose absolute value is $>\alpha$. Then, defining
    \begin{equation*}
        E := \bigsqcup_{\substack{1\leq i\leq q\\ i\notin \indexset}} B_i,
    \end{equation*}
    the matrix $A|_E$ has no eigenvalue with absolute value $>\alpha$.

    By \cref{eq:boundNminusB}, we have
    \begin{equation*}
        \abs{\N \setminus E} \ll |\indexset| H^2 + N/H.
    \end{equation*}
    Observe moreover that $|\indexset| \leq N/H^3$. Indeed, $A$ has $\leq N/H^3$ eigenvalues with absolute value $\geq \alpha$, so by Cauchy's interlacing theorem (\cref{lem:interlacing}) the same holds for its principal submatrix $A'$. This implies the desired bound $\abs{\N \setminus E} \ll N/H$.
\end{proof}

We can now prove that \cref{prop:nonbacktrackingeigenbound} implies \cref{prop:OGmatrixeigenbound}. We remark that \cref{lem:localised} will not be applied directly to $A_Y$, but rather to its restriction to $\N^{\mathrm{typ}}$. This is because $A_Y$ admits large eigenvalues arising from integers outside $\N^{\mathrm{typ}}$ (see \cref{rem:highlydivisiblenumbers}), in greater number than allowed by \cref{lem:localised}.

\begin{proof}[Proof of \cref{prop:OGmatrixeigenbound}, assuming \cref{prop:nonbacktrackingeigenbound}]
    Let $Y\subset \Z$ be the set provided by \cref{prop:nonbacktrackingeigenbound}. In particular,
    \begin{equation}
        \label{eq:boundIminusY}
        \abs{\N \setminus Y} \ll H_0^{-1/3}N.
    \end{equation}
    In addition, all but $\leq N/H^3$ eigenvalues of the non-backtracking matrix $M_Y$ have absolute value at most $e^{O(\J)}\V^{\J/2}$. By \cref{lem:adjmatrixeigenbound}, the same holds for the restricted adjacency matrix $A_Y|_{\N^{\mathrm{typ}}}$.

    Applying \cref{lem:localised} to $A_Y|_{\N^{\mathrm{typ}}}$, we obtain a subset $E\subset \N$ with
    \begin{equation}
        \label{eq:boundIminusE}
        |\N\setminus E|\ll N/H
    \end{equation}
    such that all eigenvalues of $A_Y|_{\N^{\mathrm{typ}}\cap E}$ are $\leq e^{O(\J)}\V^{\J/2}$. In other words, defining $X := \N^{\mathrm{typ}}\cap E\cap Y$, all eigenvalues of $A_\Z|_{X}$ are $\leq e^{O(\J)}\V^{\J/2}$.

    By \cite[Eq.~(1.11)]{restrictedprimes}, we have
    \begin{equation}
        \label{eq:boundINP}
        \big\lvert\N \setminus \N^{\mathrm{typ}}\big\rvert = \sum_{n\in \N} \ind{\omega_{\primes}(n) > 2\J \V} \ll e^{-\J \V/3}N.
    \end{equation}
    Combining \cref{eq:boundIminusY,eq:boundIminusE,eq:boundINP}, we conclude that
    \begin{equation*}
        \abs{\N\setminus X}\ll \big(H_0^{-1/3}+H^{-1}+e^{-\J \V/3}\big)N \ll e^{-\J \V/3} N,
    \end{equation*}
    as $\J,\V \ll \log \log H \asymp \log \log H_0$ by part \cref{item:primesets3} of \cref{lem:primesets}. This proves \cref{prop:OGmatrixeigenbound}.
\end{proof}

\section{High trace bound for the non-backtracking operator}
\label{sec:trace}

The high trace method is a standard technique to control the eigenvalues of square matrices: if $M$ is Hermitian and $\Tr(M^{k})\leq C$ for some $k\in 2\mathbb{N}$, then every eigenvalue of $M$ satisfies $|\lambda_i(M)| \leq C^{1/k}$, because $\Tr(M^{k})$ is the $k$-th moment of the eigenvalues of~$M$. A similar idea can be used to control the tail of the spectrum of general square matrices.

\begin{lemma}[Schur's inequality]
    \label{lem:schurinequality}
    Let $A \in M_n(\C)$ be a matrix with eigenvalues $\lambda_1(A), \ldots, \lambda_n(A)$. Then
    \begin{equation}
        \label{eq:schurineq}
        \sum_{i=1}^n \abs{\lambda_i(A)}^2 \leq \Tr(A^*A).
    \end{equation}
\end{lemma}
\begin{proof}
    Observe that $\Tr(A^*A) = \sum_{1\leq i,j\leq n} |a_{ij}|^2$, where $(a_{ij})$ are the entries of $A$. As a result, the inequality \cref{eq:schurineq} holds when $A$ is triangular. Since every square matrix is unitarily similar to a triangular matrix, the general case follows.
\end{proof}

\subsection{Statement of the high trace bound}
In \cref{prop:hightracebound}, we state a suitable high trace bound for the non-backtracking operator $M_Y$. We then deduce the spectral tail bound (\cref{prop:nonbacktrackingeigenbound}) from it. To effectively bound the number of large eigenvalues of $M_Y$, the power $k$ in the high trace method must be sufficiently large. In our setting, this requires taking a power that is logarithmic in $H$.

\begin{notation}
    \label{not:K}
    We define $\K := \lfloor \log H \rfloor$.
\end{notation}

\begin{restatable}[High trace bound for $M_Y$]{proposition}{hightracebound}
    \label{prop:hightracebound}
    Let $N$ be an integer such that $\log N \geq (\log H)^{2.001}$.

    Then, there exists a set $Y\subset \Z$ with $\abs{\N \setminus Y} \ll H_0^{-1/3}N$, such that the non-backtracking operator $M_Y$ associated to the weighted graph $G_Y$ satisfies the high trace bound
    \begin{equation}
        \label{eq:hightracebound}
        \Tr\big[\big((M_{Y})^{\K}\big)^*(M_{Y})^{\K}\big] \leq e^{O(\J \K)} \V^{\J\K} N.
    \end{equation}
\end{restatable}

\begin{proof}[Proof of \cref{prop:nonbacktrackingeigenbound}, assuming \cref{prop:hightracebound}]
    Let $Y\subset \Z$ be the set given in \cref{prop:hightracebound}. Let $(\lambda_i)_{i\in I}$ be the (complex) eigenvalues of $M_Y$. By the high trace bound \cref{eq:hightracebound} and Schur's inequality (\cref{lem:schurinequality}), these satisfy
    \begin{equation*}
        \sum_{i\in I} |\lambda_i|^{2\K} \leq \Tr\big[\big((M_{Y})^{\K}\big)^*(M_{Y})^{\K}\big] \leq e^{O(\J \K)} \V^{\J\K} N,
    \end{equation*}
    using that the eigenvalues of $(M_Y)^{\K}$ are the $\K$-th powers of the eigenvalues of $M_Y$. By Markov's inequality, this implies that the number of eigenvalues $\lambda_i$ with $\abs{\lambda_i} \geq C^{\J}\V^{\J/2}$ is at most
    \begin{equation*}
        e^{O(\J \K)}C^{-2\J\K}N.
    \end{equation*}
    Since $\K \asymp \log H$, this number can be made smaller than $N/H^3$, provided that $C$ is chosen to be a sufficiently large absolute constant.
\end{proof}

The rest of this paper is devoted to the proof of \cref{prop:hightracebound}. Henceforth, we fix an integer $N$ such that
\begin{equation}
    \label{eq:boundforN}
    \log N \geq (\log H)^{2.001}.
\end{equation}
In particular, we have the useful bound
\begin{equation}
    \label{eq:boundHJK}
    H^{\J\K} \ll_{\eps} N^{\eps},
\end{equation}
which follows from \cref{eq:boundforN} and the estimates $\J \leq \log \log H$ and $\K \asymp \log H$.

\subsection{Expanding the trace as a sum over walks} To prove \cref{prop:hightracebound}, we work on the physical side, computing the trace directly by expanding the matrix product.

\begin{lemma}
    \label{lem:tracesimplebound}
    Let $Y\subset \Z$, and let $M_Y$ be the non-backtracking operator associated to $G_Y$. Then
    \begin{equation}
        \label{eq:traceinlemma}
        \Tr\big[\big((M_{Y})^{\K}\big)^*(M_{Y})^{\K}\big] \ll H\!\!\!\!\! \sum_{\substack{d_1, \ldots, d_{2\K}\in \pm \D\\ d_{i+1}=-d_i \iff i=\K}} \Bigg\lvert \sum_{\substack{n\in \N\\ n+b_i\in Y\,\text{for }0\leq i\leq 2\K}} \prod_{i=1}^{2\K} \prod_{p \mid d_i}\left(\ind{p\mid n + b_i}-\frac{1}{p}\right)\Bigg\rvert \, +\,  N,
    \end{equation}
    where $b_i := \sum_{1\leq k\leq i} d_k$ for $0\leq i \leq 2\K$.
\end{lemma}

\begin{proof}
    We use the formula for $\Tr[(M^{\K})^*M^{\K}]$ given in \cref{lem:traceexpansion}, which applies to the non-backtracking operator $M$ of a general weighted graph. This gives
    \begin{equation}
        \label{eq:traceexpr1}
        \Tr\big[\big((M_{Y})^{\K}\big)^*(M_{Y})^{\K}\big] = \sum_{\substack{n_{-1}, n_0, n_1, \ldots, n_{2\K}\in \N\\ (n_{-1},n_0, \ldots, n_{\K}) \in \NB{G_Y}\\ (n_{\K}, n_{\K+1}, \ldots, n_{2\K}, n_{-1})\in \NB{G_Y}\\ (n_{0}, n_{\K-1})=(n_{2\K}, n_{\K+1})}} \prod_{i=1}^{2\K} w_Y(n_{i-1}, n_i),
    \end{equation}
    where $\NB{G_Y}$ denotes the set of non-backtracking walks in $G_Y$ (see \cref{def:nbwalks}).

    We rewrite this expression in terms of $n_0$ and the increments $d_i := n_{i} - n_{i-1}$ for $0\leq i \leq 2\K$. Since we only need an upper bound, we can apply the triangle inequality and keep only the sum over $n_0$ inside the absolute value:
    \begin{equation}
        \label{eq:traceexpr2}
        \Tr\big[\big((M_{Y})^{\K}\big)^*(M_{Y})^{\K}\big] \leq \sum_{\substack{d_0, \ldots, d_{2\K}\in \pm \D\\ d_{i+1}\neq - d_i \text{ for } 0\leq i < \K\\ d_{i+1}\neq - d_i \text{ for } \K< i < 2\K\\ d_{\K+1}=-d_{\K}}} \Bigg\lvert \sum_{\substack{n_0\in \N\\ n_0-d_0\in \N}} \prod_{i=1}^{2\K} w_Y\bigg(n_0 + \sum_{j=1}^{i-1} d_j,\, n_0+\sum_{j=1}^{i} d_j\bigg) \Bigg\rvert,
    \end{equation}
    with the convention that $w_Y(m,n)=0$ if $m\notin \N$ or $n\notin \N$. In passing from \cref{eq:traceexpr1} to \cref{eq:traceexpr2}, we have dropped the conditions $n_0=n_{2\K}$ and $n_{2\K-1}\neq n_{-1}$, which can be expressed in terms of the increments $d_i$ alone (independently of $n_0$).

    If $n_0\in \N$, the requirement that $n_0-d_0\in \N$ is automatically satisfied unless $n_0$ lies within distance $O(H)$ of the boundary of $\N$. Likewise, for each $1\leq i\leq 2\K$, the implicit condition ${n_0 + \sum_{j=1}^{i}d_j\in \N}$ can only fail if $n_0$ lies within distance $O(\K H)$ of the boundary. Discarding all these conditions trivially, we incur an error of size $\ll \K H (2H)^{2\K+1}$, which is at most $O(N)$ by~\cref{eq:boundHJK}.

    Expanding the formula for $w_Y$ (see \cref{def:AYmatrix}), the trace \cref{eq:traceexpr2} is thus bounded by
    \begin{equation*}
        \sum_{\substack{d_0, \ldots, d_{2\K}\in \pm \D\\ d_{i+1}=-d_i \iff i=\K}} \Bigg\lvert \sum_{\substack{n_0\in \N}} \prod_{i=1}^{2\K} \prod_{p \mid d_i}\left(\ind{p\mid n_0 + \sum_{j=1}^{i-1} d_j}-\frac{1}{p}\right)\prod_{i=0}^{2\K} \ind{Y}\bigg(n_0 + \sum_{j=1}^{i} d_j\bigg)\Bigg\rvert + O(N).
    \end{equation*}
    Notice that the summand is independent of the variable $d_0$, which ranges over $\ll\abs{\D}\ll H$ values. The lemma follows.
\end{proof}

\section{Overview of the proof of the high trace bound}
\label{sec:heuristics}

This section serves purely as motivation and is separate from the actual proof. Our aim here is to outline the strategy for bounding the right-hand side of \cref{eq:traceinlemma}. We will also explain why the two technical assumptions that appear there -- the non-backtracking condition $d_{i+1}\neq -d_i$ (for almost all $i$) and the restriction of the vertices $n+b_i$ to a well-chosen set $Y$ -- are essential to our arguments.

Let us first examine the situation without these hypotheses. For each $(i,j)\in \intK\times \intJ$, let $\dsub{i}{j}$ be the unique prime in $\primes_j$ dividing $d_i$. Ignoring the two technical hypotheses, the right-hand side of~\cref{eq:traceinlemma} essentially reduces to
\begin{equation}
    \label{eq:caricaturaltrace}
    \sum_{d_1, \ldots, d_{2\K}\in \pm \D} \Bigg\lvert \sum_{n\in \N} \prod_{i=1}^{2\K}\prod_{j=1}^{\J} \bigg( \ind{\dsub{i}{j}\mid n + b_i}-\frac{1}{\dsub{i}{j}} \bigg) \Bigg\rvert.
\end{equation}

The basic idea is that the factors $\ind{\dsub{i}{j}\mid n + b_i}-\frac{1}{\dsub{i}{j}}$ should generate significant cancellation in the inner sum because, for a random integer $n\in \N$, the event $\dsub{i}{j}\mid n+b_i$ occurs with probability almost exactly $1/\dsub{i}{j}$. One must be careful, however: these factors are not independent whenever they involve the same prime, and these dependencies can prevent cancellation.

With this in mind, we call a prime $p\in\primes$ a \emph{single prime} for the tuple $(d_1,\dots,d_{2\K})$ if there is exactly one pair $(i,j)\in \intK\times \intJ$ such that $p = \dsub{i}{j}$. If a tuple has at least one single prime, then the inner sum in \cref{eq:caricaturaltrace} exhibits almost perfect cancellation, by the Chinese Remainder Theorem (since $N>d_1\cdots d_{2\K}$). Thus, we only need to consider the tuples for which every prime appearing in the array $(\dsub{i}{j})$ is repeated at least twice.

In this situation, we cannot hope to obtain cancellation. We may therefore use the triangle inequality and trivially bound $|\ind{\dsub{i}{j}\mid n + b_i}-\frac{1}{\dsub{i}{j}}|\leq \ind{\dsub{i}{j}\mid n + b_i} + \frac{1}{\dsub{i}{j}}$. Expanding the resulting product, we are reduced to bounding
\begin{equation}
    \label{eq:litunlitcaric}
    \sum_{\substack{d_1, \ldots, d_{2\K}\in \pm \D\\ p\mid \prod d_i \,\Rightarrow\, p^2\mid \prod d_i}} \sum_{n\in \N}  \prod_{(i,j)\in \lit} \ind{\dsub{i}{j}\mid n + b_i} \prod_{(i,j)\in \unlit} \frac{1}{\dsub{i}{j}},
\end{equation}
for every partition $\intK\times \intJ = \lit \sqcup \unlit$ of the index set into two sets $\lit$ and $\unlit$. The pairs $(i,j)\in \lit$ are called \emph{lit indices}, while those in $\unlit$ are called \emph{unlit indices}.

Heuristically, we might expect each factor $\ind{\dsub{i}{j}\mid n + b_i}$ to behave like $1/\dsub{i}{j}$ on average, in which case we would obtain a strong bound from the fact that each prime occurs at least twice. Indeed, the common contribution of all factors corresponding to a repeated prime $p$ would then be at most $O(1/p^2)$ on average, which is very small since all primes in $\primes$ are larger than $H_0$. A more precise implementation of this idea shows that the expression \cref{eq:litunlitcaric} is negligible whenever the set $\unlit$ of unlit indices is somewhat large. We may thus assume that nearly all pairs $(i,j)\in \intK\times \intJ$ lie in~$\lit$.

The `enemy scenario' is therefore when the factors $\ind{\dsub{i}{j}\mid n + b_i}$ conspire so that their product is, on average over $d_1, \ldots, d_{2\K}$ and $n$, significantly larger than the product of $1/\dsub{i}{j}$. To rule it out, we must analyse the joint behaviour of these indicator functions.

Recall that $b_i := \sum_{1\leq k\leq i} d_k$. If a prime $p$ is repeated twice at lit indices, say $p = \dsub{i_1}{j} = \dsub{i_2}{j}$ with $i_1 < i_2$, then the congruences $p\mid n + b_{i_1}$ and $p\mid n + b_{i_2}$ imply that
\begin{equation}
    \label{eq:divcaric}
    p \bigmid \sum_{i_1 < k < i_2} d_k = \sum_{i_1 < k < i_2} \sigma_k \prod_{j=1}^{\J} \dsub{k}{j},
\end{equation}
where $\sigma_k\in\{\pm 1\}$ is the sign of $d_k$. By considering all pairs of lit indices corresponding to each repeated prime, we obtain a large system of such divisibility relations. Intuitively, this system of constraints should considerably limit the number of admissible choices for the primes $\dsub{i}{j}$, and hence the number of tuples $(d_1,\dots,d_{2\K})$ contributing significantly to \cref{eq:litunlitcaric}.\footnote{Note that the techniques of Helfgott and Radziwi\l\l~\cite{HR} cannot be applied here as the right-hand side of \cref{eq:divcaric} is a polynomial expression in the primes $\dsub{i}{j}$, rather than a linear combination with small coefficients, as soon as $\J > 1$.}

The shape of the divisibility constraints is determined by the partition of the index set $\intK\times \intJ$ according to repetitions of the same prime, and the signs $\sigma_k$. Therefore, to organise our analysis, we first fix the shape, and then consider the number of possible assignments of prime values to the variables which satisfy that system of divisibility constraints.\footnote{That is, if we have a partition of $\intK\times \intJ$, we look for values $\dsub{i}{j}$ satisfying the constraints $${\dsub{i_1}{j_1} = \dsub{i_2}{j_2} \mid \sum_{i_1 < k < i_2} d_k}$$for all $(i_1,j_1),(i_2,j_2) \in \lit$ in the same part of the partition.} The number of possible shapes is very large (more than exponential in $\J\K$). We might hope that this is balanced by there being very few assignments that satisfy the associated divisibility system, and therefore that the total contribution is acceptably small.

Unfortunately, certain shapes do yield systems with many solutions, so we cannot expect a strong uniform bound that applies to every shape. Our two-step strategy is therefore as follows. First, obtain a good bound that holds for most shapes. Second, show that the remaining shapes are rare enough that their total contribution is acceptable.

It is not necessary to obtain an optimal bound for the number of solutions to divisibility systems. Roughly speaking, it suffices to show that the system determines the values of at least $K^{0.01}$ of the primes $\dsub{i}{j}$ based on the values of the remaining primes. We do so by extracting a reasonably large `triangular' subsystem of divisibility constraints which is clearly non-degenerate, via a careful combinatorial analysis.

This combinatorial argument, however, fails for certain degenerate shapes. This failure becomes critical because some of these shapes are not sufficiently rare to be handled by the second step of our strategy. Our two technical assumptions are introduced precisely to preclude these problematic configurations.

The first such issue arises from backtracking walks. When $d_{i+1}=-d_i$ for many $i$, many sums of the form $\sum_{i_1<k<i_2} d_k$ vanish, rendering the corresponding divisibility constraints \eqref{eq:divcaric} trivial. Moreover, the contribution of these backtracking walks to the trace is genuinely excessive (reflecting the existence of large eigenvalues of the adjacency operator), meaning that the simplistic expression~\cref{eq:caricaturaltrace} cannot possibly satisfy the desired bound. This explains the use of the non-backtracking operator throughout, which eliminates this class of degenerate shapes from the outset.

The second class of problematic shapes roughly corresponds to walks that involve too few distinct primes. In this case, we can only extract a small triangular subsystem, which is insufficient to compensate for the large number of combinatorial arrangements of those primes into long walks. To discard these shapes, we restrict the vertices $n+b_i$ to lie in a well-chosen set $Y$. The set $Y$ (defined precisely in \cref{sec:thesetY}) is chosen so that every such pathological walk visits at least one vertex outside of $Y$. Although working with $Y$ introduces significant technicalities,\footnote{Notably, using single primes to obtain cancellation over $Y$ rather than over $\N$ requires a highly technical analysis based on a \emph{combinatorial sieve} for composite moduli.} this device removes the second type of degeneracies that our combinatorial methods cannot handle.

With the non-backtracking condition and the restriction to $Y$ in place, a large triangular subsystem can be extracted for all shapes satisfying a generic \emph{unpredictability} assumption. The first step of the strategy is then complete: the contribution from these unpredictable shapes is negligible.

Finally, we analyse the remaining \emph{predictable} shapes, for which no large triangular subsystem can be found. We show that, while these shapes yield the main term in the trace, they are sufficiently rare that their contribution is admissible, which concludes the second step.

In summary, we use cancellation to show that walks containing (a sufficient number of) single primes contribute negligibly to the trace; this is done in \cref{sec:single}. The only remaining walks that are not easily handled are those with many repeated primes at lit indices. For these walks, no cancellation is available. Instead, we analyse the associated system of divisibility constraints. For most shapes of such systems (the \emph{unpredictable} case, treated in \cref{sec:unpredictable}), we can extract a sufficiently large triangular subsystem, which yields enough savings to render their contribution negligible. Our ability to perform this step for most shapes relies crucially on the non-backtracking condition and the restriction to the set $Y$, which remove the worst degeneracies. The residual walks (the \emph{predictable} case, treated in \cref{sec:predictable}) are sufficiently rare that their total contribution is also suitably bounded. Putting all these estimates together yields the desired high trace bound, which is essentially sharp. The remainder of the paper is devoted to turning this sketch into a rigorous argument.

\section{Definition of the set \texorpdfstring{$\Y$}{Y}}
\label{sec:thesetY}

We finally define a set $Y \subset \Z$ satisfying \cref{prop:hightracebound}: our specific choice is denoted by $\Y$. Note that this definition will not be used in earnest until \cref{sec:unpredictable}.

The role of $\Y$ is to exclude certain tuples $(d_1, \ldots, d_{2\K})$ from the expression in \cref{lem:tracesimplebound} that our methods cannot handle. The definition of $\Y$ is quite delicate: even a slight modification can disrupt key parts of the argument and render the high trace bound intractable.

The set $\Y$ is obtained from $\Z$ by removing certain sparse arithmetic progressions, called \emph{prohibited progressions}. These are defined in terms of \emph{prohibited sequences}, which are sequences of integers that satisfy a special type of divisibility relation.

To define prohibited sequences, we first fix a maximum length parameter $\L$. This parameter is fundamental to the definition of our set $\Y$, which is why we retain it in the notation. It is chosen to be a power of $\K$ slightly smaller than one.

\begin{notation}
    \label{not:L}
    We define $\L := \K^{1-10\epsone}$.
\end{notation}

\begin{definition}
    \label{def:reduced}
    Let $\ell \geq 1$. A sequence $(d_1, \ldots, d_{\ell})\in (\pm \D)^{\ell}$ is said to be \emph{reduced} if $d_{i+1} \neq -d_i$ for all $1\leq i < \ell$.
\end{definition}

\begin{definition}
    \label{def:prohibitedsequence}
    A \emph{prohibited sequence} is a reduced sequence $(d_1, \ldots, d_{\ell})$ of $\ell$ elements of $\pm \D$, for some $2<\ell\leq \L$, with the following properties:
    \begin{enumerate}
        \item (consecutiveness) for every prime $q$, the set $\{i\in \sinterval{\ell} : q\mid d_i\}$ is a discrete interval, and
        \item (prohibited pattern) there is a prime $p$ and some $1<\ell_0< \ell$ such that $p\mid d_1$, $p\nmid d_{\ell}$ and
              \begin{equation}
                  \label{eq:divisibilityindefprohib}
                  p\bigmid \sum_{\ell_0\leq i \leq \ell} d_i.
              \end{equation}
    \end{enumerate}
    A prohibited sequence $(d_1, \ldots, d_{\ell})$ is \emph{primitive} if there is no consecutive\footnote{By `consecutive subsequence of $(d_1, \ldots, d_{\ell})$', we mean a sequence of the form $(d_{k_1}, d_{k_1+1},\ldots, d_{k_2})$ for some ${1\leq k_1<k_2\leq \ell}$.} subsequence of $(d_1, \ldots, d_{\ell})$ or of $(d_{\ell}, \ldots, d_1)$, of length $<\ell$, which is also prohibited.
\end{definition}

A key difference with \cite{HR} is that, in their situation, the authors can restrict themselves to the case $\ell_0 = 1$. This is not possible here, and leads to additional complications in the proof of \cref{lem:existencerank} (due to the fact that the constraint \cref{eq:divisibilityindefprohib} only involves a subset of the prime factors of the $d_i$). We may now define the set $\Y$.

\begin{definition}[$\prohibprog$, $\Y$]
    \label{def:prohibitedprogression}
    The \emph{prohibited progression} associated to a primitive prohibited sequence $(d_1, \ldots, d_{\ell})$ is the set of all integers $n\in \Z$ such that
    \begin{equation*}
        d_1\mid n, \quad d_2\mid n+d_1, \quad \ldots, \quad d_{\ell}\mid n+d_1+\cdots+d_{\ell-1}.
    \end{equation*}
    It is an arithmetic progression of square-free modulus $\mathrm{lcm}(d_1,\ldots, d_{\ell})$ (the consecutiveness property in \cref{def:prohibitedsequence} ensures that it is non-empty).

    Let $\prohibprog$ be the set of all prohibited progressions associated with some primitive prohibited sequence. We define
    \begin{equation*}
        \Y := \mathbb{Z} \setminus \cup \prohibprog,
    \end{equation*}
    the set of all integers that do not belong to any prohibited progression.
\end{definition}

To apply \cref{prop:hightracebound} with $Y = \Y$, we need to show that $\Y$ covers almost all of $\N$; specifically, that $\abs{\N\setminus \Y} \ll H_0^{-1/3}N$. Although this is not difficult, we postpone the proof to \cref{sec:combinatorial} (see \cref{lem:Ysmall}), where we establish a broader class of results concerning prohibited sequences and progressions.

\section{Single and repeated primes}
\label{sec:tracecomputations}

In this section, we introduce the notation that we will use to study the high trace of the non-backtracking matrix $M_{\Y}$ and make some simple observations.

\subsection{Single, lit and unlit indices}
The purpose of the following definitions is to characterise the tuples $\d = (d_1, \ldots, d_{2\K})$ that contribute to the right-hand side of \cref{eq:traceinlemma}. We will show in \cref{lem:threeindicestypes} that every such tuple belongs to one of the sets $\DD_{\s,\lit,r}$ introduced in \cref{def:D}.

\begin{definition}[$\dsub{i}{j}$, $\d_I$, $\pall$]
    \label{def:dij}
    Let $\d = (d_i)_{i\in \intK}$ be a vector with coordinates $d_i \in \pm \D$.

    We write $\dsub{i}{j}$ for the unique prime in $\primes_j$ that divides $d_i$. Thus, $\abs{d_i} = \prod_{j\in \intJ} \dsub{i}{j}$.

    For any subset $I\subset \intK \times \intJ$, we set
    \begin{equation*}
        \d_I := \prod_{(i,j)\in I} \dsub{i}{j}.
    \end{equation*}
    In the special case $I = \intK \times \intJ$, we will write $\pall$ instead of $\d_{\intK \times \intJ}$ to shorten the notation.
\end{definition}


\begin{definition}[Single indices]
    \label{def:single}
    Let $\d = (d_i)_{i\in \intK}$ be a vector with coordinates $d_i \in \pm \D$.

    We say that a pair $(i,j)\in \intK \times \intJ$ is a \emph{single index} if $\dsub{i}{j}^2\nmid \pall$, i.e.~the prime $\dsub{i}{j}$ does not appear at any pair other than $(i,j)$. In that case, we also say that $\dsub{i}{j}$ is a \emph{single prime} of $\d$.
\end{definition}

\begin{definition}
    \label{def:admissible}
    Let $R \geq 1$ and let $\lit \subset \sinterval{R} \times \intJ$.

    A vector $\d \in (\pm \D)^{R}$ is called \emph{$\lit$-admissible} if, for all $k_1<k_2$ such that $\interval{k_1}{k_2}\times \intJ \subset \lit$, neither $(d_{k_1}, d_{k_1+1}, \ldots, d_{k_2})$ nor $(d_{k_2}, d_{k_2-1}, \ldots, d_{k_1})$ is a prohibited sequence (see \cref{def:prohibitedsequence}).\footnote{Recall that a prohibited sequence has length $\leq L$ by definition, so this condition is empty when $k_2-k_1+1 > L$.}
\end{definition}

We now define the sets $\DD_{\s,\lit,r}$. The tuples $\d = (d_1, \ldots, d_{2\K})$ arising from the trace expansion in \cref{lem:tracesimplebound} are non-reduced at exactly one position, since $d_{\K+1}=-d_{\K}$ while $d_{i+1}\neq -d_i$ for $i\neq \K$. This structure is a consequence of multiplying the $\K$-th power of the non-backtracking matrix by its conjugate transpose, a necessary step for controlling the eigenvalues. While this isolated backtracking requires technical adjustments (in parts \cref{item:Dprop3} and \cref{item:Dprop4} of \cref{def:D}), the resulting complexity is largely notational rather than conceptual.

\begin{definition}[$\DD_{\s,\lit,r}$]
    \label{def:D}
    Let $\s$ and $\lit$ be disjoint subsets of $\intK\times \intJ$ and let $1\leq r\leq \K$.

    We define $\DD_{\s,\lit,r}$ to be the set of all vectors $\d = (d_i)_{i\in \intK}$ with coordinates $d_i \in \pm \D$ having the following properties.
    \begin{enumerate}
        \item \label{item:Dprop1} The set $\s$ is precisely the set of single indices for $\d$.
        \item \label{item:Dprop2} Whenever two indices $(i,j), (i',j)\in \lit$ are such that $\dsub{i}{j} = \dsub{i'}{j}$, we have
              \begin{equation*}
                  \dsub{i}{j} \bigmid \sum_{i<k<i'}d_k.
              \end{equation*}
        \item \label{item:Dprop3} The integer $r$ satisfies $(d_{\K+1}, d_{\K+2}, \ldots d_{\K+r}) = -(d_{\K}, d_{\K-1}, \ldots, d_{\K-r+1})$.
        \item \label{item:Dprop4} Let
              \begin{equation*}
                  \begin{cases}
                      I^{(1)} := \interval{1}{\K}     \\
                      I^{(2)} := \interval{\K+1}{2\K} \\
                      I^{(3)} := \interval{1}{\K-r}\cup \interval{\K+r+1}{2\K},
                  \end{cases}
              \end{equation*}
              and let $\iota^{(k)} : I^{(k)} \to \{1, \ldots, |I^{(k)}|\}$ be the unique order-preserving bijection. Then, for $k=1,2,3$, the vector
              \begin{equation*}
                  \d^{(k)} := (d_i)_{i\in I^{(k)}}
              \end{equation*}
              is reduced\footnote{See \cref{def:reduced} for the definition of reduced tuples.} and $\lit^{(k)}$-admissible, where $\lit^{(k)} := \{(\iota^{(k)}(i), j) \mid (i,j)\in \lit\}$.
    \end{enumerate}
    Note that if $r=\K$, then $I^{(3)} = \emptyset$ and the conditions on $\d^{(3)}$ can be ignored.
\end{definition}

Our goal is to bound the expression in \cref{lem:tracesimplebound}, which includes a double sum over ${\d \in (\pm \D)^{2\K}}$ and $n\in\N$. \Cref{lem:threeindicestypes} below reduces this problem to analysing the contribution of the pairs $(\d, n)$ such that
\begin{itemize}
    \item $\d\in \DD_{\s,\lit,r}$,
    \item $\dsub{i}{j} \mid n + b_i$ for $(i,j)\in \lit$, and
    \item $\dsub{i}{j} \nmid n + b_i$ for $(i,j)\in \unlit$,
\end{itemize}
where $\s$, $\lit$, and $\unlit$ are sets such that $\s \sqcup \lit \sqcup \unlit = \intK \times \intJ$, and $1\leq r\leq \K$. Thus, $\s$ is the set of single indices for $\d$, and the complement $(\intK \times \intJ)\setminus \s$ is divided into two sets according to whether the corresponding divisibility condition holds or not. Following the terminology of Helfgott and Radziwi\l\l~\cite{HR}, we call $\lit$ the set of \emph{lit indices} and $\unlit$ the set of \emph{unlit indices}.

The contribution of these pairs $(\d, n)$ is captured by the quantity $E_{\s,\lit,\unlit,r}$ defined below.

\begin{definition}[$E_{\s,\lit,\unlit,r}$]
    \label{def:Esum}
    Let $\s$, $\lit$, and $\unlit$ be sets such that $\s \sqcup \lit \sqcup \unlit = \intK\times \intJ$ and let $1\leq r\leq \K$.

    We write $E_{\s,\lit,\unlit, r}$ for the quantity
    \begin{equation}
        \label{eq:defineinnersum}
        E_{\s,\lit,\unlit, r} := \sum_{\substack{\d\in \DD_{\s,\lit,r}}} \Bigg\lvert \sum_{\substack{n\in \N\\ n+b_i\in \Y\,\text{for }0\leq i\leq 2\K}} \ind{\substack{\dsub{i}{j}\mid n+\b_i \, \forall (i,j)\in \lit \\ \dsub{i}{j}\nmid n+\b_i \, \forall (i,j)\in \unlit}} \prod_{(i,j)\in \intK\times \intJ}\left(\ind{\dsub{i}{j}\mid n + b_i}-\frac{1}{\dsub{i}{j}}\right)\Bigg\rvert,
    \end{equation}
    where $b_i := \sum_{1\leq k\leq i} d_k$ for $0\leq i \leq 2\K$.
\end{definition}

\begin{lemma}
    \label{lem:threeindicestypes}
    We have
    \begin{equation*}
        \Tr\big[\big((M_{Y})^{\K}\big)^*(M_{Y})^{\K}\big] \ll e^{O(\J\K)}\max_{\substack{\s \sqcup \lit \sqcup \unlit = \intK\times \intJ\\ 1\leq r\leq \K}} E_{\s,\lit,\unlit, r}\  + \ N.
    \end{equation*}
\end{lemma}

\begin{proof}
    For any $\d \in (\pm \D)^{2\K}$ with $d_{\K+1}=-d_\K$, let $r(\d)$ be the largest integer $1\leq r\leq \K$ such that
    \begin{equation*}
        (d_{\K+1}, d_{\K+2}, \ldots d_{\K+r}) = -(d_{\K}, d_{\K-1}, \ldots, d_{\K-r+1}).
    \end{equation*}
    By \cref{lem:tracesimplebound} and the triangle inequality, we have
    \begin{equation}
        \label{eq:tracesimpleboundrestated}
        \Tr\big[\big((M_{Y})^{\K}\big)^*(M_{Y})^{\K}\big] \ll H\sum_{\substack{\s\sqcup \lit\sqcup \unlit = \intK\times\intJ\\ 1\leq r\leq \K}} \sum_{\substack{\d\in (\pm \D)^{2\K}\\ d_{i+1}=-d_i \iff i=\K\\ \dsub{i}{j}^2\nmid \pall \iff (i,j)\in \s\\ r(\d)=r }} \abs{S_{\s,\lit,\unlit}(\d)}\  +\  N,
    \end{equation}
    where $S_{\s,\lit,\unlit}(\d)$ denotes the inner sum on the right-hand side of \cref{eq:defineinnersum}, i.e.
    \begin{equation*}
        S_{\s,\lit,\unlit}(\d) := \sum_{\substack{n\in \N\\ n+b_i\in \Y\,\text{for }0\leq i\leq 2\K}} \ind{\substack{\dsub{i}{j}\mid n+\b_i \, \forall (i,j)\in \lit \\ \dsub{i}{j}\nmid n+\b_i \, \forall (i,j)\in \unlit}} \prod_{(i,j)\in \intK\times \intJ}\left(\ind{\dsub{i}{j}\mid n + b_i}-\frac{1}{\dsub{i}{j}}\right).
    \end{equation*}
    Up to a multiplicative factor $e^{O(\J\K)}$, the sum over $\s$, $\lit$, $\unlit$ and $r$ in \cref{eq:tracesimpleboundrestated} can be replaced by a maximum. The factor $H$ can be absorbed in this new $e^{O(\J\K)}$ term as $\K \asymp \log H$.

    Hence, to prove \cref{lem:threeindicestypes}, it suffices to show that any vector $\d \in (\pm \D)^{2\K}$ satisfying the properties
    \begin{enumerate}[label=(\roman*), ref=\roman*]
        \item \label{item:threeindicestypes1} $d_{i+1}=-d_i \iff i=\K$;
        \item \label{item:threeindicestypes2} $\dsub{i}{j}^2\nmid \pall \iff (i,j)\in \s$;
        \item \label{item:threeindicestypes3} $r(\d)=r$; and
        \item \label{item:threeindicestypes4} $S_{\s,\lit,\unlit}(\d) \neq 0$
    \end{enumerate}
    is contained in $\DD_{\s,\lit,r}$.

    Let $\d$ be such a vector. We need to show that it satisfies the four properties of \cref{def:D}. Some are immediate: property~\cref{item:Dprop1} of \cref{def:D} is equivalent to \cref{item:threeindicestypes2} and, by definition of $r(\d)$, property~\cref{item:Dprop3} of \cref{def:D} follows from \cref{item:threeindicestypes3}. In addition, \cref{item:threeindicestypes3} implies that
    \begin{equation}
        \label{eq:d3glue}
        d_{\K+r+1} \neq -d_{\K-r}
    \end{equation}
    in the case $r<\K$. From \cref{eq:d3glue} and \cref{item:threeindicestypes1}, we deduce that the vectors $\d^{(1)}$, $\d^{(2)}$ and $\d^{(3)}$ from \cref{def:D} are reduced.

    It only remains to prove property \cref{item:Dprop2} of \cref{def:D} and the $\lit^{(k)}$-admissibility of $\d^{(k)}$ for ${k=1,2,3}$. This is where assumption \cref{item:threeindicestypes4} is used.

    Let $(i,j), (i',j)\in \lit$ be such that $\dsub{i}{j} = \dsub{i'}{j}$, where we assume without loss of generality that $i < i'$. By \cref{item:threeindicestypes4}, there is at least one non-zero term in the sum defining $S_{\s,\lit,\unlit}(\d)$, which means that there is at least one $n\in \N$ such that $\dsub{i}{j}\mid n + b_i$ and $\dsub{i'}{j} \mid n + b_{i'}$. Since $\dsub{i}{j} = \dsub{i'}{j}$, we can subtract these two divisibility relations to get $\dsub{i}{j} \mid b_{i'} - b_i$. Recalling that $b_l := \sum_{1\leq k\leq l} d_k$, we get $\dsub{i}{j} \mid \sum_{i<k\leq i'} d_k$, which implies property \cref{item:Dprop2} of \cref{def:D} as $\dsub{i}{j} = \dsub{i'}{j}\mid d_{i'}$.

    Finally, suppose by contradiction that $\d^{(k)}$ is not $\lit^{(k)}$-admissible, for some $k\in \{1,2,3\}$. Let ${\tau : \{1, \ldots, |I^{(k)}|\} \to I^{(k)}}$ be the inverse of $\iota^{(k)}$. Then, there are $a, \ell\in \mathbb{N}$ and $\sigma \in \{\pm 1\}$ such that
    \begin{itemize}
        \item $\{\tau(a), \tau(a+\sigma), \ldots, \tau(a+\ell\sigma)\}\times \intJ \subset \lit$,\footnote{In particular, $\{a, a+\sigma, \ldots, a+\ell\sigma\} \subset \{1, \ldots, |I^{(k)}|\}$.} and
        \item $(d_{\tau(a)}, d_{\tau(a+\sigma)}, \ldots, d_{\tau(a+\ell \sigma)})$ is a prohibited sequence.
    \end{itemize}
    By passing to a suitable subsequence, we can assume without loss of generality that this prohibited sequence is primitive. By \cref{item:threeindicestypes4}, there exists $n\in \N$ such that the term in $S_{\s,\lit,\unlit}(\d)$ corresponding to $n$ is non-zero. In particular, $n+b_i \in \Y$ for all $0\leq i\leq 2\K$. Furthermore, we have $\dsub{i}{j} \mid n+b_i$ for all $(i,j)\in \lit$, and thus
    \begin{equation}
        \label{eq:taudivcond}
        d_{\tau(a)} \mid n+b_{\tau(a)}, \quad d_{\tau(a+\sigma)} \mid n+b_{\tau(a+\sigma)}, \quad \ldots, \quad d_{\tau(a+\ell\sigma)} \mid n+b_{\tau(a+\ell\sigma)}.
    \end{equation}
    For any $1<i\leq |I^{(k)}|$, we have
    \begin{equation}
        \label{eq:bidiff}
        b_{\tau(i)} - b_{\tau(i-1)} = d_{\tau(i)}.
    \end{equation}
    While this is obvious for $k\in \{1,2\}$, it is also true for $k=3$ because $b_{\K-r}=b_{\K+r}$ by property~\cref{item:Dprop3} of \cref{def:D}, which we have already shown.

    In the case $\sigma = 1$, using \cref{eq:bidiff}, we can rewrite \cref{eq:taudivcond} as
    \begin{equation*}
        \begin{cases}
            d_{\tau(a)} \mid n+b_{\tau(a-1)}                             \\
            d_{\tau(a+1)} \mid n+b_{\tau(a-1)}+d_{\tau(a)}               \\
            d_{\tau(a+2)} \mid n+b_{\tau(a-1)}+d_{\tau(a)}+d_{\tau(a+1)} \\
            \vdots                                                       \\
            d_{\tau(a+\ell)} \mid n+b_{\tau(a-1)}+d_{\tau(a)}+d_{\tau(a+1)}+\ldots +d_{\tau(a+\ell-1)}.
        \end{cases}
    \end{equation*}
    Since $(d_{\tau(a)}, d_{\tau(a+1)}, \ldots, d_{\tau(a+\ell)})$ is a primitive prohibited sequence, this means that $n+b_{\tau(a-1)}$ belongs to a prohibited progression (see \cref{def:prohibitedprogression}), contradicting the fact that ${n+b_{\tau(a-1)}\in \Y}$.

    Likewise, in the case $\sigma = -1$, we can use \cref{eq:bidiff} to rewrite \cref{eq:taudivcond} as
    \begin{equation*}
        \begin{cases}
            d_{\tau(a)} \mid n+b_{\tau(a)}                             \\
            d_{\tau(a-1)} \mid n+b_{\tau(a)}-d_{\tau(a)}               \\
            d_{\tau(a-2)} \mid n+b_{\tau(a)}-d_{\tau(a)}-d_{\tau(a-1)} \\
            \vdots                                                     \\
            d_{\tau(a-\ell)} \mid n+b_{\tau(a)}-d_{\tau(a)}-d_{\tau(a-1)}-\ldots -d_{\tau(a-\ell+1)}.
        \end{cases}
    \end{equation*}
    Since $(d_{\tau(a)}, d_{\tau(a-1)}, \ldots, d_{\tau(a-\ell)})$ is a primitive prohibited sequence, so is its negative. Similarly to the previous case, these divisibilities imply that $n+b_{\tau(a)}$ belongs to a prohibited progression, contradicting the fact that ${n+b_{\tau(a)}\in \Y}$. This concludes the proof that $\d^{(k)}$ is $\lit^{(k)}$-admissible.
\end{proof}

\subsection{Walks with many unlit indices}
\label{sec:unlit}
The next lemma shows that the expression $E_{\s,\lit,\unlit, r}$ from \cref{def:Esum} is negligible once the number of unlit indices exceeds a modest threshold, using that the primes in $\primes$ are larger than $H_0$.
\begin{lemma}
    \label{lem:unlit}
    Let $\s$, $\lit$, $\unlit$ be sets such that $\intK\times\intJ = \s\sqcup\lit\sqcup\unlit$ and let $1\leq r\leq \K$. Suppose that $|\unlit|\geq \K^{2\epsone}$. Then
    \begin{equation*}
        E_{\s,\lit,\unlit, r} \ll N.
    \end{equation*}
\end{lemma}

To prove \cref{lem:unlit}, we need the following basic estimate.

\begin{lemma}
    \label{lem:trivialbound}
    We have
    \begin{equation*}
        \sum_{\d\in (\pm \D)^{2\K}}  \prod_{p\mid \pall}\frac{1}{p} \ll\K^{O(\J\K)}.
    \end{equation*}
\end{lemma}
\begin{proof}
    Any $\d\in (\pm \D)^{2\K}$ induces a partition of $\intK\times\intJ$, where $(i,j)$ and $(i',j')$ are in the same class if and only if $\dsub{i}{j} = \dsub{i'}{j'}$. Since the sets $\primes_j$ are disjoint, this partition has the property that every class $\alpha$ is contained in $\intK\times \{j_{\alpha}\}$ for some $j_{\alpha}\in \intJ$.

    Observe that $\d$ can be fully reconstructed from a sequence of $2\K$ signs (the signs of the $d_i$), a partition of $\intK\times\intJ$ with the above property, and the assignment of a prime $p_{\alpha}\in \primes_{\! j_{\alpha}}$ to every class $\alpha$ of this partition. Recall that ${\sum_{p\in \primes_j}1/p = \V_j\leq \V}$ for every $j$. Summing over all sequences of signs $\sigma$, suitable partitions $\Pi$ of $\intK\times\intJ$ and primes $(p_{\alpha})_{\alpha\in \Pi}$, we get
    \begin{equation*}
        \sum_{\d\in (\pm \D)^{2\K}} \prod_{p\mid \pall} \frac{1}{p} \leq  \sum_{\sigma\in \{\pm1\}^{2\K}} \sum_{\Pi} \prod_{\alpha\in \Pi} \,\sum_{p_{\alpha}\in \primes_{\!j_{\alpha}}} \frac{1}{p_{\alpha}} \leq 2^{2\K} (2\J\K)^{2\J\K } \V^{2\J\K} \leq \K^{O(\J\K)},
    \end{equation*}
    where we used that the number of partitions of a finite set $S$ is $\leq |S|^{|S|}$ and that $\V, \J \ll \log K$.
\end{proof}

\begin{proof}[Proof of \cref{lem:unlit}]
    By the triangle inequality,
    \begin{align*}
        E_{\s,\lit,\unlit, r} & \leq \sum_{\s_1\sqcup \s_2 = \s}\, \sum_{\substack{\d\in \DD_{\s,\lit,r}}}\, \sum_{\substack{n\in \N}}
        \bigg(\prod_{(i,j)\in \lit \sqcup \s_1} \ind{\dsub{i}{j}\mid n+\b_i}\bigg)\bigg( \prod_{(i,j)\in \unlit\sqcup \s_2}\frac{1}{\dsub{i}{j}}\bigg) \\ &\ll N\sum_{\s_1\sqcup \s_2 = \s}\, \sum_{\substack{\d\in \DD_{\s,\lit,r}}} \bigg( \prod_{p\mid \d_{\lit\sqcup \s_1}} \frac{1}{p}\bigg)\bigg( \prod_{(i,j)\in \unlit\sqcup \s_2}\frac{1}{\dsub{i}{j}}\bigg),
    \end{align*}
    using that the arithmetic progression defined by the divisibility conditions on $n$ has modulus $\prod_{p\mid \d_{\lit \sqcup \s_1}} p \leq H^{\J\K}$, which is $N^{o(1)}$ by \cref{eq:boundHJK}.

    We now examine the number of occurrences of the primes $p\mid \pall$ in the two products
    \begin{equation*}
        P_1 := \prod_{p\mid \d_{\lit\sqcup \s_1}} \frac{1}{p} \quad \text{and} \quad P_2 := \prod_{(i,j)\in \unlit\sqcup \s_2}\frac{1}{\dsub{i}{j}}.
    \end{equation*}
    Every prime $p\mid \d_{\s}$ appears exactly once, either in $P_1$ or in $P_2$. Every prime $p\mid \d_{\lit}$ with $p \nmid \d_{\unlit}$ appears exactly once in $P_1$ and does not appear in $P_2$. If instead $p\mid \d_{\unlit}$ and $p \nmid \d_{\lit}$, then $p$ does not appear in $P_1$ but appears $\nu_p(\d_{\unlit})\geq 2$ times in $P_2$. Finally, every prime $p$ dividing both $\d_{\lit}$ and $\d_{\unlit}$ appears exactly once in $P_1$ and $\nu_p(\d_{\unlit})\geq 1$ times in $P_2$. Therefore,
    \begin{equation*}
        P_1P_2 = \prod_{p\mid \pall} \frac{1}{p}\cdot \prod_{p\mid \d_{\unlit}}\frac{1}{p^{\nu_p(\d_{\unlit})-\ind{p\nmid \d_{\lit}}}}\leq \prod_{p\mid \pall} \frac{1}{p}\cdot \prod_{p\mid \d_{\unlit}}\frac{1}{p^{\nu_p(\d_{\unlit})/2}} \leq H_0^{-\abs{\unlit}/2}\prod_{p\mid \pall} \frac{1}{p}.
    \end{equation*}
    Hence, by \cref{lem:trivialbound},
    \begin{equation*}
        E_{\s,\lit,\unlit, r} \ll N  2^{\abs{\s}} H_0^{-\abs{\unlit}/2} \sum_{\d\in (\pm \D)^{2\K}} \prod_{p\mid \pall} \frac{1}{p}\ll N \K^{O(\J\K)} H_0^{-\abs{\unlit}/2}.
    \end{equation*}
    Recall that ${\log H_0 \gg\K^{1-\epsone}}$, while ${|\unlit|\geq\K^{2\epsone}}$ by assumption. Thus
    \begin{equation*}
        \K^{O(\J\K)} H_0^{-\abs{\unlit}} \ll \K^{O(\K\log \K)} e^{-K^{1+\epsone}} \ll 1,
    \end{equation*}
    and the bound $E_{\s,\lit,\unlit, r} \ll N$ follows.
\end{proof}

\section{Cancellation from single primes}
\label{sec:single}

In \cref{lem:unlit}, we proved that $E_{\s,\lit,\unlit, r}$ is small whenever the number of unlit indices is moderately large. In this section, we will show that the same is true when the number of single indices is large.

\begin{proposition}
    \label{prop:cancellationoverY}
    Let $\s$, $\lit$, $\unlit$ be sets such that $\intK\times\intJ = \s\sqcup\lit\sqcup\unlit$ and let $1\leq r\leq \K$.

    Suppose that $|\s|\geq \K^{1-\epsone}$ and $|\unlit| \leq\K^{2\epsone}$. Then
    \begin{equation*}
        E_{\s,\lit,\unlit, r} \ll e^{O(\J\K)} N.
    \end{equation*}
\end{proposition}

To prove \cref{prop:cancellationoverY}, we need to exhibit cancellation in a sum ranging over the intersection of $\N$ with shifted copies of $\Y$. If we did not have to deal with $\Y$, this would be a simple task. However, the presence of $\Y$ requires us to develop a more sophisticated approach. We will use a combinatorial sieve to approximate the indicator of $\Y$ by a linear combination of arithmetic progressions, which are easier to handle.

\subsection{Cancellation over arithmetic progressions}

We first prove cancellation over arithmetic progressions in the presence of many single indices.
\begin{lemma}
    \label{lem:APcancellation}
    Let $\s$, $\lit$, $\unlit$ be sets such that $\intK\times\intJ = \s\sqcup\lit\sqcup\unlit$. Let $1\leq r\leq \K$ and $\d\in \DD_{\s,\lit,r}$.

    Let $R\subset \Z$ be an arithmetic progression with square-free modulus $q_R$ not divisible by $\d_{\s}$. Then

    \begin{equation}
        \label{eq:sumoverAP}
        \sum_{\substack{n\in \N\cap R}} \ind{\substack{\dsub{i}{j}\mid n+\b_i \, \forall (i,j)\in \lit \\ \dsub{i}{j}\nmid n+\b_i \, \forall (i,j)\in \unlit}} \prod_{(i,j)\in \intK\times \intJ}\left(\ind{\dsub{i}{j}\mid n + b_i}-\frac{1}{\dsub{i}{j}}\right) \ll H^{2\K}q_R.
    \end{equation}
\end{lemma}
\begin{proof}
    Since $\d_{\s}$ does not divide $q_R$, there exists an index $(i_0, j_0)\in \s$ such that $p_0 := \dsub{i_0}{j_0}\nmid q_R$.

    Introducing the variable
    \begin{equation*}
        Q := \prod_{\substack{p\mid \pall q_R\\ p\neq p_0}} p,
    \end{equation*}
    we may express the sum \cref{eq:sumoverAP} as
    \begin{equation*}
        \sum_{a\spmod Q}  c(a) \sum_{\substack{n\in \N\\ n\equiv a \spmod{Q}}} \left(\ind{p_0\mid n + b_{i_0}}-\frac{1}{p_0}\right)
    \end{equation*}
    for some $1$-bounded coefficients $c(a)$. Since $p_0\nmid Q$, the inner sum exhibits near-perfect cancellation:
    \begin{equation*}
        \sum_{\substack{n\in \N\\ n\equiv a \spmod{Q}}} \left(\ind{p_0\mid n + b_{i_0}}-\frac{1}{p_0}\right) = \frac{N}{p_0Q} + O(1) - \frac{1}{p_0} \left(\frac{N}{Q}+O(1)\right) \ll 1.
    \end{equation*}
    Thus, the left-hand side of \cref{eq:sumoverAP} is bounded by $\ll Q \leq \pall q_R \leq H^{2\K}q_R$, as desired.
\end{proof}

\subsection{Cancellation over \texorpdfstring{$\Y$}{Y}}
Recall that $\prohibprog$ is the set of all prohibited progressions, and $\Y$ is the complement of the union of these prohibited progressions. We need to use a suitable version of the inclusion-exclusion principle to express $\ind{\Y}$ as a linear combination of indicators of intersections of prohibited progressions.

The exact inclusion-exclusion formula
\begin{equation}
    \label{eq:inclexcl}
    \ind{n\in \Y} = \ind{n\not\in P\ \forall P\in \prohibprog} = 1 - \sum_{P_1\in \prohibprog} \ind{n\in P_1} + \sum_{\substack{P_1, P_2\in \prohibprog\\ \text{distinct}}} \ind{n\in P_1 \cap P_2} - \sum_{\substack{P_1, P_2, P_3\in \prohibprog\\ \text{distinct}}} \ind{n\in P_1 \cap P_2\cap P_3} + \cdots.
\end{equation}
has too many terms to be useful. We require a truncated version, also known as a combinatorial sieve. The combinatorial sieve we will use was developed by Helfgott and Radziwiłł \cite{HR}, using ideas from the theory of the Möbius function of partially ordered sets. Its two main features are the following.
\begin{itemize}
    \item Because the progressions $P\in \prohibprog$ have composite (square-free) moduli, several intersections of progressions in $\prohibprog$ can yield the same result. For example,
          \begin{equation*}
              5\Z \cap 6\Z \cap 7\Z= 14\Z \cap 30\Z= 2\Z \cap 6\Z \cap 15\Z \cap 21\Z.
          \end{equation*}
          Let $R$ be a progression. In the right-hand side of \cref{eq:inclexcl}, all of the terms $\pm \ind{n\in P_1\cap \ldots \cap P_i}$ with $i\geq 1$ and $P_1\cap \ldots \cap P_i = R$ can be combined, and simplify to $c_R \ind{n\in R}$ for some integer coefficient $c_R$. However, if the modulus $q_{R}$ of $R$ has $k$ prime factors, there can be close to $2^{2^k}$ ways of expressing $R$ as an intersection of distinct arithmetic progressions. The most naive bound would thus give $\abs{c_R} \leq 2^{2^k}$. This is much larger than what we can allow. Fortunately, the combinatorial interpretation\footnote{In combinatorial language, $c_R$ is a value of the Möbius function of the partially ordered set consisting of all possible intersections of prohibited progressions.} of this coefficient $c_R$ shows that there is a vast amount of cancellation from the $\pm1$ signs, and the much more reasonable bound $\abs{c_R} \leq 2^k$ holds.\footnote{Optimal bounds for $c_R$ are due to Sagan, Yeh and Ziegler (see \cite[after Corollary~2.5]{mobius}). Helfgott and Radziwiłł~\cite{HR} gave a one-line proof of the slightly weaker bound $\abs{c_R} \leq 2^k$ (see \cref{lem:rota}).}
    \item A classical way to approximate the inclusion-exclusion formula is by means of the \emph{Bonferroni inequalities}. These imply that, for any $t\geq 1$,
          \begin{equation*}
              \ind{n\in \Y} = \sum_{i=0}^{t-1} \sum_{\substack{P_1, \ldots, P_i \in \prohibprog\\ \text{distinct}}} (-1)^{i} \ind{n\in P_1\cap \ldots \cap P_i} + O\Bigg(\sum_{\substack{P_1, \ldots, P_t \in \prohibprog\\ \text{distinct}}} \ind{n\in P_1\cap \ldots \cap P_t}\Bigg).
          \end{equation*}
          In this simple version, the terms $(-1)^{i} \ind{n\in P_1\cap \ldots \cap P_i}$ with $i<t$ are kept in the main term, those with $i=t$ constitute the remainder term and the others can be discarded. We require a more flexible truncation method, not just based on the number $i$ of sets in the intersection, but on specific properties of the progression $P_1\cap \ldots \cap P_i$. For Helfgott and Radziwiłł \cite{HR}, this cut-off was determined by the number of prime factors of the modulus of the intersection $P_1\cap \ldots \cap P_i$. In this paper, the truncation and its analysis are significantly more technical.
\end{itemize}

The combinatorial sieve of Helfgott and Radziwiłł is stated in \cref{prop:sievesquare-free} for a general cut-off. We provide a self-contained proof of it in \cref{appendix:sieve} (a shortened version of that in \cite{HR}). We now apply it to rewrite the term $\ind{\forall i,\,n+\b_i \in \Y}$ as a suitable combination of arithmetic progressions.

\begin{notation}
    \label{def:Qintersect}
    Let $\s$, $\lit$, $\unlit$ be sets such that $\intK\times\intJ = \s\sqcup\lit\sqcup\unlit$ and let $1\leq r\leq \K$. Let~$\d\in \DD_{\s,\lit,r}$ and let $\bm{b} := (b_i)_{0\leq i\leq 2\K}$ be the associated vector of partial sums $b_i := \sum_{1\leq k\leq i} d_k$. We write
    \begin{equation*}
        {\shiftedprog := \{P-\b_i : P \in \prohibprog, \,0\leq i\leq 2\K\}}.
    \end{equation*}
    We also define
    \begin{equation*}
        \interprog := \bigg\{\bigcap_{P\in X}P : X \subset \shiftedprog\bigg\},
    \end{equation*}
    the set of all possible intersections of such shifted progressions (with the convention $\bigcap_{P\in \emptyset}P := \Z$).
\end{notation}

The next lemma captures our application of the combinatorial sieve. It is rather technical, and we defer its proof to \cref{sec:cutoffanalysis}. The statement of \cref{lem:existencerank} can be understood as follows. In \cref{item:ranksieve}, the approximate inclusion-exclusion formula is given, with a main term and a remainder term. The main term is a sum over all progressions with small \emph{rank}. The rank of a progression can be thought of as a measure of its complexity. It is a quantity depending on $\d$, but its precise definition is not immediately needed and hence will only be given later, in \cref{def:rank}. Two simple properties of the rank are given in \cref{item:rank1} and \cref{item:rank2}. Finally, the bound \cref{item:rank3} will be used to control the remainder term in the combinatorial sieve. The proof of \cref{item:rank3} is particularly intricate.
\begin{restatable}{lemma}{existencerank}
    \label{lem:existencerank}
    Let $\s$, $\lit$, $\unlit$ be sets such that $\intK\times\intJ = \s\sqcup\lit\sqcup\unlit$ and let $1\leq r\leq \K$.

    For every $\d \in \DD_{\s,\lit,r}$, there exists a function
    \begin{equation*}
        \rank : \interprog \to \Z^{\geq 0} \cup \{+\infty\}
    \end{equation*}
    satisfying the conditions below.

    Define the arithmetic progression $\progression := \{n\in \Z : \forall (i,j)\in \lit, \ \dsub{i}{j} \mid n+\b_i\}$.

    Let~$\Xd$ be the set of all $R\in \interprog$ such that ${\rank(R) < \K^{5\epsone}}$. Let $\partial \Xd$ be the set of all $R\in \interprog\setminus \Xd$ of the form $R = R'\cap P$ for some $R'\in \Xd$ and $P\in \shiftedprog$.

    Then, the following properties hold.
    \begin{enumerate}
        \item \label{item:rank1} (Primes dividing the modulus and $\d_{\s}$) For every $R\in \interprog$,
              \begin{equation*}
                  \abs{ \{p: p\mid (q_R, \d_{\s})\}}\leq \J\L  \, \rank(R).
              \end{equation*}
        \item \label{item:rank2} (Primes dividing the modulus) For every $R\in \interprog$,
              \begin{equation*}
                  \omega(q_{R}) \ll \J\L  \, \rank(R) + \J\K.
              \end{equation*}
        \item \label{item:ranksieve} (Combinatorial sieve) For all $n\in \Z$, we have
              \begin{equation*}
                  \ind{\forall i,\,n+\b_i \in \Y}\ind{\progression}(n)  = \sum_{\substack{R\in \Xd}} c_{R,\d}  \ind{n\in R\cap \progression}\,  +\,  O\Bigg(  e^{O(\J\K)}  \!\!\! \sum_{\substack{R\in \partial\Xd \\ R\cap \progression \neq \emptyset}} \!\!\! \ind{n\in R\cap \progression} \Bigg),
              \end{equation*}
              where the coefficients $c_{R,\d}$ are independent of $n$ and satisfy $|c_{R,\d}| \leq e^{O(\J\K)}$.
        \item \label{item:rank3} (Remainder term bound) Suppose that $|\unlit| \leq\K^{2\epsone}$. Then
              \begin{equation*}
                  \sum_{\d\in \DD_{\s,\lit,r}} \sum_{\substack{R\in \partial\Xd \\ R\cap \progression \neq \emptyset}}\, \prod_{\substack{p\mid q_{R} \pall}}\frac{1}{p} \ll 1.
              \end{equation*}
    \end{enumerate}
\end{restatable}

We now have all the necessary ingredients to prove \cref{prop:cancellationoverY}.

\begin{proof}[Proof of \cref{prop:cancellationoverY}, assuming \cref{lem:existencerank}]
    We use the combinatorial sieve to rewrite the term involving $\Y$ in terms of arithmetic progressions: by part \cref{item:ranksieve} of \cref{lem:existencerank} and the triangle inequality, we have
    \begin{equation}
        \label{eq:cancellationEsum}
        E_{\s,\lit,\unlit, r} \leq e^{O(\J\K)} \left(F_1+F_2\right)
    \end{equation}
    where
    \begin{equation*}
        F_1 := \sum_{\d\in \DD_{\s,\lit,r}} \sum_{R\in \Xd} \Bigg\lvert \sum_{\substack{n\in \N\cap R}} \ind{\substack{\dsub{i}{j}\mid n+\b_i \, \forall (i,j)\in \lit \\ \dsub{i}{j}\nmid n+\b_i \, \forall (i,j)\in \unlit}} \prod_{(i,j)\in \intK\times \intJ}\left(\ind{\dsub{i}{j}\mid n + b_i}-\frac{1}{\dsub{i}{j}}\right)\Bigg\rvert
    \end{equation*}
    and
    \begin{equation*}
        F_2 := \sum_{\substack{\d\in \DD_{\s,\lit,r}}} \sum_{\substack{R\in \partial\Xd\\ R\cap \progression \neq \emptyset}} \sum_{\substack{n\in \N\cap R}} \ind{\substack{\dsub{i}{j}\mid n+\b_i \, \forall (i,j)\in \lit \\ \dsub{i}{j}\nmid n+\b_i \, \forall (i,j)\in \unlit}} \prod_{(i,j)\in \intK\times \intJ}\abs{\ind{\dsub{i}{j}\mid n + b_i}-\frac{1}{\dsub{i}{j}}}.
    \end{equation*}

    We wish to apply \cref{lem:APcancellation} to bound the innermost sum in $F_1$.  To do so, we must check that the modulus of the arithmetic progression $R$ is square-free and is not divisible by $\d_{\s}$. The first condition is immediate since $R\in \interprog$ and all prohibited progressions have square-free moduli. To verify the second condition, observe that the number of primes dividing both $q_R$ and $\d_{\s}$ is bounded by
    \begin{equation*}
        \leq \J\L  \, \rank(R) \leq \J\K^{1-10\epsone} \K^{5\epsone} \ll \K^{1-4\epsone},
    \end{equation*}
    using part \cref{item:rank1} of \cref{lem:existencerank} and the definition of $\Xd$. On the other hand, $\d_{\s}$ has $\abs{\s} \geq \K^{1-\epsone}$ distinct prime factors by assumption. Since $K$ is sufficiently large, these bounds ensure that $\d_{\s} \nmid q_R$.

    Therefore, by \cref{lem:APcancellation}, for every $R \in \Xd$ we have
    \begin{equation*}
        \sum_{\substack{n\in \N\cap R}} \ind{\substack{\dsub{i}{j}\mid n+\b_i \, \forall (i,j)\in \lit \\ \dsub{i}{j}\nmid n+\b_i \, \forall (i,j)\in \unlit}} \prod_{(i,j)\in \intK\times \intJ}\left(\ind{\dsub{i}{j}\mid n + b_i}-\frac{1}{\dsub{i}{j}}\right) \ll H^{2\K} q_R.
    \end{equation*}
    All primes dividing $q_R$ are in $\primes$, so by part \cref{item:rank2} of \cref{lem:existencerank} and the definition of $\Xd$, we have
    \begin{equation}
        \label{eq:boundqR}
        q_R \leq H^{\omega(q_R)} \leq H^{O(\J\L  \, \rank(R)+\J\K)}\ll H^{O(\J\K)}.
    \end{equation}
    To complete the estimation of $F_1$, we only need a crude bound for the sums over $\d\in \DD_{\s,\lit,r}$ and $R\in \Xd$. Since an arithmetic progression $R$ is uniquely determined by its modulus $q_R$ and a residue class modulo ${q_R}$, by \cref{eq:boundqR}, there are $\ll H^{O(\J\K)}$ choices for $R\in \Xd$. Moreover, there are at most $(2H)^{2\K}$ choices for $\d\in \DD_{\s,\lit,r}$. We thus have
    \begin{equation}
        \label{eq:cancellationboundS1}
        F_1 \ll H^{O(\J\K)}.
    \end{equation}

    We turn to the second term $F_2$. By the triangle inequality, we have
    \begin{equation}
        \label{eq:firstboundF2}
        F_2 \leq \sum_{\s_1\sqcup \s_2 = \s}\, \sum_{\d\in \DD_{\s,\lit,r}} \sum_{\substack{R\in \partial\Xd\\ R\cap \progression \neq \emptyset}} \sum_{\substack{n\in \N\cap R}} \bigg(\prod_{(i,j)\in \lit \sqcup \s_1} \ind{\dsub{i}{j}\mid n+\b_i}\bigg)\bigg( \prod_{(i,j)\in \unlit\sqcup \s_2}\frac{1}{\dsub{i}{j}}\bigg).
    \end{equation}

    For each $R\in \partial\Xd$, by definition of $\partial\Xd$, there are progressions $R'\in \Xd$ and $P\in \shiftedprog$ such that $R = R'\cap P$. This implies that $q_R \leq q_{R'} q_P$. Note that $\omega(q_P) \leq \J\L$ by definition of a prohibited progression, so that $q_P \leq H^{\J\L}$. In addition, we have $q_{R'} \ll H^{O(\J\K)}$ by \cref{eq:boundqR}. This proves that
    \begin{equation}
        \label{eq:boundqR2}
        q_R \ll H^{O(\J\K)}
    \end{equation}
    for all $R\in \partial\Xd$.

    The estimate \cref{eq:boundqR2}, the simple bound $\pall \leq H^{2\K}$ and \cref{eq:boundHJK} imply that
    \begin{equation}
        \label{eq:Ycancellationinterbound}
        \sum_{\substack{n\in \N\cap R}}\, \prod_{(i,j)\in \lit \sqcup \s_1} \ind{\dsub{i}{j}\mid n+\b_i} \ll N \prod_{p\mid q_{R} \d_{\lit \sqcup \s_1}}\frac{1}{p}.
    \end{equation}
    Substituting \cref{eq:Ycancellationinterbound} into \cref{eq:firstboundF2}, we obtain
    \begin{equation*}
        F_2 \ll 2^{\abs{S}} N \sum_{\d \in \DD_{\s,\lit,r}} \sum_{\substack{R\in \partial\Xd\\ R\cap \progression \neq \emptyset}}\, \prod_{p\mid q_{R} \pall }\frac{1}{p}.
    \end{equation*}
    Here, we recognise the expression appearing in part \cref{item:rank3} of \cref{lem:existencerank}. Applying this lemma, we get
    \begin{equation}
        \label{eq:cancellationboundS2}
        F_2\ll e^{O(\J\K)}N.
    \end{equation}

    Finally, substituting \cref{eq:cancellationboundS1} and \cref{eq:cancellationboundS2} into \cref{eq:cancellationEsum}, we obtain the claimed bound $E_{\s,\lit,\unlit, r} \ll e^{O(\J\K)}N$.
\end{proof}

\section{Predictable walks}
\label{sec:predictable}

By \cref{lem:unlit,prop:cancellationoverY}, it remains to bound the expression $E_{\s,\lit,\unlit, r}$ in the case where there are few single indices and very few unlit indices. In this regime, the inner sum appearing in the definition of $E_{\s,\lit,\unlit, r}$ yields no meaningful cancellation, so we begin by applying the triangle inequality.

\begin{lemma}
    \label{lem:trivboundE}
    Let $\s$, $\lit$, $\unlit$ be sets such that $\intK\times\intJ = \s\sqcup\lit\sqcup\unlit$ and let $1\leq r\leq \K$. Then
    \begin{equation*}
        E_{\s,\lit,\unlit, r} \ll e^{O(\J\K)} N \sum_{\d \in \DD_{\s,\lit,r}}\, \prod_{p\mid \pall} \frac{1}{p}.
    \end{equation*}
\end{lemma}
\begin{proof}
    By definition of $E_{\s,\lit,\unlit, r}$ and the triangle inequality,
    \begin{equation*}
        E_{\s,\lit,\unlit, r} \leq \sum_{\s_1\sqcup \s_2 = \s}\, \sum_{\d\in \DD_{\s,\lit,r}} \sum_{n\in \N}
        \bigg(\prod_{(i,j)\in \lit \sqcup \s_1} \ind{\dsub{i}{j}\mid n+\b_i}\bigg)\bigg( \prod_{(i,j)\in \unlit\sqcup \s_2}\frac{1}{\dsub{i}{j}}\bigg).
    \end{equation*}
    Since $\pall \leq H^{2\K} \ll N$ by \cref{eq:boundHJK}, we have
    \begin{equation*}
        \sum_{n\in \N}
        \prod_{(i,j)\in \lit \sqcup \s_1} \ind{\dsub{i}{j}\mid n+\b_i} \ll N \prod_{p\mid \d_{\lit \sqcup \s_1}} \frac{1}{p}.
    \end{equation*}
    Simplifying gives the lemma.
\end{proof}

By \cref{lem:trivboundE}, our task is reduced to giving a good bound for
\begin{equation}
    \label{eq:sumoverD}
    \sum_{\substack{\d\in \DD_{\s,\lit,r}}}\,\prod_{p\mid \pall}\, \frac{1}{p}.
\end{equation}
A similar sum over the larger domain $\d \in (\pm \D)^{2\K}$ was analysed in \cref{lem:trivialbound}, where we obtained an optimal bound of order $\K^{O(\J\K)}$. However, such a bound is far too large for our purposes here, as we require a bound of the form $e^{O(\J\K)}\V^{\J\K}$. To achieve this, we must exploit the structure of the set $\DD_{\s,\lit,r}$. This marks the beginning of a purely combinatorial phase of the argument, where we analyse how the conditions defining $\DD_{\s,\lit,r}$ limit the complexity of the sum.

In the present setting, most indices are lit. The central condition is property \cref{item:Dprop2} of \cref{def:D}, which imposes a divisibility relation for every pair of lit indices corresponding to the same prime. Obtaining sufficient control over the dependencies among these divisibility conditions requires a detailed combinatorial analysis.

\subsection{Predictable words}
For our combinatorial work, we adopt the language of words and letters. The general statements we prove in this setting will later be applied to sequences of primes from a fixed set~$\primes_j$, to track repeated occurrences of the same prime along a walk.

\begin{definition}
    \label{def:words}
    Let $\mathcal{A}$ be a finite set (the alphabet). Let $\word_n$ be the set of all $n$-letter words on $\mathcal{A}$, where no two consecutive letters are the same. Let $\word_n^{\neq} \subset \word_n$ be the set of all $n$-letter words on $\mathcal{A}$ with distinct letters. Let $\word = \bigcup_{n\geq 1}\word_n$ and $\word^{\neq} = \bigcup_{n\geq 1}\word_n^{\neq}$.

    For $w\in \word_n$ and $1\leq k\leq n$, we write $w[k]$ for the $k$-th letter of $w$. We denote by $\letters{w}$ the set of all letters of $w$.

    We denote the set of all positions of the letter \upletters{A} in $w$ by $\pos{w}{\upletters{A}} := \{k\in \sinterval{n} : w[k] = \upletters{A}\}$. For~$l\in \sinterval{n}$, we also write $\posbis{w}{l} :=  \{k\in \sinterval{n} : w[k] = w[l]\}$ (instead of `$\posbis{w}{w[l]}$').

    The notation $v \sqsubset w$ means that $v$ is a \emph{substring} of $w$, i.e.~a sequence of consecutive letters of $w$.

    We write $\overleftarrow{w}$ for the word obtained by writing the letters of $w$ in the reversed order.

    The \emph{concatenation} of two words $w_1$ and $w_2$ is the word obtained by appending the letters of $w_2$ at the end of $w_1$. We denote it by $w_1w_2$.
\end{definition}

We now introduce a measure of the amount of structure in a word. We do so by counting the number of letters with \emph{constant neighbours}. These are letters for which every occurrence is always surrounded by the same set of letters. When most letters in a word have constant neighbours, their repetition patterns are jointly well-behaved.

\begin{definition}
    \label{def:cstneighb}
    Let $w\in \word$ and $\upletters{A}\in \mathcal{A}$. We say that $\upletters{A}$ has \emph{variable neighbours in $w$} if there exist two occurrences of $\upletters{A}$ in $w$, such that the letters immediately adjacent to $\upletters{A}$ at these two positions form two distinct sets. Otherwise, $\upletters{A}$ is said to have \emph{constant neighbours in $w$}.
\end{definition}

For example,
\begin{center}
    \begin{tabular}{c|c|c}
        $w$                     & set of neighbours for every occurrence of \upletters{A} in $w$  & neighbours of \upletters{A} in $w$ \\ \hline
        \upletters{XAYZYAXAY}   & \{\upletters{X,Y}\},  \{\upletters{X,Y}\},  \{\upletters{X,Y}\} & constant                           \\
        \upletters{AXYXAXZY}    & \{\upletters{X}\}, \{\upletters{X}\}                            & constant                           \\
        \upletters{XAYZYAYZXAY} & \{\upletters{X,Y}\}, \{\upletters{Y}\}, \{\upletters{X,Y}\}     & variable                           \\
        \upletters{YAXYZAXA}    & \{\upletters{X,Y}\}, \{\upletters{X,Z}\}, \{\upletters{X}\}     & variable
    \end{tabular}.
\end{center}

\begin{definition}
    \label{def:tpredictable}
    A word $w\in \word$ is said to be \emph{$t$-predictable} if the following conditions both hold.
    \begin{enumerate}
        \item Every letter appears $\leq t$ times in $w$.
        \item There are $\leq t$ letters with variable neighbours in $w$.
    \end{enumerate}
    Otherwise $w$ is called \emph{$t$-unpredictable}.
\end{definition}

\subsection{Counting predictable words}
We show in \cref{lem:countpredict} that, up to letter relabelling, there are $\leq n^{O(t^3)}$ words of length $n$ which are $t$-predictable. For moderate values of $t$, this improves on the trivial bound given by the number of partitions of $\sinterval{n}$, which is $n^{(1+o(1))n}$.

For this section, we could have used the language of partitions since our primary focus is on the positions of the letters, and not the letters themselves. However, we found it more convenient to use words for \cref{sec:unpredictable}, so we will use them here as well.

\begin{definition}
    \label{def:wordequivalence}
    Two words $w_1$ and $w_2$ are said to be \emph{equivalent} if one can be obtained from the other by letter relabelling. That is, they have the same length $n$ and there is a bijection $\phi : \letters{w_1} \to \letters{w_2}$ such that $w_2[i] = \phi(w_1[i])$ for all $i\in \sinterval{n}$.
\end{definition}

The next lemma states that the sets of positions of a certain subset $L_w$ of the letters of any word $w\in \word$ uniquely determine the sets of positions of all the other letters.
\begin{lemma}
    \label{lem:uniquenesspredict}
    Let $n\geq 2$ and let $w_1, w_2\in \word_n$. For $i\in \{1, 2\}$, let $L_i \subset \letters{w_i}$ be defined by
    \begin{align*}
        L_i := & \{w_i[1], w_i[2]\} \cup \{\upletters{A}\in \letters{w_i} : \upletters{A}\text{ has variable neighbours in }w_i\} \\&\quad \cup \{\upletters{B}\in \letters{w_i} : \upletters{B}\text{ appears in $w_i$ next to a letter }\upletters{A}\text{ having variable neighbours in }w_i\}.
    \end{align*}
    Suppose that $\big\{ \pos{w_1}{\upletters{A}} : \upletters{A} \in L_1 \big\} = \big\{ \pos{w_2}{\upletters{A}} : \upletters{A} \in L_2 \big\}$. Then $w_1$ and $w_2$ are equivalent.
\end{lemma}
\begin{proof}
    Suppose that the conclusion does not hold, and let $k\geq 1$ be minimal with the property that $\posbis{w_1}{k}\neq \posbis{w_2}{k}$. Hence, $w_1[k] \not\in L_1$ and $w_2[k] \not\in L_2$ by the assumption in the statement. In particular, $k\geq 3$ since $w_i[1], w_i[2]\in L_i$.

    Note that $w_1[k]\neq w_1[k-2]$. Indeed, if $w_1[k] = w_1[k-2]$, we would have
    \begin{equation*}
        k\in \posbis{w_1}{k} = \posbis{w_1}{k-2} = \posbis{w_2}{k-2},
    \end{equation*}
    where the last equality follows from the minimality of $k$. This would imply that
    \begin{equation*}
        \posbis{w_2}{k} = \posbis{w_2}{k-2} = \posbis{w_1}{k},
    \end{equation*}
    contradicting our assumption.

    By definition of $L_1$, both $w_1[k]$ and $w_1[k-1]$ have constant neighbours in $w_1$. This means that every occurrence of the letter $w_1[k-1]$ in $w_1$ is surrounded by the (distinct) letters $w_1[k-2]$ and $w_1[k]$, in any order. In addition, every appearance of $w_1[k]$ is adjacent to an occurrence of $w_1[k-1]$. Thus, we may describe $\posbis{w_1}{k}$ exactly as
    \begin{equation}
        \label{eq:expressionPos}
        \posbis{w_1}{k} = \big\{l\in \sinterval{n}\  :\  \{l-1, l+1\} \cap \posbis{w_1}{k-1} \neq \emptyset,\  l\not\in \posbis{w_1}{k-2}\big\}.
    \end{equation}

    The same reasoning with $w_2$ shows that
    \begin{equation}
        \label{eq:expressionPos2}
        \posbis{w_2}{k} = \big\{l\in \sinterval{n}\  :\  \{l-1, l+1\} \cap \posbis{w_2}{k-1} \neq \emptyset,\  l\not\in \posbis{w_2}{k-2}\big\}.
    \end{equation}
    However, $\posbis{w_1}{k-2} = \posbis{w_2}{k-2}$ and $\posbis{w_1}{k-1} = \posbis{w_2}{k-1}$ by minimality of $k$, so \cref{eq:expressionPos,eq:expressionPos2} imply that $\posbis{w_1}{k} = \posbis{w_2}{k}$, a contradiction.
\end{proof}

We can now prove our bound for the number of inequivalent $t$-predictable words of length $n$.
\begin{lemma}
    \label{lem:countpredict}
    Let $n, t\geq 1$. There are $\leq n^{O(t^3)}$ $t$-predictable words in $\word_n$, up to equivalence.
\end{lemma}

\begin{proof}
    By \cref{lem:uniquenesspredict}, it suffices to bound the number of possibilities for the set
    \begin{equation}
        \label{eq:prepartition}
        \big\{ \pos{w}{\upletters{A}} : \upletters{A} \in L_w \big\}
    \end{equation}
    where $w$ ranges over the set of $t$-predictable words in $\word_n$, and $L_w\subset \letters{w}$ is the set defined in \cref{lem:uniquenesspredict} (with $w$ instead of $w_i$).

    If $w$ is $t$-predictable, there are $\leq t$ letters with variable neighbours. Moreover, every letter appears $\leq t$ times, so for every letter \upletters{A} there are $\leq 2t$ letters adjacent to an occurrence of \upletters{A}. Thus, by definition of $L_w$ we have
    \begin{equation*}
        \abs{L_w} \leq 2+t+t\cdot 2t \leq 5t^2.
    \end{equation*}

    For every $\upletters{A}\in L_w$, the set $\pos{w}{\upletters{A}}$ is a subset of $\sinterval{n}$ of size $\leq t$, and there are $\leq n^t$ such sets. Hence, there are $\leq (n^t)^{5t^2}$ possibilities for the set in \cref{eq:prepartition}, which concludes the proof.
\end{proof}

\subsection{Contribution of predictable walks}
We define a subset $\pred\subset \DD_{\s,\lit,r}$ of tuples $\d$ for which three words defined in terms of the primes $\dsub{i}{j}$ are predictable. By \cref{lem:countpredict}, this combinatorial restriction allows us to bound the contribution of these $\d$ to \cref{eq:sumoverD}. The remaining tuples will be handled in \cref{sec:unpredictable}.

\begin{definition}[$v_{j,\d}^{(k)}$, $w_{j,\d}^{(k)}$]
    \label{def:primewords}
    Let $\s$ and $\lit$ be disjoint subsets of $\intK\times \intJ$ and let $1\leq r\leq \K$.

    Let $\d \in \DD_{\s,\lit,r}$ and $j\in \intJ$. For each $1\leq k \leq 3$, we define a word $v_{j,\d}^{(k)}$ on the alphabet $\primes_j$, whose letters are the primes in $\primes_j$ corresponding to the tuple $\d^{(k)}$ from \cref{def:D}. Explicitly,
    \begin{equation*}
        \begin{cases}
            v_{j,\d}^{(1)} := \dsub{1}{j}\dsub{2}{j}\cdots \dsub{\K}{j},                \\[5pt]
            v_{j,\d}^{(2)} := \dsub{(\K+1)}{j} \dsub{(\K+2)}{j} \cdots \dsub{(2\K)}{j}, \\[5pt]
            v_{j,\d}^{(3)} := \dsub{1}{j} \dsub{2}{j} \cdots \dsub{(\K-r)}{j} \dsub{(\K+r+1)}{j} \dsub{(\K+r+2)}{j} \cdots \dsub{(2\K)}{j}.
        \end{cases}
    \end{equation*}

    We define $w_{j,\d}^{(k)}$ to be the \emph{compression} of $v_{j,\d}^{(k)}$, meaning the word formed by replacing any string of consecutive occurrences of a letter with a single instance of that letter. Thus, $w_{j,\d}^{(k)} \in \word$.
\end{definition}

\begin{definition}[$\pred$, $\unpred$]
    \label{def:predunpred}
    Let $\s$ and $\lit$ be disjoint subsets of $\intK\times \intJ$ and let $1\leq r\leq \K$.

    We define $\pred$ to be the set of all tuples $\d\in \DD_{\s,\lit,r}$ such that, for every $j\in \intJ$, the three words $w_{j,\d}^{(1)}$, $w_{j,\d}^{(2)}$ and $w_{j,\d}^{(3)}$ are $\K^{1/4}$-predictable. We write $\unpred := \DD_{\s,\lit,r} \setminus \pred$.
\end{definition}

We conclude this section by bounding the contribution of the tuples $\d \in \pred$ to \cref{eq:sumoverD}.

\begin{proposition}
    \label{prop:predictablevector}
    Let $\s$, $\lit$, $\unlit$ be sets such that $\intK\times\intJ = \s\sqcup\lit\sqcup\unlit$ and let $1\leq r\leq \K$. Then
    \begin{equation*}
        \sum_{\substack{\d \in \pred}} \, \prod_{p\mid \pall} \frac{1}{p} \ll e^{O(\J\K)}\V^{|\s|+(\abs{\lit}+\abs{\unlit})/2}.
    \end{equation*}
\end{proposition}

\begin{proof}
    The proof resembles that of \cref{lem:trivialbound}. The key difference is that we restrict to partitions coming from $\K^{1/4}$-predictable words, thereby avoiding combinatorial explosion.

    We first identify the combinatorial data needed to uniquely reconstruct a tuple $\d\in \pred$. Any $\d \in (\pm \D)^{2\K}$ is uniquely determined by a sequence $\sigma \in \{\pm 1\}^{2\K}$ of signs and, for each $j\in \intJ$, a word $u_j$ of length $2\K$, along with a sequence of distinct primes $(p_{\upletters{A}})_{\upletters{A}\in \letters{u_j}}$ with each $p_{\upletters{A}}\in \primes_j$. Indeed, given this data, we can reconstruct $\d$ by setting $\mathrm{sign}(d_i) = \sigma_i$ and $\dsub{i}{j} = p_{\upletters{A}}$ if $u_j[i] = \upletters{A}$.

    Fix some $j\in \intJ$. We claim that there are at most $e^{O(\K)}$ possibilities for the word $u_j$ up to equivalence, as $\d$ ranges over $\pred$. To prove this, it suffices to show that:
    \begin{enumerate}[label=(\roman*), ref=\roman*]
        \item \label{item:countwords} there are $\leq e^{O(\K)}$ possible words $\big(v_{j,\d}^{(k)}\big)_{1\leq k\leq 3}$ up to equivalence, and
        \item \label{item:countpartitions} the equivalence class of $u_j$ is uniquely determined by the equivalence classes of $\big(v_{j,\d}^{(k)}\big)_{1\leq k\leq 3}$.
    \end{enumerate}

    Let $1\leq k\leq 3$, and denote by $n_k$ the length of $w_{j,\d}^{(k)}$. Since $n_k\leq 2\K$, by \cref{lem:countpredict}, there are
    \begin{equation*}
        \leq 2\K \cdot (2\K)^{O(\K^{3/4})} \ll e^{\K}
    \end{equation*}
    possibilities for the $\K^{1/4}$-predictable word $w_{j,\d}^{(k)}$ up to equivalence. To recover the equivalence class of $v_{j,\d}^{(k)}$ from that of the compressed word $w_{j,\d}^{(k)}$, we need to specify a vector $(c_1, c_2, \ldots, c_{n_k})$ of positive integers, where $c_i$ records the number of consecutive repetitions of the $i$-th letter of $w_{j,\d}^{(k)}$ in $v_{j,\d}^{(k)}$. Since the length of $v_{j,\d}^{(k)}$ is $\leq 2\K$, the number of such integer vectors is $\leq e^{O(\K)}$. Hence, the total number of possible~$v_{j,\d}^{(k)}$, modulo equivalence, is also $\leq e^{O(\K)}$, which proves \cref{item:countwords}.

    We turn to \cref{item:countpartitions}. Fix $\big(v_{j,\d}^{(k)}\big)_{1\leq k\leq 3}$ modulo equivalence, and let $u$ and $u'$ be words of length $2\K$ that are possible candidates for $u_j$. We need to show that $u$ and $u'$ are equivalent. By \cref{def:primewords}, the compatibility of $u$ and $u'$ with $\big(v_{j,\d}^{(k)}\big)_{1\leq k\leq 3}$ means that
    \begin{equation}
        \label{eq:compatibility}
        \begin{cases}
            u|_{\sinterval{\K}} \sim u'|_{\sinterval{\K}}             \\
            u|_{\interval{\K+1}{2\K}} \sim u'|_{\interval{\K+1}{2\K}} \\
            u|_{\sinterval{\K-r} \cup \interval{\K+r+1}{2\K}} \sim u'|_{\sinterval{\K-r} \cup \interval{\K+r+1}{2\K}}
        \end{cases}
    \end{equation}
    where $\sim$ denotes equivalence of words and $u|_I$ denotes the subword of $u$ restricted to the indices in $I$. By part \cref{item:Dprop3} of \cref{def:D}, we also know that
    \begin{equation}
        \label{eq:compatibility2}
        u[\K+\ell] = u[\K-\ell+1] \quad \text{and} \quad u'[\K+\ell] = u'[\K-\ell+1]
    \end{equation}
    for all $\ell \in \sinterval{r}$.

    Suppose that there are $i_1<i_2$ such that $u[i_1] = u[i_2]$, and let us show that $u'[i_1] = u'[i_2]$.
    \begin{itemize}
        \item If $i_2\leq \K$, the first equation in \cref{eq:compatibility} implies that $u'[i_1] = u'[i_2]$.
        \item If $i_1 >\K$, the second equation in \cref{eq:compatibility} implies that $u'[i_1] = u'[i_2]$.
        \item If $i_1, i_2 \in \sinterval{\K-r} \cup \interval{\K+r+1}{2\K}$, the third equation in \cref{eq:compatibility} implies that $u'[i_1] = u'[i_2]$.
        \item If $i_1 \leq \K$ and $i_2 = \K+\ell$ for some $\ell\in \sinterval{r}$, then $u[i_1] =u[i_2] = u[\K-\ell+1]$ by \cref{eq:compatibility2}. The first equation in \cref{eq:compatibility} implies that $u'[i_1] = u'[\K-\ell+1]$, and thus $u'[i_1] = u'[i_2]$ by \cref{eq:compatibility2}.
        \item If $i_1 = \K-\ell+1$ for some $\ell\in \sinterval{r}$ and $i_2 > \K$, then $u[i_2] = u[i_1] = u[\K+\ell]$ by \cref{eq:compatibility2}. The second equation in \cref{eq:compatibility} implies that $u'[i_2] = u'[\K+\ell]$, and thus $u'[i_1] = u'[i_2]$ by \cref{eq:compatibility2}.
    \end{itemize}
    Therefore, $u[i_1] = u[i_2]$ implies that $u'[i_1] = u'[i_2]$ in all cases. Of course, the converse is also true by symmetry. This shows that $u$ and $u'$ are equivalent, proving \cref{item:countpartitions}.

    We can now finish the proof of the proposition. For every $\d\in \DD_{\s,\lit,r}$, the number of distinct primes dividing $\pall$ is $\leq |\s| + \tfrac{1}{2}(\abs{\lit}+\abs{\unlit})$, as every $\dsub{i}{j}$ with $(i,j)\not\in \s$ appears at least twice. Summing over all sequences of signs $\sigma$, possible words $(u_j)_{j\in \intJ}$ modulo equivalence, and primes $(p_{\upletters{A}})_{\upletters{A}\in \letters{u_j}}$ for every $j$, we get
    \begin{equation*}
        \sum_{\substack{\d \in \pred}} \, \prod_{p\mid \pall} \frac{1}{p} \leq 2^{2\K}\big(e^{O(\K)}\big)^{\J}\V^{|\s|+(\abs{\lit}+\abs{\unlit})/2},
    \end{equation*}
    using that ${\sum_{p\in \primes_j}1/p = \V_j\leq \V}$ for every $j$. Simplifying, we obtain the desired bound.
\end{proof}

\section{Triangular systems of divisibility constraints}
\label{sec:systems}

To bound the contribution of unpredictable walks and to treat prohibited sequences, we will need to count vectors $\d$ with coordinates $d_i\in \pm \D$ satisfying certain divisibility relations. The precise form of these relations will vary. In this section, we introduce a general framework for them that encompasses all the cases that will arise.

\subsection{Constraints}
We begin by giving a precise definition of a \emph{constraint} on a vector $\d$.

\begin{notation}
    \label{not:dijagain}
    For $R \geq 1$ and $d\in (\pm \D)^{R}$, we write $\dsub{i}{j}$ for the unique prime in $\primes_j$ dividing $d_i$.
\end{notation}

\begin{definition}
    \label{def:constraints}
    Let $R \geq 1$ and let $\d \in (\pm \D)^{R}$. A \emph{constraint} on $\d$ is any predicate of the form
    \begin{equation}
        \label{eq:shapeofconstraint}
        \dsub{i_0}{j_0} \bigmid\  \sum_{i\in I} d_i + \kappa
    \end{equation}
    for some $I\subset \sinterval{R}$, $(i_0, j_0)\in \sinterval{R} \times \intJ$ and $\kappa\in \Z$, where $I$ is the union of at most $10$ discrete intervals.\footnote{This condition on the combinatorial complexity of $I$ ensures that the number of possible $I$ is polynomial in $R$. It will always hold in practice.} The integer $\kappa$ is called the \emph{shift} of the constraint.
\end{definition}

This constraint \cref{eq:shapeofconstraint} should be viewed as a polynomial divisibility condition on the primes $\dsub{i}{j}$. In many cases, the shift $\kappa$ will be zero.

We now define what it means for a prime to be \emph{absent} from a constraint, and \emph{involved} in a constraint.

\begin{definition}
    \label{def:absent}
    A prime $p\in \primes$ is \emph{absent} from a constraint `$\dsub{i_0}{j_0} \mid \sum_{i\in I} d_i + \kappa$' if $p\neq \dsub{i_0}{j_0}$ and $p\nmid d_i$ for all $i\in I$.
\end{definition}

The definition of a prime $p$ being \emph{involved} in a constraint is not merely the negation of being absent, because we want to make sure that the constraint is not `degenerate' when viewed as a condition on $p$. For instance, consider the constraint ${\dsub{1}{1} \mid d_2 + d_3}$. If it happens that $\dsub{1}{1} = \dsub{2}{1} = \dsub{3}{1}$ and $\J=2$, the condition becomes $\dsub{1}{1} \mid \dsub{1}{1}(\dsub{2}{2} + \dsub{3}{2})$, which is always true. Such a degenerate case provides no useful information about any of the primes, so we wish to exclude it. This motivates the following definition.

\begin{definition}
    \label{def:involved}
    A prime $p\in \primes$ is said to be \emph{involved} in a constraint `$\dsub{i_0}{j_0} \mid \sum_{i\in I} d_i + \kappa$' if
    \begin{equation*}
        \sum_{\substack{i\in I\\ p \mid d_i}} d_i \not\equiv 0 \pmod{\dsub{i_0}{j_0}}.
    \end{equation*}
    That is, $p$ appears in the sum $\sum_{i\in I} d_i$ and its contribution to this sum does not vanish modulo $\dsub{i_0}{j_0}$. In particular, $p \neq \dsub{i_0}{j_0}$.
\end{definition}

\Cref{def:involved} is by no means the most general possible, but it is well adapted to the cases we will encounter.

\subsection{Triangular systems of constraints}
Our applications require bounding the number of vectors $\d \in (\pm \D)^{R}$ satisfying systems with many constraints. Handling arbitrary systems of constraints is often intractable, as constraints are non-linear divisibility conditions to very large, possibly distinct moduli $\dsub{i_0}{j_0}$. However, we can effectively derive meaningful bounds for systems of constraints that have a certain `triangular' structure. Specifically, after a suitable ordering of the variables, each variable is essentially determined by its corresponding constraint together with the preceding variables.

\begin{definition}
    \label{def:system}
    A \emph{triangular system with shift $\kappa$} on $\d$ is a sequence $C_1(\d), \ldots, C_{T}(\d)$ of constraints on $\d$ with shift $\kappa$ such that, for each $t\in \sinterval{{T}}$, there exists a prime $p_t$ that is involved in $C_t(\d)$ and absent from $C_1(\d), C_2(\d), \ldots, C_{t-1}(\d)$.
\end{definition}

Our main task in the forthcoming sections will be to exhibit such triangular systems of constraints. Once we achieve this, we will be able to bound the weighted number of solutions to these systems using the following lemma.

\begin{lemma}
    \label{lem:constraints}
    Let $1\leq T \leq R\leq 20 \K$. Let $B\geq 1$. Let $\XX\subset (\pm \D)^{R}$ be a set such that each $\d \in \XX$ satisfies a triangular system of $T$ constraints with shift $\kappa\in [-B, B]$ (thus, the system and its shift may depend on $\d$). Then
    \begin{equation*}
        \sum_{\d\in \XX} \prod_{p\mid d_1\cdots d_{R}} \frac1p \ll B \K^{O( \J R)} H_0^{-T/2}.
    \end{equation*}
\end{lemma}

The proof consists of simple iterated substitutions, but is quite heavy on the notational side. The key takeaway is that each constraint in a triangular system produces a saving of size $\sqrt{H_0}$.

\begin{proof}
    Let $\mathcal{C}$ denote the data consisting of:
    \begin{itemize}
        \item a sequence of signs $\sigma \in \{\pm1\}^{R}$;
        \item a partition $\Pi$ of $\sinterval{R} \times \intJ$;
        \item for each $t\in \sinterval{T}$, a set $I_t\subset \sinterval{R}$ which is a union of at most $10$ discrete intervals;
        \item for each $t\in \sinterval{T}$, a pair $(i_t,j_t)\in \sinterval{R} \times \intJ$;
        \item an integer $\kappa\in [-B, B]$;
        \item a function $f : \sinterval{T} \to \Pi$.
    \end{itemize}
    We define $\XX_{\mathcal{C}}$ to be the set of all $\d\in \XX$ such that
    \begin{itemize}
        \item $\mathrm{sign}(d_i) = \sigma_i$ for $i\in \sinterval{R}$;
        \item for all $(i,j), (i',j')\in \sinterval{R} \times \intJ$, $\dsub{i}{j} =\dsub{i'}{j'}$ iff $(i,j)$ and $(i',j')$ are in the same class in $\Pi$;
        \item for every $t\in \sinterval{T}$, the constraint $C_t(\d)$ defined by
              \begin{equation*}
                  \dsub{i_t}{j_t} \bigmid\  \sum_{i\in I_t} d_i + \kappa
              \end{equation*}
              is satisfied by $\d$;
        \item for $t\in \sinterval{T}$, the prime $d_{f(t)}$ is involved in $C_{t}(\d)$ but absent from $C_s(\d)$ for $s<t$.\footnote{For $\alpha\in \Pi$, we write $d_{\alpha}$ for the prime $\dsub{i}{j}$, where $(i,j)$ is any element of~$\alpha$; this is well-defined by construction.}
    \end{itemize}
    Thus, each $\d \in \XX$ is contained in $\XX_{\mathcal{C}}$ for at least one choice of $\mathcal{C}$ as above.

    We will show that, for each such choice of $\mathcal{C}$, we have
    \begin{equation}
        \label{eq:toproveinconstraints}
        \sum_{\d\in \XX_{\mathcal{C}}}  \prod_{p\mid d_1\cdots d_{R}} \frac1p \ll  \V^{R \J} H_0^{-T/2}.
    \end{equation}
    This is enough to prove \cref{lem:constraints}. Indeed, to bound the sum over $\d\in \XX$, it suffices to multiply the right-hand side of \cref{eq:toproveinconstraints} by the number of possibilities for $\mathcal{C}$. There are $2^{R}$ choices for $\sigma$. The number of partitions of $\sinterval{R} \times \intJ$ is $\leq (R\J)^{R \J}$. For $t\in \sinterval{T}$, since $I_t$ is a union of at most 10 discrete intervals, it is uniquely determined by 20 elements of $\sinterval{R}$. Thus, the number of choices for $(I_t, i_t, j_t)_{t\in \sinterval{T}}$ is $\leq (R^{20}R\J)^{T}$. There are $\leq 2B+1$ choices for $\kappa\in [-B, B]\cap \Z$. Any function $f :  \sinterval{T} \to \Pi$ induces a function $\sinterval{T} \to\sinterval{R} \times \intJ$ which uniquely determines $f$, so there are $\leq (R \J)^{T}$ possibilities for $f$. Therefore, assuming \cref{eq:toproveinconstraints}, we have
    \begin{equation*}
        \sum_{\d\in \XX} \prod_{p\mid d_1\cdots d_R} \frac1p \ll B \, 2^{R} (R \J)^{R \J+22T} \V^{R \J} H_0^{-T/2}.
    \end{equation*}
    By \cref{lem:primesets}, we have $\V^{\J} \ll \K$. Using $T \leq R \ll \K$, we can simplify the above to obtain
    \begin{equation*}
        \sum_{\d\in \XX} \prod_{p\mid d_1\cdots d_R} \frac1p \ll B \K^{O(\J R)} H_0^{-T/2},
    \end{equation*}
    as desired.

    It remains to prove \cref{eq:toproveinconstraints}. Fix $\mathcal{C} = (\sigma,\Pi, (I_t),(i_t),(j_t),\kappa,f)$ such that $\XX_{\mathcal{C}}$ is non-empty. Note that every class $\alpha$ of $\Pi$ is contained in $\sinterval{R} \times  \{j(\alpha)\}$ for some $j(\alpha)\in \intJ$, which is the unique integer such that $d_{\alpha} \in \primes_{\!j(\alpha)}$ for all $\d \in  \XX_{\mathcal{C}}$.

    Any $\d \in \XX_{\mathcal{C}}$ is uniquely determined by the sequence of primes $(d_{\alpha})_{\alpha\in \Pi}$.

    We define a nested sequence $(\Pi_t)_{0\leq t\leq T}$ of subsets of $\Pi$ that reveals the classes $f(t)\in \Pi$ one at a time, by setting
    \begin{equation*}
        \Pi_0 := \Pi\setminus \{f(t) : t\in \sinterval{T}\} \quad\text{and, for $t\in \sinterval{T}$,}\quad \Pi_t := \Pi_{t-1} \cup \{f(t)\}.
    \end{equation*}

    Let $W_0$ be the set of all sequences $\bm{p}_0 := (p_{\alpha})_{\alpha\in \Pi_0}$ with $p_{\alpha}\in \primes_{\!j(\alpha)}$ for all $\alpha\in \Pi_0$.

    For any $t\in \sinterval{T}$ and any sequence of primes $\bm{p}_{<t} := (p_{\alpha})_{\alpha\in \Pi_{t-1}}$, we define $W_t\big[ \bm{p}_{<t} \big]$ to be the set of all primes ${p\in \primes_{\!j(f(t))}}$ for which there is some $\d\in \XX_{\mathcal{C}}$ such that $d_{\alpha} = p_{\alpha}$ for all $\alpha\in \Pi_{t-1}$ and $d_{f(t)} = p$.

    Then, we have
    \begin{equation}
        \label{eq:nestedsums}
        \sum_{\d\in \XX_{\mathcal{C}}} \prod_{p\mid d_1\cdots d_R} \frac1p =
        \sum_{\bm{p}_0 \in W_0} \bigg(\prod_{\alpha\in \Pi_0} \frac{1}{p_{\alpha}} \bigg) \sum_{p_{f(1)}\in W_1} \frac{1}{p_{f(1)}}
        \sum_{p_{f(2)}\in W_2} \frac{1}{p_{f(2)}} \cdots \sum_{p_{f(T)}\in W_{T}} \frac{1}{p_{f(T)}},
    \end{equation}
    writing $W_t$ instead of $W_t\big[ \bm{p}_{<t}\big]$ to shorten notation.

    Fix some $t\in \sinterval{T}$ and some sequence $\bm{p}_{<t} = (p_{\alpha})_{\alpha\in \Pi_{t-1}}$. We claim that
    \begin{equation}
        \label{eq:sumwithcond}
        \sum_{p\in W_t} \frac{1}{p} \ll \frac{\log H}{H_0}.
    \end{equation}
    Recall that, for any $p\in W_t$, there is some $\d\in \XX_{\mathcal{C}}$ with $d_{\alpha} = p_{\alpha}$ for all $\alpha\in \Pi_{t-1}$ and $d_{f(t)} = p$.

    In particular, $p$ is involved in the constraint $C_t(\d)$ defined earlier, which means that
    \begin{equation}
        \label{eq:divcondAB}
        \dsub{i_t}{j_t} \mid A p + B + \kappa
    \end{equation}
    where
    \begin{equation*}
        A := \frac{1}{p}\sum_{\substack{i\in I_t\\ p \mid d_i}} d_i \not\equiv 0 \pmod{\dsub{i_t}{j_t}}\qquad\text{and}\qquad B := \sum_{\substack{i\in I_t\\ p \nmid d_i}} d_i.
    \end{equation*}
    Observe that $A$ and $B$ are explicit expressions in the primes $\bm{p}_{<t}$. Indeed, $f(t)$ is of the form ${f(t) = I_t' \times \{j(f(t))\}}$ for some $I_t'\subset \sinterval{R}$, and we may rewrite
    \begin{equation*}
        A = \sum_{i\in I_t \cap I_t'}\, \prod_{\substack{j\in \intJ\setminus \{j(f(t))\}}} \dsub{i}{j} \qquad\text{and}\qquad B = \sum_{i\in I_t\setminus I_t'} \prod_{j\in \intJ} \dsub{i}{j}.
    \end{equation*}
    By definition of $f(t)$ and $I_t'$, the prime $p = d_{f(t)}$ does not appear in $A$ or $B$. Furthermore, by construction, the primes $d_{f(t+1)} ,\ldots, d_{f(T)}$ are absent from $C_t(\d)$, which means that $d_{f(t+1)} ,\ldots, d_{f(T)}$ cannot be any of the primes $\dsub{i}{j}$ occurring in $A$ or $B$ either. Hence, $A$ and $B$ are fully determined by the primes $\bm{p}_{<t} = (p_{\alpha})_{\alpha\in \Pi_{t-1}}$.

    Thus, this divisibility condition \cref{eq:divcondAB} uniquely determines the congruence class of $p$ modulo $\dsub{i_t}{j_t}$. Using that $\primes \subset (H_0, H)$, we have, for any $x$,
    \begin{equation*}
        \sum_{\substack{p\in \primes \\ p \equiv x\spmod{\dsub{i_t}{j_t}}}} \frac{1}{p} \leq \max_{a} \sum_{\substack{H_0\leq n\leq H \\ n\equiv a \spmod{\lfloor H_0\rfloor }}} \frac{1}{n} \ll \frac{\log H}{H_0},
    \end{equation*}
    which proves \cref{eq:sumwithcond}.

    Applying \cref{eq:sumwithcond} in \cref{eq:nestedsums} successively for $t = T, \,T-1,\,\ldots, \, 1$ yields
    \begin{equation*}
        \sum_{\d\in \XX_{\mathcal{C}}} \prod_{p\mid d_1\cdots d_R} \frac1p \leq \left(\frac{O(\log H)}{H_0}\right)^{T} \sum_{\bm{p}_0\in W_0} \prod_{\alpha\in \Pi_0} \frac{1}{p_{\alpha}} \ll H_0^{-T/2} \V^{\abs{\Pi_0}}.
    \end{equation*}
    Equation \cref{eq:toproveinconstraints} follows, which concludes the proof of \cref{lem:constraints}.
\end{proof}

\section{Unpredictable walks}
\label{sec:unpredictable}

The goal of this section is to prove the following proposition, which states that the contribution of tuples $\d$ arising from unpredictable words is negligible.

\begin{restatable}{proposition}{unpredcontrib}
    \label{prop:unpredictable}
    Let $\s$, $\lit$, $\unlit$ be sets such that $\intK\times\intJ = \s\sqcup\lit\sqcup\unlit$ and $|\unlit| \leq\K^{2\epsone}$. Let $1\leq r\leq \K$. Then
    \begin{equation*}
        \sum_{\d\in \unpred}\,\prod_{p\mid \pall}\, \frac{1}{p} \ll 1.
    \end{equation*}
\end{restatable}

Our strategy is as follows. Whenever a prime is repeated at lit indices, part~\cref{item:Dprop2} of \cref{def:D} yields a divisibility condition on $\d$. These conditions restrict the possibilities for $\d$: generically, we might hope to win a factor $\approx H_0$ from each such condition, which would be more than sufficient. Unfortunately, the multitude of dependencies between these conditions makes it hard to rule out highly degenerate systems.

However, since $H_0$ is much larger than~$\K$, it is enough to win a moderate number of factors $H_0$ to beat the trivial bound of \cref{lem:trivialbound}. To do so, we extract from the original system of lit conditions a subset that is trivially non-singular, namely a \emph{triangular system}. The existence of such a triangular system will ultimately follow from the unpredictability of the words associated to $\d\in \unpred$.

\subsection{Structure of unpredictable words}
Our intermediate aim is to prove \cref{prop:structureunpred}, which states that $t$-unpredictable words must contain one of several special patterns. These patterns will allow us to extract large triangular systems from the divisibility conditions associated to the lit indices, when $\d\in \unpred$.

Recall that $\word$ denotes the set of words on the alphabet $\mathcal{A}$ with no two consecutive equal letters, and $\word^{\neq}\subset \word$ is the set of words with distinct letters.

\begin{definition}
    \label{def:repetitions}
    A word $w\in \word$ contains $n$ \emph{separated repetitions} if it has a substring of the form
    \begin{equation*}
        \upletters{A\makesubscript{1}\ldots A\makesubscript{1}\ldots A\makesubscript{2}\ldots A\makesubscript{2}\ldots$\cdots$\ldots A\makesubscript{n}\ldots A\makesubscript{n}},
    \end{equation*}
    for some not necessarily distinct letters \upletters{A\makesubscript{1}$,\ldots,$A\makesubscript{n}}. The three dots \upletters{\ldots} represent a string of letters of arbitrary length (possibly empty). In other words, there are $k_1<l_1<k_2<...<k_n<l_n$ such that $w[k_i]=w[l_i]$ for all $i$.
\end{definition}

\begin{lemma}
    \label{lem:unpredstep0}
    Let $m \geq 10$. Let $k_1<k_2<\ldots<k_{m}$ be positive integers. Let $w\in \word$ be a word of length $\geq k_{m}$. Then, either $w$ contains $\gg m^{1/2}$ separated repetitions, or there are $i,j\in \sinterval{m}$ with $j-i \gg m^{1/2}$ such that the substring
    \begin{equation*}
        w[k_i]w[k_{i}+1] \cdots w[k_j]
    \end{equation*}
    of $w$ has distinct letters.
\end{lemma}

\begin{proof}
    Let $n = \lfloor m^{1/2}\rfloor$. If the second conclusion does not hold, there must be a repeated letter in the substring $w[k_{rn+1}]w[k_{rn+1}+1]\cdots w[k_{(r+1)n}]$, for each $0\leq r\leq n-1$. This implies that $w$ contains $n$ separated repetitions.
\end{proof}

\begin{lemma}
    \label{lem:unpredstep1}
    Let $\upletters{A}, \upletters{B}, \upletters{C}\in \mathcal{A}$. Let $w_1, w_2\in \word^{\neq}$ be two words of the form
    \begin{equation*}
        \upletters{A\ldots B\ldots C}.
    \end{equation*}
    Suppose that \upletters{B} has variable neighbours in the concatenation $w_1w_2$.\footnote{This just means that the two letters adjacent to \upletters{B} in $w_1$ are not the same as the two letters adjacent to \upletters{B} in $w_2$}

    Then, there exist substrings $v_1 \sqsubset w_1$ and $v_2 \sqsubset w_2$, both of the form $\upletters{A\ldots Y}$ for some letter \upletters{Y} (possibly equal to \upletters{B} or \upletters{C}), with distinct sets of letters, i.e.~$\letters{v_1} \neq \letters{v_2}$.
\end{lemma}

\begin{proof}
    If $\letters{w_1} \neq \letters{w_2}$, we can just take $v_1 := w_1$, $v_2 := w_2$ and $\upletters{Y} := \upletters{C}$.

    Otherwise, $w_1$ and $w_2$ have the same sets of letters, and thus the same length, as $w_1, w_2\in \word^{\neq}$. Let $k\geq 2$ be minimal such that $w_1[k] \neq w_2[k]$. We know that $k$ exists, since $w_1\neq w_2$. Let $\upletters{X} := w_1[k]$ and $\upletters{Y} := w_2[k]$. The letter \upletters{Y} is present in $w_1$ as both words have the same letters. By minimality of $k$, we must have $\upletters{Y} = w_1[l]$ for some $l>k$. Set ${v_1 := w_1[1]w_1[2]\cdots w_1[l] = \upletters{A\ldots X\ldots Y}}$ and $v_2 := w_2[1]w_2[2]\cdots w_2[k] = \upletters{A\ldots Y}$. Then $\letters{v_1}\neq \letters{v_2}$ as $\upletters{X} \in \letters{v_1}\setminus \letters{v_2}$.
\end{proof}

\begin{notation}
    \label{not:substring}
    Let $w\in \word^{\neq}$, and suppose that $w$ is of the form $\upletters{A\ldots X\ldots Y\ldots B}$. We write $w|_{\upletters{X\ldots Y}}$ for the unique substring of $w$ of the form $\upletters{X\ldots Y}$. This is well-defined as $w$ has distinct letters.
\end{notation}

\begin{lemma}
    \label{lem:unpredstep2}
    Let $m\geq 1$. Let $\upletters{A\makesubscript{0}}, \ldots, \upletters{A\makesubscript{2m}}\in \mathcal{A}$. Let $w_1,w_2\in \word^{\neq}$ be two words of the form
    \begin{equation*}
        \upletters{A\makesubscript{0}\ldots A\makesubscript{1}\ldots A\makesubscript{2}\ldots$\cdots$\ldots A\makesubscript{2m}}
    \end{equation*}
    such that, for all $1\leq i\leq 2m-1$, the letter \upletters{A\makesubscript{i}} has variable neighbours in the concatenation $w_1w_2$.

    Then, there are substrings $v_1 \sqsubset w_1$ and $v_2 \sqsubset w_2$, both of the form
    \begin{equation*}
        \upletters{Y\makesubscript{0}\ldots Y\makesubscript{1}\ldots Y\makesubscript{2}\ldots$\cdots$\ldots Y\makesubscript{m}},
    \end{equation*}
    for some letters $\upletters{Y\makesubscript{0}},\ldots,\upletters{Y\makesubscript{m}}$ (possibly equal to some of the $\upletters{A\makesubscript{i}}$) such that, for all $j\in \sinterval{m}$, the sets of letters of $v_1|_{\upletters{Y\makesubscript{j\!-\!1}\ldots Y\makesubscript{j}}}$ and $v_2|_{\upletters{Y\makesubscript{j\!-\!1}\ldots Y\makesubscript{j}}}$ are distinct.
\end{lemma}

\begin{proof}
    Set $\upletters{Y\makesubscript{0}} := \upletters{A\makesubscript{0}}$. Let $w_1^{1} := w_1|_{\upletters{Y\makesubscript{0}\ldots A\makesubscript{1}\ldots A\makesubscript{2}}}$  and $w_2^1 := w_2|_{\upletters{Y\makesubscript{0}\ldots A\makesubscript{1}\ldots A\makesubscript{2}}}$. Applying \cref{lem:unpredstep1} to these words $w_1^{1}$ and $w_2^{1}$, we find two further substrings $v_1^1 \sqsubset w_1^1$ and $v_2^1 \sqsubset w_2^1$ of the form $\upletters{Y\makesubscript{0}\ldots Y\makesubscript{1}}$ for some common ending letter $\upletters{Y\makesubscript{1}}$, such that $\letters{v_1^1} \neq \letters{v_2^1}$. Notice that $w_1$ and $w_2$ are of the form
    \begin{equation*}
        \upletters{Y\makesubscript{0}\ldots Y\makesubscript{1}\ldots A\makesubscript{3}\ldots A\makesubscript{4}\ldots$\cdots$\ldots A\makesubscript{2m}}.
    \end{equation*}
    We may thus define the substrings $w_1^2 := w_1|_{\upletters{Y\makesubscript{1}\ldots A\makesubscript{3}\ldots A\makesubscript{4}}}$ and $w_2^2 := w_2|_{\upletters{Y\makesubscript{1}\ldots A\makesubscript{3}\ldots A\makesubscript{4}}}$. Applying \cref{lem:unpredstep1} again with $w_1^2$ and $w_2^2$, we obtain two substrings $v_1^2 \sqsubset w_1^2$ and $v_2^2 \sqsubset w_2^2$ of the form $\upletters{Y\makesubscript{1}\ldots Y\makesubscript{2}}$, with $\letters{v_1^2} \neq \letters{v_2^2}$. In particular, $w_1$ and $w_2$ can now be written as
    \begin{equation*}
        \upletters{Y\makesubscript{0}\ldots Y\makesubscript{1}\ldots Y\makesubscript{2}\ldots A\makesubscript{5}\ldots A\makesubscript{6}\ldots$\cdots$\ldots A\makesubscript{2m}}.
    \end{equation*}
    We can repeat this process; after $m$ applications of \cref{lem:unpredstep1}, we obtain substrings of $w_1$ and $w_2$ of the form $\upletters{Y\makesubscript{0}\ldots Y\makesubscript{1}\ldots Y\makesubscript{2}\ldots$\cdots$\ldots Y\makesubscript{m}}$ with the required properties.
\end{proof}

\begin{lemma}
    \label{lem:unpredstep3}
    Let $m\geq 1$. Let $\upletters{Y\makesubscript{0}}, \ldots, \upletters{Y\makesubscript{4m}}\in \mathcal{A}$. Let $w_1, w_2\in \word^{\neq}$ be two words of the form
    \begin{equation*}
        \upletters{Y\makesubscript{0}\ldots Y\makesubscript{1}\ldots Y\makesubscript{2}\ldots$\cdots$\ldots Y\makesubscript{4m}}.
    \end{equation*}
    Suppose that, for all $j\in \sinterval{4m}$, the words $w_1|_{\upletters{Y\makesubscript{j\!-\!1}\ldots Y\makesubscript{j}}}$ and $w_2|_{\upletters{Y\makesubscript{j\!-\!1}\ldots Y\makesubscript{j}}}$ have distinct sets of letters.

    Then, there is a pair of words $(w_1', w_2') \in \{(w_1, w_2), (w_2, w_1), (\overleftarrow{w_1}, \overleftarrow{w_2}), (\overleftarrow{w_2}, \overleftarrow{w_1})\}$ with the following properties.

    There are letters $\upletters{X\makesubscript{1}},\ldots,\upletters{X\makesubscript{m}},\upletters{Z\makesubscript{0}},\upletters{Z\makesubscript{1}},\ldots,\upletters{Z\makesubscript{m}},\upletters{Z\makesubscript{m\!+\!1}}$ (possibly equal to some of the $\upletters{Y\makesubscript{j}}$) such that $w_1'$ is of the form
    \begin{equation*}
        \upletters{Z\makesubscript{0}\ldots X\makesubscript{1}\ldots Z\makesubscript{1}\ldots X\makesubscript{2}\ldots Z\makesubscript{2}\ldots$\cdots$\ldots Z\makesubscript{m\!-\!1}\ldots X\makesubscript{m}\ldots Z\makesubscript{m}\ldots Z\makesubscript{m\!+\!1}},
    \end{equation*}
    $w_2'$ is of the form
    \begin{equation*}
        \upletters{Z\makesubscript{0}\ldots Z\makesubscript{1}\ldots Z\makesubscript{2}\ldots$\cdots$\ldots Z\makesubscript{m}\ldots Z\makesubscript{m\!+\!1}},
    \end{equation*}
    and, for all $j\in \sinterval{m}$, the letter \upletters{X\makesubscript{j}} does not appear in $w_2'|_{\upletters{Z\makesubscript{0}\ldots Z\makesubscript{j}}}$.
\end{lemma}

\begin{proof}
    Let $J_1$ be the set of all $j\in \sinterval{4m}$ such that $w_1|_{\upletters{Y\makesubscript{j\!-\!1}\ldots Y\makesubscript{j}}}$ contains a letter not appearing in $w_2|_{\upletters{Y\makesubscript{j\!-\!1}\ldots Y\makesubscript{j}}}$. Similarly, let $J_2$ be the set of all $j\in \sinterval{4m}$ such that $w_2|_{\upletters{Y\makesubscript{j\!-\!1}\ldots Y\makesubscript{j}}}$ has a letter that is not present in $w_1|_{\upletters{Y\makesubscript{j\!-\!1}\ldots Y\makesubscript{j}}}$. By assumption, $J_1 \cup J_2 = \sinterval{4m}$, so one of $J_1$ and $J_2$ has size $\geq 2m$. Without loss of generality, assume that $\abs{J_1} \geq 2m$, swapping $w_1$ and $w_2$ if necessary.

    Let $j\in J_1$, and let $\upletters{X}$ be a letter present in $w_1|_{\upletters{Y\makesubscript{j\!-\!1}\ldots Y\makesubscript{j}}}$ but not in $w_2|_{\upletters{Y\makesubscript{j\!-\!1}\ldots Y\makesubscript{j}}}$. The letter \upletters{X} could possibly appear in $w_2|_{\upletters{Y\makesubscript{0}\ldots Y\makesubscript{j\!-\!1}}}$ or in $w_2|_{\upletters{Y\makesubscript{j}\ldots Y\makesubscript{4m}}}$, but not in both as $w_2 \in \word^{\neq}$.

    We define $J_1^<$ to be the set of all $j\in J_1$ for which there exists a letter $\upletters{X}$ present in $w_1|_{\upletters{Y\makesubscript{j\!-\!1}\ldots Y\makesubscript{j}}}$ but not in $w_2|_{\upletters{Y\makesubscript{0}\ldots Y\makesubscript{j}}}$. Similarly, we define $J_1^>$ to be the set of all $j\in J_1$ for which there exists a letter $\upletters{X}$ of $w_1|_{\upletters{Y\makesubscript{j\!-\!1}\ldots Y\makesubscript{j}}}$ not appearing in $w_2|_{\upletters{Y\makesubscript{j\!-\!1}\ldots Y\makesubscript{4m}}}$. By the previous observation, we have $J_1^<\cup J_1^> = J_1$, so one of $J_1^<$ and $J_1^>$ has size $\geq m$. Considering the reversed words if necessary, we may assume without loss of generality that $\abs{J_1^<} \geq m$.

    Let $j_1<j_2<\ldots<j_m$ be elements of $J_1^<$. For $i\in \sinterval{m}$, let $\upletters{X\makesubscript{i}}$ be a letter of $w_1|_{\upletters{Y\makesubscript{j_i\!-\!1}\ldots Y\makesubscript{j_i}}}$ not appearing in the substring $w_2|_{\upletters{Y\makesubscript{0}\ldots Y\makesubscript{j_i}}}$ of $w_2$. Then $w_1$ is of the form
    \begin{equation*}
        \upletters{Y\makesubscript{0}\ldots X\makesubscript{1}\ldots Y\makesubscript{j\makesubscript{1}}\ldots X\makesubscript{2}\ldots Y\makesubscript{j\makesubscript{2}}\ldots$\cdots$\ldots Y\makesubscript{j\makesubscript{m\!-\!1}}\ldots X\makesubscript{m}\ldots Y\makesubscript{j\makesubscript{m}}\ldots Y\makesubscript{4m}}.
    \end{equation*}
    The lemma follows, defining $\upletters{Z\makesubscript{0}} := \upletters{Y\makesubscript{0}}$, $\upletters{Z\makesubscript{m\!+\!1}} := \upletters{Y\makesubscript{4m}}$ and $\upletters{Z\makesubscript{i}} := \upletters{Y\makesubscript{j\makesubscript{i}}}$ for all $i\in \sinterval{m}$.
\end{proof}

Combining \cref{lem:unpredstep2} and \cref{lem:unpredstep3}, we immediately obtain the following.

\begin{lemma}
    \label{lem:unpredstep4}
    Let $m\geq 1$. Let $\upletters{A\makesubscript{0}}, \ldots, \upletters{A\makesubscript{8m}}\in \mathcal{A}$. Let $w_1,w_2\in \word^{\neq}$ be two words of the form
    \begin{equation*}
        \upletters{A\makesubscript{0}\ldots A\makesubscript{1}\ldots A\makesubscript{2}\ldots$\cdots$\ldots A\makesubscript{8m}}.
    \end{equation*}
    Suppose that, for $1\leq i\leq 8m-1$, the letter \upletters{A\makesubscript{i}} has variable neighbours in the concatenation $w_1w_2$.

    Then, after possibly replacing $(w_1, w_2)$ with an element of $\{(w_1, w_2), (w_2, w_1), (\overleftarrow{w_1}, \overleftarrow{w_2}), (\overleftarrow{w_2}, \overleftarrow{w_1})\}$, the following applies.

    For some letters $\upletters{X\makesubscript{1}},\ldots,\upletters{X\makesubscript{m}},\upletters{Z\makesubscript{0}},\upletters{Z\makesubscript{1}},\ldots,\upletters{Z\makesubscript{m}}$ (possibly equal to some of the $\upletters{A\makesubscript{j}}$), there are words $v_1\sqsubset w_1$ and $v_2 \sqsubset w_2$, with $v_1$ of the form
    \begin{equation*}
        \upletters{Z\makesubscript{0}\ldots X\makesubscript{1}\ldots Z\makesubscript{1}\ldots X\makesubscript{2}\ldots Z\makesubscript{2}\ldots$\cdots$\ldots Z\makesubscript{m\!-\!1}\ldots X\makesubscript{m}\ldots Z\makesubscript{m}}
    \end{equation*}
    and $v_2$ of the form
    \begin{equation*}
        \upletters{Z\makesubscript{0}\ldots Z\makesubscript{1}\ldots Z\makesubscript{2}\ldots$\cdots$\ldots Z\makesubscript{m}},
    \end{equation*}
    such that, for all $j\in \sinterval{m}$, the letter \upletters{X\makesubscript{j}} does not appear in the substring $v_2|_{\upletters{Z\makesubscript{0}\ldots Z\makesubscript{j}}}$.
\end{lemma}

It is a well-known combinatorial fact that from any sequence of $n$ distinct real numbers one can always extract an increasing or decreasing subsequence of length $\gg \sqrt{n}$. We will use a similar result about \emph{pairs} of real numbers.

\begin{lemma}
    \label{lem:combipairs}
    Let $S$ be a set of $n$ pairs of real numbers, such that
    \begin{itemize}
        \item if $(a,b)\in S$ then $a<b$, and
        \item if  $(a,b), (c,d)\in S$ are two distinct pairs, then $\{a, b\}\cap \{c,d\} = \emptyset$.
    \end{itemize}

    Then, there exists $S' \subset S$ of size $n' \geq n^{1/4}$ such that one of the following holds.\footnote{The bound $n' \geq n^{1/4}$ can be improved, but that is not relevant for us.}
    \begin{enumerate}[label=(\roman*), ref=\roman*]
        \item \label{item:pairs1} $S' = \{(a_1, b_1), (a_2, b_2),\ldots, (a_{n'}, b_{n'})\}$ for some $a_1<b_1<a_2<b_2<\cdots <b_{n'}$.
        \item \label{item:pairs2} $S' = \{(a_1, b_1), (a_2, b_2),\ldots, (a_{n'}, b_{n'})\}$ for some $a_1<a_2<\cdots <a_{n'}<b_1<b_2<\cdots <b_{n'}$.
        \item \label{item:pairs3} $S' = \{(a_1, b_1), (a_2, b_2),\ldots, (a_{n'}, b_{n'})\}$ for some $a_1<a_2<\cdots <a_{n'}<b_{n'}<b_{n'-1}<\cdots <b_1$.
    \end{enumerate}
\end{lemma}
\begin{proof}
    Define a strict partial order $\prec^1$ on $S$ by setting $(a,b) \prec^1 (c,d)$ iff $b<c$. A well-known consequence of Dilworth's theorem states that any partially ordered set on $n$ elements contains a chain or an antichain\footnote{Recall that a \emph{chain} is a totally ordered subset of a partially ordered set, and an \emph{antichain} is a subset in which no two elements are comparable.} of size $\geq n^{1/2}$ (see \cite[Proposition~2.5.9]{ordered}). If $S$ contains a chain of size $\geq n^{1/2}$ for $\prec^1$, we are in case \cref{item:pairs1}. Suppose that $S$ contains an antichain $A$ of size $\geq n^{1/2}$. We introduce another strict partial order $\prec^2$ on $A$ by defining $(a,b) \prec^2 (c,d)$ iff $a<c<d<b$. By the same combinatorial fact, either $A$ contains a chain for $\prec^2$ of size $\geq n^{1/4}$, and case \cref{item:pairs3} applies, or $A$ contains an antichain $A'$ for $\prec^2$ of size $\geq n^{1/4}$. Suppose that the latter possibility occurs. Let $(a_1, b_1), \ldots, (a_{n'}, b_{n'})$ be the elements of $A'$, with $a_1<a_2 <\cdots < a_{n'}$. Since $A'$ is an antichain for $\prec^1$, all the $b_i$ are greater than $a_{n'}$. Since $A'$ is also an antichain for $\prec^2$, we deduce that $b_1<b_2<\cdots <b_{n'}$, and we are in case~\cref{item:pairs2}.
\end{proof}

We will combine the previous lemmas to extract useful substructures in unpredictable words.

\begin{proposition}
    \label{prop:structureunpred}
    There is an absolute constant $c_1> 0$ such that the following holds.

    Let $n,t\geq 10$. Let $w\in \word_n$ be a $t$-unpredictable word. Then, for some $m \gg t^{c_1}$, at least one of the properties below is satisfied.
    \begin{enumerate}
        \item $w$ has $m$ separated repetitions.
        \item There are words $v_1, v_2$ with all of the following properties:
              \begin{enumerate}[label=(\roman*), ref=\roman*]
                  \item $v_1\in \word^{\neq}$ and $v_2\in \word^{\neq}$;
                  \item $v_1 \sqsubset w$ or $\overleftarrow{v_1} \sqsubset w$;
                  \item $v_2 \sqsubset w$ or $\overleftarrow{v_2} \sqsubset w$;
                  \item there are letters $\upletters{X\makesubscript{1}},\ldots,\upletters{X\makesubscript{m}},\upletters{Z\makesubscript{0}},\ldots,\upletters{Z\makesubscript{m}}$ such that $v_1$ is of the form
                        \begin{equation*}
                            \upletters{Z\makesubscript{0}\ldots X\makesubscript{1}\ldots Z\makesubscript{1}\ldots X\makesubscript{2}\ldots Z\makesubscript{2}\ldots$\cdots$\ldots Z\makesubscript{m-1}\ldots X\makesubscript{m}\ldots Z\makesubscript{m}}
                        \end{equation*}
                        and $v_2$ is of the form
                        \begin{equation*}
                            \upletters{Z\makesubscript{0}\ldots Z\makesubscript{1}\ldots Z\makesubscript{2}\ldots$\cdots$\ldots Z\makesubscript{m}}.
                        \end{equation*}
                        Moreover, for all $j\in \sinterval{m}$, the letter \upletters{X\makesubscript{j}} does not appear in $v_2|_{\upletters{Z\makesubscript{0}\ldots Z\makesubscript{j}}}$.
              \end{enumerate}
    \end{enumerate}
\end{proposition}

\begin{proof}
    By definition of unpredictability, either $w$ contains a letter repeated $>t$ times, or it has $>t$ letters with variable neighbours. In the first case, we immediately see that $w$ has $\lfloor t/2 \rfloor$ separated repetitions.

    Suppose now that there are $>t$ letters with variable neighbours in $w$. Let $E$ be the set of all these letters, with the possible exception of the first and last letters of $w$ which are discarded (to simplify the notation below). Thus, $\abs{E} \geq t-2$. For every letter $\upletters{A} \in E$, there are two positions $1<k_{\upletters{A}}<l_{\upletters{A}}<n$ such that $w[k_{\upletters{A}}] = w[l_{\upletters{A}}] = \upletters{A}$, and the sets of letters adjacent to these two occurrences of \upletters{A} are different, i.e.~$\{w[k_{\upletters{A}}-1], w[k_{\upletters{A}}+1]\} \neq \{w[l_{\upletters{A}}-1], w[l_{\upletters{A}}+1]\}$.

    We apply \cref{lem:combipairs} to the set $S = \{(k_{\upletters{A}}, l_{\upletters{A}}) : \upletters{A}\in E\}$. If case \cref{item:pairs1} occurs, we can immediately conclude that $w$ has $\gg t^{1/4}$ separated repetitions and we are done.

    Suppose that case \cref{item:pairs2} of \cref{lem:combipairs} applies. This implies that, for some $c\gg1$, there exists a subset
    \begin{equation*}
        F = \{\upletters{A\makesubscript{1}},\ldots,\upletters{A\makesubscript{\abs{F}}}\}\subset E
    \end{equation*}
    of size $\abs{F} \geq t^{c}$ such that
    \begin{equation*}
        1<k_{\upletters{A\makesubscript{1}}}<k_{\upletters{A\makesubscript{2}}}<\cdots <k_{\upletters{A\makesubscript{\abs{F}}}}<l_\upletters{A\makesubscript{1}}<l_\upletters{A\makesubscript{2}}<\cdots <l_{\upletters{A\makesubscript{\abs{F}}}}<n.
    \end{equation*}
    By \cref{lem:unpredstep0}, either $w$ has $\gg t^{c/2}$ separated repetitions, and the first conclusion holds, or we can find a `large' substring of $w[k_{\upletters{A\makesubscript{1}}}]w[k_{\upletters{A\makesubscript{1}}}+1]\cdots w[k_{\upletters{A\makesubscript{\abs{F}}}}]$ with distinct letters. Without loss of generality (by replacing $F$ with a smaller subset, $c$ with a smaller absolute constant and relabelling the letters), we may thus assume that the word $w[k_{\upletters{A\makesubscript{1}}}]w[k_{\upletters{A\makesubscript{1}}}+1]\cdots w[k_{\upletters{A\makesubscript{\abs{F}}}}]$ itself has distinct letters. By a further application of \cref{lem:unpredstep0}, we may also assume that the word $w[l_{\upletters{A\makesubscript{1}}}]w[l_{\upletters{A\makesubscript{1}}}+1]\cdots w[l_{\upletters{A\makesubscript{\abs{F}}}}]$ has distinct letters.

    We apply \cref{lem:unpredstep4} with $w_1 := w[k_{\upletters{A\makesubscript{1}}}]w[k_{\upletters{A\makesubscript{1}}}+1]\cdots w[k_{\upletters{A\makesubscript{\abs{F}}}}]$ and $w_2 := w[l_{\upletters{A\makesubscript{1}}}]w[l_{\upletters{A\makesubscript{1}}}+1]\cdots w[l_{\upletters{A\makesubscript{\abs{F}}}}]$. These are two words in $\word^{\neq}$ of the form
    \begin{equation*}
        \upletters{A\makesubscript{1}\ldots A\makesubscript{2}\ldots A\makesubscript{3}\ldots} \cdots \upletters{\ldots A\makesubscript{\abs{F}}},
    \end{equation*}
    so the assumptions of \cref{lem:unpredstep4} are satisfied (of course, we may assume that $\abs{F} \equiv 1 \pmod{8}$ without loss of generality). The conclusion of \cref{lem:unpredstep4} provides us with two words $v_1$ and $v_2$ precisely satisfying the second conclusion of \cref{prop:structureunpred}.

    The treatment of case \cref{item:pairs3} of \cref{lem:combipairs} is similar. For some $c\gg1$, there exists a subset
    \begin{equation*}
        F = \{\upletters{A\makesubscript{1}},\ldots,\upletters{A\makesubscript{\abs{F}}}\}\subset E
    \end{equation*}
    of size $\abs{F} \geq t^{c}$ such that
    \begin{equation*}
        1<k_{\upletters{A\makesubscript{1}}}<k_{\upletters{A\makesubscript{2}}}<\cdots <k_{\upletters{A\makesubscript{\abs{F}}}}<l_\upletters{A\makesubscript{\abs{F}}}<l_\upletters{A\makesubscript{\abs{F}\!-\!1}}<\cdots <l_{\upletters{A\makesubscript{1}}}<n.
    \end{equation*}
    By two successive applications of \cref{lem:unpredstep0}, we may assume, without loss of generality, that the substrings $w_1 := w[k_{\upletters{A\makesubscript{1}}}]w[k_{\upletters{A\makesubscript{1}}}+1]\cdots w[k_{\upletters{A\makesubscript{\abs{F}}}}]$ and $w_2 := w[l_{\upletters{A\makesubscript{\abs{F}}}}]w[l_{\upletters{A\makesubscript{\abs{F}}}}-1]\cdots w[l_{\upletters{A\makesubscript{1}}}]$ each have distinct letters. Then, applying \cref{lem:unpredstep4} with these two substrings $w_1$ and $w_2$ produces two words $v_1$ and $v_2$ with the required properties.
\end{proof}

\subsection{Triangular systems for unpredictable walks}

We now show that tuples $\d$ whose associated words are unpredictable (and thus satisfy the conclusion of \cref{prop:structureunpred}) admit triangular systems with many constraints.

For the rest of this section, we use the following notation.

\begin{notation}
    Let $R\geq 1$ and $\d \in (\pm \D)^{R}$. For $I\subset \sinterval{R}$, we write $\d_{I} := \prod_{i\in I} d_i$.
\end{notation}

\begin{definition}[$\red{R, \s, \lit}$]
    \label{def:reducedset}
    Let $R\geq 1$ and $\s, \lit$ be disjoint subsets of $\sinterval{R} \times \intJ$.

    We define $\red{R, \s, \lit}$ to be the set of all $\lit$-admissible reduced tuples\footnote{See \cref{def:reduced} and \cref{def:admissible} for the definitions of reduced tuples and $\lit$-admissibility.} $\d \in (\pm \D)^{R}$ such that:
    \begin{enumerate}
        \item \label{item:redset1} For all $(i,j)\in \s$, we have $\dsub{i}{j}^2 \nmid \d_{\sinterval{R}}$.
        \item \label{item:redset2} Whenever two indices $(i,j), (i', j)\in \lit$ satisfy $\dsub{i}{j} = \dsub{i'}{j}$, we have
              \begin{equation*}
                  \dsub{i}{j} \bigmid \sum_{i<k<i'} d_k.
              \end{equation*}
    \end{enumerate}
\end{definition}

\begin{definition}
    \label{def:divtriple}
    Let $\d \in \red{R, \s, \lit}$. A \emph{divisibility triple} for $\d$ is a triple $(x,y,p) \in \sinterval{R}^2 \times \primes$ such that $p\mid d_x$ and $p\mid d_y$, and there exists $x<i<y$ with $p\nmid d_i$.\footnote{In particular, $y\geq x+2$.}
\end{definition}

\begin{lemma}
    \label{lem:divtriple}
    Let $1\leq R< H_0$ and $\s, \lit$ be disjoint subsets of $\sinterval{R} \times \intJ$. Let $\d \in \red{R, \s, \lit}$.

    Let $(x,y,p)$ be a divisibility triple for $\d$. Then, one of the following holds.
    \begin{enumerate}
        \item \label{item:minimaldivtrip1} There are $x',y'\in \interval{x}{y}$ and $p'\in \primes$ such that $(x',y',p')$ is a divisibility triple for $\d$ and $\abs{y'-x'} < \abs{y-x}$.
        \item \label{item:minimaldivtrip2} There exists $q\in \primes$ dividing $d_{y-1}$, such that
              \begin{equation*}
                  \sum_{\substack{x<i<y\\ q\mid d_{i}}} d_{i} \not\equiv 0 \pmod p.
              \end{equation*}
    \end{enumerate}
\end{lemma}

\begin{proof}
    Suppose that case \cref{item:minimaldivtrip1} does not hold. In particular, this implies that $p\nmid d_{y-1}$.

    Let $k$ be the smallest integer with $x<k<y$ such that
    \begin{equation*}
        \abs{d_k} = \abs{d_{k+1}} = \ldots = \abs{d_{y-1}}.
    \end{equation*}
    Since $\d$ is reduced, we in fact have $d_k = d_{k+1} = \ldots = d_{y-1}$.

    If $k=x+1$, we can directly compute, for any prime $q\mid d_{y-1}$,
    \begin{equation*}
        \sum_{\substack{x<i<y\\ q\mid d_i}} d_{i} = (y-x-1) d_{y-1},
    \end{equation*}
    which is not divisible by $p$ since $p\nmid d_{y-1}$ and $0<y-x-1 \leq R < H_0 \leq p$. Hence, case \cref{item:minimaldivtrip2} holds.

    Suppose now that $k > x+1$. Then, since $\abs{d_{k-1}} \neq \abs{d_k}$, there is a prime $q$ dividing $d_k = d_{y-1}$ but not $d_{k-1}$. Note that $q$ cannot divide any $d_i$ with $x<i<k-1$, since otherwise $(i,k,q)$ would be a divisibility triple satisfying case \cref{item:minimaldivtrip1}. Therefore,
    \begin{equation*}
        \sum_{\substack{x<i<y\\ q\mid d_i}} d_{i} = (y-k) d_{y-1}.
    \end{equation*}
    As in the previous case, this is not divisible by $p$, so case \cref{item:minimaldivtrip2} holds.
\end{proof}

The following lemma corresponds to the tuples $\d$ that satisfy the first conclusion of \cref{prop:structureunpred}, i.e.~tuples having an associated word with $m$ separated repetitions.

\begin{lemma}
    \label{lem:systemfromrepetitions}
    Let $1\leq R < H_0$ and let $\s,\lit$ be disjoint subsets of $\sinterval{R}\times \intJ$.

    Let $\d \in \red{R, \s, \lit}$. Suppose that, for some $m \geq 4R/\L$, there exist
    \begin{equation*}
        1\leq x_1 < y_1 < x_2 < y_2 < \ldots < x_m < y_m \leq R
    \end{equation*}
    and $p_1, \ldots, p_m\in \primes$ such that $(x_i,y_i,p_i)$ is a divisibility triple for $\d$, for all $i\in \sinterval{m}$.

    Assume that $\s \sqcup \lit \supset \interval{x_1}{y_m}\times \intJ$.

    Then, $\d$ satisfies a triangular system of $\gg m$ constraints with shift $\kappa = 0$.
\end{lemma}

\begin{proof}
    Without loss of generality, we may assume that, for each $i\in \sinterval{m}$, the divisibility triple $(x_i,y_i,p_i)$ is minimal, in the sense that there is no divisibility triple $(x',y',p')$ with $x_i\leq x'<y'\leq y_i$ and $\abs{y'-x'}<\abs{y_i-x_i}$.

    Thus, by \cref{lem:divtriple}, for each $i\in \sinterval{m}$, there is a prime $q_i \in \primes$ dividing $d_{y_i-1}$ such that
    \begin{equation}
        \label{eq:condnonzero2}
        \sum_{\substack{x_i<k<y_i\\ q_i\mid d_{k}}} d_{k} \not\equiv 0 \pmod {p_i}.
    \end{equation}

    Each $(x_i,y_i,p_i)$ is a divisibility triple for $\d$, which implies that $p_i = \dsub{x_i}{j_i} = \dsub{y_i}{j_i}$ for some $j_i\in \intJ$. Since every element of $\interval{x_1}{y_m}\times \intJ$ is either in $\s$ or in $\lit$, and all elements of $\s$ correspond to primes that appear only once, it must be that $(x_i,j_i), (y_i,j_i)\in \lit$. Therefore, by \cref{def:reducedset}, we have
    \begin{equation}
        \label{eq:theconstraintinrep}
        p_i = \dsub{x_i}{j_i} \bigmid\, \sum_{x_i<k<y_i} d_k
    \end{equation}
    for every $i\in \sinterval{m}$.

    The divisibility relation \cref{eq:theconstraintinrep} defines a constraint $C_i$ on $\d$ with shift $\kappa = 0$ (see \cref{def:constraints}). Moreover, $q_i$ is involved in this constraint by \cref{eq:condnonzero2} (see \cref{def:involved}). We now perform a case analysis to extract a triangular system of constraints.

    Let $I_0$ be the set of all $i \in \sinterval{m}$ such that $q_i$ does not divide $\prod_{k<i} \d_{\interval{x_{k}}{y_{k}}}$.

    Suppose that $\abs{I_0}\geq m/4$. If $i \in I_0$, we know that $q_i$ is absent from the constraints $C_{k}$ with $k<i$. Therefore, $(C_i)_{i \in I_0}$ forms a triangular system of $\gg m$ constraints satisfied by $\d$, so we are done. Henceforth, we assume that $\abs{I_0}< m/4$.

    Let $I_1$ be the set of all $i \in \sinterval{m}$ such that $\interval{x_i+1}{y_i-1}\times \intJ$ contains a pair $(s_i, t_i)\in \s$.

    Suppose that $\abs{I_1}\geq m/4$. Notice that $\dsub{s_i}{t_i}$ is involved in the constraint $C_i$, as
    \begin{equation*}
        \sum_{\substack{x_i<k<y_i\\ \dsub{s_i}{t_i}\mid d_{k}}} d_{k} = d_{s_i} \not\equiv 0 \pmod {p_i}.
    \end{equation*}
    Here we used that $(s_i, t_i)\in \s$ for the first equality and the minimality of $(x_i,y_i,p_i)$ to say that $p_i \nmid d_{s_i}$. In addition, $\dsub{s_i}{t_i}$ is absent from all other constraints $(C_{i'})_{i'\neq i}$, since $(s_i, t_i)\in \s$. Thus, $(C_i)_{i \in I_1}$ is a triangular system of $\gg m$ constraints satisfied by $\d$, as desired. We now assume that $\abs{I_1} < m/4$.

    Let $I_2$ be the set of all $i \in \sinterval{m}$ such that $\{x_i\}\times \intJ$ contains a pair $(s_i, t_i)\in \s$ (thus $s_i = x_i$).

    Suppose that $\abs{I_2} \geq m/2$. Then $\abs{I_2\setminus I_0} \geq m/4$. This time, we will use a different sequence of constraints.  Let $i \in I_2\setminus I_0$. By definition of $I_0$, we know that there exists
    \begin{equation*}
        u_i \in \bigcup_{i'<i} \interval{x_{i'}}{y_{i'}}
    \end{equation*}
    such that $q_i \mid d_{u_i}$. We also know that $q_i\mid d_{v_i}$ for some $x_i<v_i<y_i$ by \cref{eq:condnonzero2}. Hence, by \cref{def:reducedset}, we obtain
    \begin{equation*}
        q_i \bigmid\,  \sum_{u_i<k<v_i} d_k,
    \end{equation*}
    which is a constraint on $\d$ (with shift $\kappa = 0$) that we call $D_i$. Since
    \begin{equation}
        \label{eq:orderinginequal}
        u_i<x_i=s_i<v_i<y_i
    \end{equation}
    and $(s_i, t_i)\in \s$, we have
    \begin{equation}
        \label{eq:somesitisum}
        \sum_{\substack{u_i < k < v_i \\ \dsub{s_i}{t_i} \mid d_k}} d_k = d_{s_i}.
    \end{equation}
    We claim that $d_{s_i} = d_{x_i}$ is not divisible by $q_i$, which by \cref{eq:somesitisum} implies that $\dsub{s_i}{t_i}$ is involved in the constraint $D_i$. Indeed, if $q_i$ were to divide $d_{x_i}$, recalling that $q_i$ divides $d_{y_i-1}$, we would have $q_i \mid d_k$ for all $x_i \leq k \leq y_i-1$ by minimality of the divisibility triple $(x_i,y_i,p_i)$. This is incompatible with \cref{eq:condnonzero2} and \cref{eq:theconstraintinrep}.

    Therefore, $\dsub{s_i}{t_i}$ is involved in $D_i$ for all $i \in I_2\setminus I_0$. Moreover, for all $i',i\in I_2\setminus I_0$ with $i'<i$, equation~\cref{eq:orderinginequal} and the fact that $(s_{i}, t_{i})\in \s$ show that $\dsub{s_{i}}{t_{i}}$ is absent from $D_{i'}$. Thus, $(D_i)_{i\in I_2\setminus I_0}$ is a triangular system of $\gg m$ constraints satisfied by $\d$. We may assume henceforth that $\abs{I_2} < m/2$.

    We have reached the final case of the proof. Let $I_3 = \sinterval{m} \setminus (I_1\cup I_2)$, so that $\abs{I_3} \geq m/4$. For $i \in I_3$, by definition of $I_1$ and $I_2$, the set $\interval{x_i}{y_i-1}\times \intJ$ has empty intersection with $\s$, so that
    \begin{equation}
        \label{eq:subsetoflit}
        \interval{x_i}{y_i-1}\times \intJ\subset \lit.
    \end{equation}

    By the pigeonhole principle, since $\abs{I_3} \geq m/4 \geq R/\L$, there exists $i\in I_3$ such that
    \begin{equation}
        \label{eq:prohiblength}
        \abs{y_i-x_i} \leq \L.
    \end{equation}
    We claim that $(d_{x_i}, d_{x_i+1}, \ldots, d_{y_i-1})$ is a prohibited sequence. Indeed, it satisfies all the assumptions of \cref{def:prohibitedsequence}:
    \begin{itemize}
        \item $(d_{x_i}, d_{x_i+1}, \ldots, d_{y_i-1})$ is reduced as $\d\in \red{R,\s,\lit}$;
        \item for every prime $q$, the set $\{i\in \interval{x_i}{y_i-1} : q\mid d_i\}$ is a discrete interval by minimality of the divisibility triple $(x_i,y_i,p_i)$;
        \item $p_i$ divides $d_{x_i}$ and satisfies the divisibility condition \cref{eq:theconstraintinrep};
        \item $p_i$ does not divide $d_{y_i-1}$ by minimality of $(x_i,y_i,p_i)$ again;
        \item the tuple $(d_{x_i}, d_{x_i+1}, \ldots, d_{y_i-1})$ has length $\abs{y_i-x_i} \leq \L$ by \cref{eq:prohiblength}.\footnote{Note that the condition $\ell > 2$ in the definition of a prohibited sequence is automatically satisfied.}
    \end{itemize}
    By \cref{eq:subsetoflit}, this contradicts the assumption that $\d$ is $\lit$-admissible (see \cref{def:admissible,def:reducedset}). Thus, this final case cannot occur, which concludes the proof.
\end{proof}

The next lemma deals with tuples $\d$ having an associated word that satisfies the second conclusion of \cref{prop:structureunpred}.

\begin{lemma}
    \label{lem:systemfromotherpattern}
    Let $1\leq R < H_0$ and let $\s,\lit$ be disjoint subsets of $\sinterval{R}\times \intJ$.

    Let $\d \in \red{R, \s, \lit}$ and let $j\in \intJ$. Suppose that, for some $m \geq 16 R^2/\L^2$, there exist
    \begin{equation*}
        1\leq x_0 < z_1 < x_1 < z_2 < x_2 <\ldots < x_{m-1} < z_m < x_m < y_0 < y_1 < y_2 < \ldots < y_m \leq R
    \end{equation*}
    and $p_0, \ldots, p_m\in \primes_j$ such that $(x_i,y_i,p_i)$ is a divisibility triple for $\d$, for all $0\leq i \leq m$.

    Define $q_i := \dsub{z_i}{j}$ for $i\in \sinterval{m}$. Suppose that $q_i \nmid \d_{\interval{x_0}{x_{i-1}}}\d_{\interval{y_0}{y_i}}$ and $p_i \nmid \d_{\interval{x_0}{z_i}}$ for all $i\in \sinterval{m}$.

    Finally, assume that $\s \sqcup \lit \supset \big(\interval{x_0}{x_m}\cup \interval{y_0}{y_m} \big)\times \intJ$.

    Then, the concatenation\footnote{We work with the concatenation of $\d$ and $-\d$ to allow for negative signs in the constraints. The reason for this will be apparent in the proof.} of $\d$ and $-\d$ satisfies a triangular system of $\gg m^{1/2}$ constraints with shift~$\kappa = O(RH)$.
\end{lemma}

\begin{proof}
    Call an integer $i\in \sinterval{m}$ \emph{unsuitable} if $\abs{x_i-x_{i-1}}\geq \L$, $\abs{y_i-y_{i-1}}\geq \L$, or there exists a divisibility triple $(x,y,p)$ with $x_{i-1}<x<y\leq x_i$. Otherwise, we call $i$ \emph{suitable}.

    The first two scenarios can only happen for $\leq R/\L \leq m^{1/2}$ values of $i \in \sinterval{m}$. If the third one applies to $\geq m^{1/2} \geq 4R/\L$ values of $i$, then we are done by \cref{lem:systemfromrepetitions}. Therefore, we can assume that there are at most $O(m^{1/2})$ unsuitable integers $i\in \sinterval{m}$.

    We will henceforth work with a subinterval of $\sinterval{m}$ containing only suitable integers: there exists such a subinterval $\interval{m_1}{m_2}\subset \sinterval{m}$ with $m_2-m_1 \gg m^{1/2}$.

    Let $i\in \interval{m_1}{m_2}$. Since $p_i = \dsub{x_i}{j} = \dsub{y_i}{j}$, we have $(x_i,j),(y_i,j)\not\in \s$, and thus $(x_i,j),(y_i,j)\in \lit$. Therefore, by \cref{def:reducedset}, we have
    \begin{equation*}
        p_i  \bigmid \, \sum_{x_i< k <y_i} d_k.
    \end{equation*}
    We rewrite this as
    \begin{equation}
        \label{eq:lastconstraint}
        p_i = \dsub{x_i}{j} \bigmid \, \kappa +\sum_{x_0\leq k <x_i}-d_k + \sum_{y_0< k <y_i}d_k,
    \end{equation}
    with $\kappa := \sum_{x_0\leq z\leq y_0} d_z$. This is a constraint with shift $\kappa$ on the concatenation of $\d$ and $-\d$; we call this constraint $C_i$. Note that the same shift $\kappa$ is used for all constraints $C_i$, and it satisfies $\abs{\kappa}\leq R H$. We will show that an appropriate subset of these constraints forms a triangular system.

    We define $I$ to be the set of all $i\in \interval{m_1}{m_2}$ such that $\big(\interval{x_{i-1}}{x_i}\times \intJ\big)\cap \s \neq \emptyset$.

    Suppose first that $\abs{I}\geq \tfrac12 (m_2-m_1)$. Then one of the sets
    \begin{equation*}
        I_0 := \{i\in I : i\equiv 0~\spmod 2\} \quad\text{and}\quad I_1 := \{i\in I : i\equiv 1~\spmod 2\}
    \end{equation*}
    has size $\gg m^{1/2}$. Without loss of generality, suppose that $\abs{I_0}\gg m^{1/2}$. For each $i\in I_0$, there is some $(s_i, t_i)\in \s$ such that $s_i \in \interval{x_{i-1}}{x_i}$. We claim that the constraints $(C_{i+1})_{i\in I_0\setminus \{m\}}$ form a triangular system. Indeed, the prime $\dsub{s_i}{t_i}$ is involved in $C_{i+1}$ as
    \begin{equation*}
        \sum_{\substack{x_0\leq k<x_{i+1}\\ \dsub{s_i}{t_i} \mid d_k}}-d_k + \sum_{\substack{y_0< k<y_{i+1}\\ \dsub{s_i}{t_i} \mid d_k}}d_k= -d_{s_i} \not\equiv 0\pmod{p_{i+1}},
    \end{equation*}
    using that $(s_i, t_i)\in \s$, and the assumption $p_{i+1} \nmid \d_{\interval{x_0}{z_{i+1}}}$ in the statement. Furthermore, $\dsub{s_i}{t_i}$ is absent from $C_{i'+1}$ for all even $i'<i$, as for such an $i'$ we have $x_{i'+1} \leq  s_i < y_0$. Thus, the concatenation of $\d$ and $-\d$ satisfies a triangular system of $\gg m^{1/2}$ constraints with shift $\kappa$, as required.

    Hence, we may assume that $\abs{I} < \tfrac12(m_2-m_1)$. Let $i\in \interval{m_1}{m_2}\setminus I$. We will now use the integers~$z_i$ from the statement. We claim that the prime $q_i = \dsub{z_i}{j}$ is involved in the constraint $C_i$ defined by~\cref{eq:lastconstraint}.

    Suppose for contradiction that $q_i$ is not involved in $C_i$. Recalling that $q_i \nmid \d_{\interval{x_0}{x_{i-1}}}\d_{\interval{y_0}{y_i}}$ by assumption, this means that
    \begin{equation}
        \label{eq:finalprohibitedpattern}
        \sum_{\substack{x_{i-1}<k<x_i\\ q_i\mid d_k}}d_k \equiv 0 \pmod{p_i}.
    \end{equation}
    Since $i$ is suitable, the set $\{x_{i-1}<k< x_i: q_i\mid d_k\}$ is a discrete interval, say $\interval{r}{t}$ (otherwise, there would be a divisibility triple $(x,y,q_i)$ for some $x_{i-1}<x<y< x_i$).

    We claim that the vector $(d_{x_i}, d_{x_i-1}, d_{x_i-2},\ldots, d_{r})$ is a prohibited sequence. Indeed, it satisfies all the assumptions of \cref{def:prohibitedsequence}:
    \begin{itemize}
        \item $(d_{x_i}, d_{x_i-1}, d_{x_i-2},\ldots, d_{r})$ is reduced as $\d\in \red{R,\s,\lit}$;
        \item for every prime $q$, the set $\{i\in \interval{r}{x_i} : q\mid d_i\}$ is a discrete interval, since there is no divisibility triple $(x,y,p)$ with $x_{i-1}<x<y\leq x_i$ as $i$ is suitable;
        \item $p_i$ divides $d_{x_i}$ and satisfies the divisibility condition
              \begin{equation*}
                  p_i \bigmid \, \sum_{\substack{x_{i-1}<k<x_i\\ q_i\mid d_k}}d_k  = \sum_{r\leq k\leq t} d_k
              \end{equation*}
              by \cref{eq:finalprohibitedpattern};
        \item $p_i$ does not divide $d_{r}$, as otherwise $p_i$ would equal $\dsub{r}{j} = \dsub{z_i}{j} = q_i$, contradicting the assumption $p_i \nmid \d_{\interval{x_0}{z_i}}$ in the statement;
        \item the tuple $(d_{x_i}, d_{x_i-1}, d_{x_i-2},\ldots, d_{r})$ has length $\abs{x_i-r+1}\leq \abs{x_i-x_{i-1}}< \L$ as $i$ is suitable.
    \end{itemize}
    However, since $i\notin I$, we have
    \begin{equation*}
        \interval{x_{i-1}}{x_i}\times \intJ\subset \lit.
    \end{equation*}
    This contradicts the $\lit$-admissibility of $\d$.

    As a consequence, we deduce that $q_i$ is involved in the constraint $C_i$. Since $q_i \nmid \d_{\interval{x_0}{x_{i-1}}}\d_{\interval{y_0}{y_i}}$ by assumption, it follows that $q_i = \dsub{z_i}{j}$ is absent from $C_{i'}$ for all $i'<i$. Therefore, the concatenation of $\d$ and $-\d$ satisfies the triangular system $(C_i)_{i\in \interval{m_1}{m_2}\setminus I}$, which is made of $\gg m^{1/2}$ constraints with shift $\kappa = O(RH)$. This concludes the proof.
\end{proof}

We also state a variant of \cref{lem:systemfromotherpattern} where the $y_i$ come in the opposite order. The proof is completely analogous and left to the reader.

\begin{lemma}
    \label{lem:systemfromotherpatternbis}
    Let $1\leq R < H_0$ and let $\s,\lit$ be disjoint subsets of $\sinterval{R}\times \intJ$.

    Let $\d \in \red{R, \s, \lit}$ and let $j\in \intJ$. Suppose that, for some $m \geq 16 R^2/\L^2$, there exist
    \begin{equation*}
        1\leq x_0 < z_1 < x_1 < z_2 < x_2 <\ldots < x_{m-1} < z_m < x_m < y_m < y_{m-1} < \ldots < y_0 \leq R
    \end{equation*}
    and $p_0, \ldots, p_m\in \primes_j$ such that $(x_i,y_i,p_i)$ is a divisibility triple for $\d$, for all $0\leq i \leq m$.

    Define $q_i := \dsub{z_i}{j}$ for $i\in \sinterval{m}$. Suppose that $q_i \nmid \d_{\interval{x_0}{x_{i-1}}}\d_{\interval{y_i}{y_0}}$ and $p_i \nmid \d_{\interval{x_0}{z_i}}$ for all $i\in \sinterval{m}$.

    Finally, assume that $\s \sqcup \lit \supset \big(\interval{x_0}{x_m}\cup \interval{y_m}{y_0} \big)\times \intJ$.

    Then, the concatenation of $\d$ and $-\d$ satisfies a triangular system of $\gg m^{1/2}$ constraints with shift~$\kappa = O(RH)$.
\end{lemma}

\subsection{Contribution of unpredictable walks}
We can finally prove \cref{prop:unpredictable}, which we restate here for convenience.

\unpredcontrib*

\begin{proof}[Proof of \cref{prop:unpredictable}]
    Let $\d\in \unpred$. By \cref{def:predunpred}, there exist $j\in \intJ$ and $k \in \{1,2,3\}$ such that the word $w_{j,\d}^{(k)}$ introduced in \cref{def:primewords} is $\K^{1/4}$-unpredictable. Let $\d^{(k)}$, $I^{(k)}$, $\iota^{(k)}$ and $\lit^{(k)}$ be as in \cref{def:D}, and define
    \begin{equation*}
        \s^{(k)} := \{(\iota^{(k)}(i), j) \mid (i,j)\in \s\} \quad\text{and}\quad \unlit^{(k)} := \{(\iota^{(k)}(i), j) \mid (i,j)\in \unlit\}.
    \end{equation*}

    We claim that $\d^{(k)}\in \red{R, \s^{(k)}, \lit^{(k)}}$ where $R := |I^{(k)}|$. To prove this, we need to check the properties of \cref{def:reducedset}. The fact that $\d^{(k)}$ is reduced and $\lit^{(k)}$-admissible follows from part \cref{item:Dprop4} of \cref{def:D}. Property \cref{item:redset1} of \cref{def:reducedset} follows from the fact that $\s$ is the set of single indices for $\d$. Finally, property \cref{item:redset2} of \cref{def:reducedset} follows from part \cref{item:Dprop2} of \cref{def:D}, together with the identity $\sum_{\K-r<i\leq \K+r} d_i = 0$ in the case $k=3$ (which follows from part \cref{item:Dprop3} of \cref{def:D}).

    Therefore, $\d^{(k)}\in \red{R, \s^{(k)}, \lit^{(k)}}$, which will enable us to apply \cref{lem:systemfromrepetitions}, \cref{lem:systemfromotherpattern} or \cref{lem:systemfromotherpatternbis} to $\d^{(k)}$, depending on the structure of the word $w_{j,\d}^{(k)}$, to exhibit a triangular system of constraints.

    By \cref{prop:structureunpred}, since $w_{j,\d}^{(k)}$ is $\K^{1/4}$-unpredictable, there exists $m \gg \K^{c_2}$ for some absolute constant $c_2>0$, such that one of the conclusions of \cref{prop:structureunpred} holds with this $m$.

    The first possibility is that the word $w_{j,\d}^{(k)}$ has $m$ separated repetitions (see \cref{def:repetitions}). Since $w_{j,\d}^{(k)}$ is the compression of the word $v_{j,\d}^{(k)}$ (see \cref{def:primewords}), this means that there are
    \begin{equation*}
        1\leq x_1 < y_1 < x_2 < y_2 < \ldots < x_m < y_m \leq R
    \end{equation*}
    such that $(x_i,y_i,p_i)$ is a divisibility triple for $\d^{(k)}$ for all $i\in \sinterval{m}$, where $p_i$ is the unique prime in $\primes_j$ dividing the $x_i$-th coordinate of $\d^{(k)}$. The assumptions of \cref{lem:systemfromrepetitions} are satisfied, except possibly for the condition that $\interval{x_1}{y_m}\times \intJ$ does not intersect $\unlit^{(k)}$. Since $\abs{\unlit^{(k)}}\leq \abs{\unlit} \leq \K^{2\epsone}$, $m \gg \K^{c_2}$ and $\epsone$ is sufficiently small, we can find a discrete interval $\interval{a}{b} \subset \sinterval{m}$ with $b-a \gg m^{1/2}$ such that $\interval{x_a}{y_b}\times \intJ$ does not intersect $\unlit^{(k)}$. Then, applying \cref{lem:systemfromrepetitions} to $\d^{(k)}$ with the divisibility triples $(x_i,y_i,p_i)_{a\leq i\leq b}$, we conclude that $\d^{(k)}$ satisfies a triangular system of $\gg \K^{c_2/2}$ constraints with shift $\kappa = 0$.

    Otherwise, the second conclusion of \cref{prop:structureunpred} holds. Unpacking the notation, this means that either $\d^{(k)}$ or its reverse admits divisibility triples arranged as in \cref{lem:systemfromotherpattern} or \cref{lem:systemfromotherpatternbis}. As in the previous paragraph, the condition that $\big(\interval{x_0}{x_m}\cup \interval{y_0}{y_m} \big)\times \intJ$ (for \cref{lem:systemfromotherpattern}) or ${\big(\interval{x_0}{x_m}\cup \interval{y_m}{y_0} \big)\times \intJ}$ (for \cref{lem:systemfromotherpatternbis}) is disjoint from $\unlit^{(k)}$ may not be initially satisfied, but it can be guaranteed by passing to a consecutive subsequence of $\gg m^{1/2}$ divisibility triples. Then by \cref{lem:systemfromotherpattern} or \cref{lem:systemfromotherpatternbis}, we conclude that the concatenation of $\d^{(k)}$, $-\d^{(k)}$ and their reverses,\footnote{Including the reverses of $\d^{(k)}$ and $-\d^{(k)}$ is not strictly necessary.} satisfies a triangular system of $\gg \K^{c_2/4}$ constraints with shift $\kappa = O(\K H)$.

    In summary, we have shown that, for every $\d\in \unpred$, the concatenation of all $\d^{(k)}$, $-\d^{(k)}$ and their reverses satisfies a triangular system of $\gg \K^{c_2/4}$ constraints with shift $\kappa = O(\K H)$.

    By \cref{lem:constraints} (and since every $\d$ is uniquely determined by $\d^{(1)}$ and $\d^{(2)}$), we obtain the bound
    \begin{equation*}
        \sum_{\substack{\d\in \unpred}}\,\prod_{p\mid \pall}\, \frac{1}{p} \ll \K H \K^{O(\J\K)} H_0^{-c \K^{c'}} \ll \K^{O(\K\log \K)} H_0^{-c \K^{c'}}
    \end{equation*}
    for some absolute constants $c,c'>0$. Since ${\log H_0 \gg\K^{1-\epsone}}$ and $\epsone$ is assumed to be sufficiently small, this is $\ll 1$, as desired.
\end{proof}

\section{Prohibited progressions and the combinatorial sieve}
\label{sec:combinatorial}

In this final section, we establish several results on prohibited sequences and their associated progressions. We first bound the size of $\N\setminus \Y$ in \cref{lem:Ysmall}. We also prove \cref{lem:existencerank}, which constructs a suitable cut-off function for the combinatorial sieve and was stated without proof in \cref{sec:single}.

We conclude the paper by proving the high trace bound for the non-backtracking operator of $G_{\Y}$ (\cref{prop:hightracebound}), bringing together the results from \cref{sec:trace} onwards.

\subsection{Primitive prohibited sequences}
\label{sec:primitiveanalysis}
In this section, we prove a technical lemma that allows us to find constraints and involved primes in primitive prohibited sequences.

The divisibility condition in the definition of prohibited sequences (see \cref{def:prohibitedsequence}) only brings up a subset of the prime factors of the $d_i$. Even the primes that do appear in that constraint might not be involved in the sense of \cref{def:involved}. \Cref{lem:prohibitedanalysis} is a useful tool to circumvent this problem: it allows us to pass from an arbitrary prime to a (possibly different) involved prime.

\begin{lemma}
    \label{lem:prohibitedanalysis}
    Let $\d = (d_1, \ldots, d_{\ell})$ be a primitive prohibited sequence. Let $\Gamma$ be the set of all constraints with shift $\kappa = 0$ satisfied by $\d$.

    For every prime $p\mid d_1\cdots d_{\ell}$, one of the following holds:
    \begin{enumerate}
        \item \label{item:prohibitedanalysis1} $p$ is involved in a constraint $C\in \Gamma$;
        \item \label{item:prohibitedanalysis2} there is another prime $q$ involved in a constraint of $\,\Gamma$, such that $q\mid d_k$ for some $k\in \sinterval{\ell}$, and
              \begin{equation*}
                  \sum_{\substack{1\leq i <k\\ p\mid d_i}} d_i \not\equiv 0 \pmod q.
              \end{equation*}
    \end{enumerate}
\end{lemma}

\begin{proof}
    By definition of a prohibited sequence (\cref{def:prohibitedsequence}), there are $1<\ell_0<\ell$ and $j_0\in \intJ$ such that $\dsub{1}{j_0}\nmid d_{\ell}$ and
    \begin{equation}
        \label{eq:firstconstraint}
        \dsub{1}{j_0} \bigmid \sum_{\ell_0\leq i\leq \ell} d_i.
    \end{equation}
    Observe that this is a constraint with shift $\kappa = 0$ satisfied by $\d$; we denote it by $C_1$.

    Let $h\in \sinterval{\ell}$ be the smallest positive integer such that
    \begin{equation*}
        \abs{d_h} = \abs{d_{h+1}} = \ldots = \abs{d_{\ell}}.
    \end{equation*}
    Since prohibited sequences are reduced, we must in fact have $d_h = d_{h+1} = \ldots = d_{\ell}$. Note that $h\geq 2$ since $\dsub{1}{j_0}\nmid d_{\ell}$.

    Let $p_1\in \primes$ be any prime divisor of $d_h$ but not $d_{h-1}$. We claim that $p_1$ is involved in $C_1$. Indeed,
    \begin{equation}
        \label{eq:auxiliaryinvolved}
        \sum_{\substack{i \in \interval{\ell_0}{\ell}\\ p_1\mid d_i}} d_i = \sum_{i\in \interval{h}{\ell}} d_i = (\ell - h + 1) d_{\ell},
    \end{equation}
    which is not divisible by $\dsub{1}{j_0}$ since $\ell - h + 1 < L < H_0 < \dsub{1}{j_0}$. To obtain the first equality in \cref{eq:auxiliaryinvolved}, we used that, for any prime $p'\in \primes$, the set
    \begin{equation*}
        I(p') := \{i\in \sinterval{{\ell}}: p'\mid d_i\}
    \end{equation*}
    is a discrete interval, by definition of a prohibited sequence.

    We are now ready to start the proof of \cref{lem:prohibitedanalysis} in earnest. Let $p\mid d_1\cdots d_{\ell}$ be a prime.

    We can assume that $I(p) \setminus \interval{h}{\ell}$ is non-empty: otherwise, $p$ divides $d_h$ but not $d_{h-1}$, which implies that $p$ is involved in $C_1$ by the previous paragraph, and we are in case \cref{item:prohibitedanalysis1}.

    Hence, $I(p) \setminus \interval{h}{\ell}$ is a non-empty discrete interval, which we denote by $I(p) \setminus \interval{h}{\ell} = \interval{a_1}{a_2}$.

    Suppose first that
    \begin{equation}
        \label{eq:prohibassumption1}
        \sum_{\substack{i\in \sinterval{\ell}\\ p\mid d_i}} d_i \not\equiv 0 \pmod{p_1}.
    \end{equation}
    Then, we are immediately done, as case \cref{item:prohibitedanalysis2} holds with $q = p_1$ and $k = h$.

    Suppose now that \cref{eq:prohibassumption1} fails, which means that
    \begin{equation}
        \label{eq:prohibassumption2}
        p_1 \bigmid\,  \sum_{i\in \interval{a_1}{a_2}} d_i.
    \end{equation}
    This divisibility condition \cref{eq:prohibassumption2} implies that $\big(d_{h}, d_{h-1}, \ldots, d_{a_1+1}, d_{a_1}\big)$ is a prohibited sequence. Since~$\d$ is primitive, this is only possible if $a_1=1$ and $h=\ell$. In particular, $a_2\geq 2$ (otherwise, we would have $p_1\mid d_1$ by \cref{eq:prohibassumption2}).

    We will now exhibit another prime $p_2$ for which case \cref{item:prohibitedanalysis2} of the lemma holds with $q = p_2$.

    Let $p_2$ be the prime $\dsub{2}{j_0}$. Note that $p_2 \nmid d_1$, or else we would have $p_2 = \dsub{1}{j_0}$, and thus ${p_2 \mid \sum_{\ell_0\leq i\leq \ell} d_i}$ by \cref{eq:firstconstraint}. This would imply that $(d_{2}, d_{3}, \ldots, d_{\ell})$ is a prohibited sequence, which is impossible since $\d$ is primitive.

    We claim that $p_2$ is involved in the constraint $C_2$ defined by the divisibility relation \cref{eq:prohibassumption2}. Suppose for contradiction that this is not the case. This means that
    \begin{equation}
        \label{eq:uxtfgt}
        p_1 \bigmid\, \sum_{\substack{i\in \interval{1}{a_2}\\ p_2 \mid d_i}} d_i = \sum_{\substack{i\in \interval{2}{a_2}\\ p_2 \mid d_i}} d_i,
    \end{equation}
    using that $p_2 \nmid d_1$. Note that $\{i\in \interval{2}{a_2}: p_2 \mid d_i\}$ is a discrete interval containing $2$ and not containing $h$. Thus, \cref{eq:uxtfgt} implies that $\big(d_{h}, d_{h-1}, \ldots, d_{2}\big)$ is a prohibited sequence, contradicting that $\d$ is primitive. Hence, $p_2$ is involved in $C_2$.

    To check that case \cref{item:prohibitedanalysis2} holds with $q = p_2$, it only remains to verify that
    \begin{equation*}
        \sum_{\substack{1\leq i <k\\ p\mid d_i}} d_i \not\equiv 0 \pmod{p_2}
    \end{equation*}
    for some index $k$ such that $p_2\mid d_k$. It suffices to take $k=2$.
\end{proof}

We can use the previous lemma (in fact, a much weaker version would suffice) to bound $\abs{\N\setminus \Y}$.

\begin{lemma}
    \label{lem:Ysmall}
    Whenever $\log N \geq (\log H)^{2}$, we have $\abs{\N \setminus \Y} \ll H_0^{-1/3}N$.
\end{lemma}

\begin{proof}
    Recall that $\Z \setminus \Y$ is the union of all prohibited progressions $P\in \prohibprog$. For any $P\in \prohibprog$, there is a primitive prohibited sequence $\d$ of length at most $\L = \K^{1-10\epsone}$ such that $P$ is the prohibited progression associated to $\d$. In particular, the modulus of $P$ satisfies
    \begin{equation*}
        \modulus_{P} \leq H^{\L} \ll \exp\!\big(O\big(\K^{2-10\epsone} \big)\big) \ll N,
    \end{equation*}
    using that $\log N \gg \K^{2}$. By the union bound, we deduce that
    \begin{equation*}
        \abs{\N \setminus \Y} \leq \sum_{P\in \prohibprog} \abs{\N \cap P} \ll N \sum_{P\in \prohibprog} \frac{1}{\modulus_{P}}.
    \end{equation*}

    By \cref{lem:prohibitedanalysis}, there exists a constraint $C$ satisfied by $\d$, which has zero shift and involves at least one prime. This constraint alone can be viewed as a triangular system. Applying \cref{lem:constraints} with $T=1$ and $R = \ell$ for each $2<\ell\leq \L$, we obtain
    \begin{equation*}
        \sum_{P\in \prohibprog} \frac{1}{\modulus_{P}} \leq \sum_{2<\ell\leq \L}\, \sum_{\d\in \XX_{\ell}}\, \prod_{p\mid d_1\cdots d_{\ell}} \frac{1}{p} \ll \L \K^{O(\J \L)} H_0^{-1/2} \ll \exp\!\big(O\big(\K^{1-10\epsone} (\log \K)^2\big)\big) H_0^{-1/2},
    \end{equation*}
    where $\XX_{\ell}$ denotes the set of all primitive prohibited sequences of length $\ell$. The conclusion follows as ${\log H_0 \gg\K^{1-\epsone}}$.
\end{proof}

\subsection{Cut-off function for the combinatorial sieve}
\label{sec:cutoffanalysis}
We turn to the proof of \cref{lem:existencerank}.

Recall that $\interprog$ is the set defined in \cref{def:Qintersect}. In the next definition, we introduce the function $\rank : \interprog \to \Z^{\geq 0} \cup \{+\infty\}$ which is used as a cut-off for the combinatorial sieve (or rather, a family of such functions, one for every $\d$).

\begin{definition}
    \label{def:rank}
    Let $\s$, $\lit$, $\unlit$ be sets such that $\intK\times\intJ = \s\sqcup\lit\sqcup\unlit$ and let $1\leq r\leq \K$.

    Let $\d\in \DD_{\s,\lit,r}$ and let $\progression \subset \Z$ be the arithmetic progression
    \begin{equation*}
        \progression := \{n\in \Z \,:\, \forall (i,j)\in \lit, \, \dsub{i}{j} \mid n+\b_i\}.
    \end{equation*}
    Let $R \in \interprog$. If $R\cap \progression = \emptyset$, we set $\rank(R) := +\infty$. Otherwise, we define $\rank(R)$ to be the largest integer $T \geq 0$ for which there exist progressions $Q_1, \ldots, Q_{T}\in \shiftedprog$ containing $R$ such that, for each $t\in \sinterval{T}$, the modulus $\modulus_{Q_t}$ does not divide $\modulus_{\progression}\prod_{s\in \sinterval{T}\setminus \{t\}}\modulus_{Q_s}$.
\end{definition}

We need to show that these rank functions satisfy the four properties of \cref{lem:existencerank}. We will be able to quickly derive the first three properties from the following simple fact.

\begin{lemma}
    \label{lem:rankintersection}
    Let $R \in \interprog$ be such that $\rank(R) = T < +\infty$. Then, there are progressions ${Q_1, \ldots, Q_{T}\in \shiftedprog}$ such that
    \begin{equation*}
        \emptyset \neq R \cap \progression = \bigcap_{t\in\sinterval{T}} Q_t \cap \progression.
    \end{equation*}
\end{lemma}
\begin{proof}
    By definition of $\interprog$, we may write $R = \bigcap_{i\in I} Q_i$ for some finite set $I$ and some $Q_i\in \shiftedprog$. Let~$I_0$ be a minimal subset of $I$ such that
    \begin{equation}
        \label{eq:writeintersection}
        R \cap \progression = \bigcap_{i\in I_0} Q_i \cap \progression.
    \end{equation}
    Note that the modulus of a non-empty intersection of arithmetic progressions is the least common multiple of the moduli of these progressions.
    There is no $i_0\in I_0$ such that $\modulus_{Q_{i_0}}$ divides $\modulus_{\progression}\prod_{i\in I_0\setminus \{i_0\}}\modulus_{Q_i}$, for otherwise $\bigcap_{i\in I_0\setminus \{i_0\}} Q_i \cap \progression$ and $\bigcap_{i\in I_0} Q_i \cap \progression$ would have the same modulus, so these progressions would be equal, contradicting the minimality of $I_0$. This shows that $\abs{I_0} \leq \rank(R)$. Thus, \cref{eq:writeintersection} means that we have been able to write $R \cap \progression$ as an intersection of at most $\rank(R)$ progressions $Q_i \cap \progression$. Repeating some $Q_i$ if necessary, we can make it an intersection of exactly $\rank(R)$ sets.
\end{proof}

We restate \cref{lem:existencerank} here for convenience.

\existencerank*

\begin{proof}[Proof of parts \cref{item:rank1} and \cref{item:rank2} of \cref{lem:existencerank}]
    Let $\d\in \DD_{\s,\lit,r}$ and let $R\in \interprog$ be a progression with ${\rank(R) = T < +\infty}$. By \cref{lem:rankintersection}, there are ${Q_1, \ldots, Q_{T}\in \shiftedprog}$ such that
    \begin{equation*}
        \emptyset \neq R \cap \progression = \bigcap_{t\in\sinterval{T}} Q_t \cap \progression.
    \end{equation*}
    Property \cref{item:rank2} follows, since
    \begin{equation*}
        \omega(\modulus_{R}) \leq \omega(\modulus_{R\cap \progression}) \leq \omega(\modulus_{\progression})+\sum_{t\in\sinterval{T}} \omega(\modulus_{Q_t}) \leq 2\J\K + \J\L T.
    \end{equation*}
    For property \cref{item:rank1}, write $\omega_{\s}(n) := \sum_{p\mid \d_{\s}} \ind{p\mid n}$. We similarly obtain
    \begin{equation*}
        \omega_{\s}(\modulus_{R}) \leq \omega_{\s}(\modulus_{R\cap \progression}) \leq \omega_{\s}(\modulus_{\progression})+ \sum_{t\in\sinterval{T}} \omega_{\s}(\modulus_{Q_t})  \leq 0 + \J\L T
    \end{equation*}
    as $\modulus_{\progression}$ is only divisible by the primes $p\mid \d_{\lit}$.
\end{proof}

For part \cref{item:ranksieve} of \cref{lem:existencerank}, namely the combinatorial sieve, we just need to use \cref{prop:sievesquare-free}, checking that the hypotheses are satisfied.
\begin{proof}[Proof of part \cref{item:ranksieve} of \cref{lem:existencerank}]
    We use \cref{prop:sievesquare-free} with the initial set of arithmetic progressions being $\shiftedprog$, and with $X = \Xd$ being the set of all $R\in \interprog$ such that $\rank(R) <  \K^{5\epsone}$. Note that $\Xd\neq \emptyset$ as $\Z\in \Xd$.

    For any $R,R'\in \interprog$ with $R\subset R'$, it is clear from \cref{def:rank} that $\rank(R') \leq \rank(R)$. Therefore $\Xd$ is closed under containment. Furthermore, $\omega(\modulus_{R}) \ll \J \K$ for all $R\in \Xd$, by property~\cref{item:rank2} of \cref{lem:existencerank}, as $\L = \K^{1-10\epsone}$. For elements $R\in \partial \Xd\setminus \{\emptyset\}$, we also have $\omega(\modulus_{R}) \ll \J \K$, as any $P\in \shiftedprog$ has ${\omega(\modulus_{P}) \leq \J \L \leq \J\K}$ by definition of a prohibited progression. The conclusion follows from \cref{prop:sievesquare-free}, observing that `$n\not \in P$ for all $P\in \shiftedprog$' is equivalent to `$n+\b_i \in \Y$ for all $i\in \intK$'.
\end{proof}

To prove part \cref{item:rank3} of \cref{lem:existencerank}, we will need two technical lemmas: \cref{lem:excludedprimes} and \cref{lem:remainderterm}.

\begin{lemma}
    \label{lem:excludedprimes}
    Let $\s$, $\lit$, $\unlit$ be sets such that $\intK\times\intJ = \s\sqcup\lit\sqcup\unlit$ and let $1\leq r\leq \K$.

    Let $\d \in \DD_{\s,\lit,r}$.
    Let $\primes' \subset \primes$ be a set of size $\leq 3\J\K$ containing the prime divisors of $\modulus_{\progression}$. Then
    \begin{equation*}
        \sum_{\substack{R\in \partial \Xd \\ R\cap \progression \neq \emptyset}} \,\prod_{\substack{p\mid \modulus_{R}\\ p\not\in \primes'}} \frac{1}{p} \ll e^{O(\J\K)}.
    \end{equation*}
\end{lemma}

\begin{proof}
    Let $\mathcal{X}$ be the set of all $R_1\in \interprog \setminus \{\emptyset\}$ of the form $R_1 = R\cap \progression$ for some $R\in \partial \Xd$. Since the prime factors of $\modulus_{\progression}$ are in $\primes'$, we may rewrite the expression in the statement as
    \begin{equation*}
        \sum_{\substack{R\in \partial \Xd \\ R\cap \progression \neq \emptyset}}\, \prod_{\substack{p\mid \modulus_{R}\\ p\not\in \primes' }}\frac{1}{p} = \sum_{R_1\in \mathcal{X}}  \prod_{\substack{p\mid \modulus_{R_1}\\ p\not\in \primes'}}\frac{1}{p} \sum_{\substack{R\in \partial \Xd \\ R\cap \progression = R_1}} 1.
    \end{equation*}
    To bound the inner sum, we use the following fact: for any arithmetic progression $R_1$ and any $n\mid \modulus_{R_1}$, there is a unique arithmetic progression $R\supset R_1$ with $\modulus_{R_1}/\modulus_{R} = n$; moreover, all progressions $R\supset R_1$ are obtained in this way. Therefore, the inner sum is bounded by the number of divisors of $\modulus_{R_1}$. For every $R_1\in \mathcal{X}$, we have ${\omega(\modulus_{R_1}) \ll \J\K}$. This follows from part \cref{item:rank2} of \cref{lem:existencerank}, using that $R_1 = R_0\cap P \cap \progression$ for some $R_0\in \Xd$ and $P\in \shiftedprog$. Therefore, $\modulus_{R_1}$ has $e^{O(\J\K)}$ divisors, and hence the inner sum is $e^{O(\J\K)}$.

    It remains to show that
    \begin{equation}
        \label{eq:partrank4}
        \sum_{R_1\in \mathcal{X}}  \prod_{\substack{p\mid \modulus_{R_1}\\ p\not\in \primes'}}\frac{1}{p} \ll e^{O(\J\K)}.
    \end{equation}
    This is a simple counting problem, similar to \cref{lem:trivialbound} or \cref{prop:predictablevector}, but the notation is much heavier in this setting.

    Let $T := \lfloor \K^{5\epsone} \rfloor+1$. By definition of $\Xd$ and \cref{lem:rankintersection}, for any $R_0\in \Xd$, the intersection $R_0 \cap \progression$ can be written as the intersection of $T-1$ progressions of the form $Q\cap \progression$ where $Q\in \shiftedprog$. Hence, the same is true for any $R_0\in \partial \Xd$ with $T$ instead of $T-1$, by definition of $\partial \Xd$.

    Let $R_1\in \mathcal{X}$. By definition of $\mathcal{X}$ and the previous paragraph, we can write
    \begin{equation}
        \label{eq:reprR1}
        R_1 = \bigcap_{t\in \sinterval{T}} \big(Q_t - \b_{k_t}\big) \cap \progression
    \end{equation}
    for some $Q_t\in \prohibprog$ and $k_t\in \interval{0}{2\K}$. For $t\in \sinterval{T}$, let $\d^{[t]}$ be a primitive prohibited sequence having $Q_t$ as its associated prohibited progression. Let $\ell_t$ be the length of $\d^{[t]}$ and let $\sigma_t\in \{\pm1\}^{\ell_t}$ be the sequence of signs of the coordinates of $\d^{[t]}$. As usual, for $(i,j)\in \sinterval{\ell_t} \!\times \!\intJ$ we write $d^{[t]}_{ij}$ for the unique prime in $\primes_j$ dividing $d^{[t]}_i$. Let $\sim$ be the equivalence relation on $\bigsqcup_{t\in \sinterval{T}} \big(\{t\}\!\times \!\sinterval{\ell_t} \!\times \!\intJ\big)$ defined by
    \begin{equation*}
        (t_1,i_1,j_1) \sim (t_2,i_2,j_2) \quad \iff \quad d^{[t_1]}_{i_1j_1} = d^{[t_2]}_{i_2j_2}.
    \end{equation*}
    For any equivalence class $\alpha$ of $\sim$, we write $p_{\alpha}$ for the prime $d^{[t]}_{ij}$, where $(t,i,j)$ is any element of~$\alpha$. This definition does not depend on the choice of representative, by definition of $\sim$. Let $E$ be the set of all equivalence classes $\alpha$ for $\sim$ such that $p_{\alpha} \in \primes'$. Let $\phi : E \to\primes'$ be the map defined by $\phi(\alpha) = p_{\alpha}$. We call the tuple $\theta := ((k_t)_{t\in \sinterval{T}}, (\ell_t)_{t\in \sinterval{T}}, (\sigma_t)_{t\in \sinterval{T}}, \sim, E, \phi)$ a \emph{template} for $R_1$. Thus, to every progression $R_1\in \mathcal{X}$ we may associate a template $\theta$ (there may not be a canonical choice for the template associated to $R_1$, as it depends on the choice of a representation of $R_1$ as in~\cref{eq:reprR1}).

    Let $\Theta$ denote the set of all tuples $\theta$ which are the template of some progression $R_1 \in \mathcal{X}$.

    Fix some $\template = ((k_t)_{t\in \sinterval{T}}, (\ell_t)_{t\in \sinterval{T}}, (\sigma_t)_{t\in \sinterval{T}}, \sim, E, \phi) \in \Theta$ and let $\mathcal{X}_\template$ be the set of all $R_1\in \mathcal{X}$ for which $\template$ is a template. Observe that any $R_1\in \mathcal{X}_\template$ is uniquely determined by the sequence of primes $(p_{\alpha})_{\alpha\in E'}$, where $E'$ is the set of all equivalence classes $\alpha$ of $\sim$ such that $\alpha \notin E$. Thus
    \begin{equation}
        \label{eq:partrank4bis}
        \sum_{R_1\in \mathcal{X}_\template} \, \prod_{\substack{p\mid \modulus_{R_1}\\ p\not\in \primes'}}\frac{1}{p} \leq \sum_{(p_{\alpha})_{\alpha\in E'}} \, \prod_{\alpha\in E'} \frac{1}{p_{\alpha}} \leq \V^{\abs{E'}} \leq \V^{T \L \J} \leq e^{O(\J\K)},
    \end{equation}
    where we used that $T \ll \K^{5\epsone}$ and $\L<\K^{1-10\epsone}$ for the last inequality.

    We proceed to sum \cref{eq:partrank4bis} over all possible choices of $\template\in \Theta$. We will be done provided that the number of such templates is $e^{O(\J\K)}$. The number of choices for $(k_t)_{t\in \sinterval{T}}$, $(\ell_t)_{t\in \sinterval{T}}$ and $(\sigma_t)_{t\in \sinterval{T}}$ is at most $(2\K+1)^{T}$, $\L^{T}$ and $(2^{\L})^{T}$ respectively. Since $\sim$ is an equivalence relation on a set of size $\leq T \L \J$, there are $\leq (T \L \J)^{T \L \J}$ choices for $\sim$. There are $\leq 2^{T \L \J}$ choices for $E$. Finally, $\phi$ is a map from a set of size $\leq T \L \J$ to a set of size $\leq 3\J\K$ (using the assumption on $\primes'$ in the statement), so there are $\leq (3\J\K)^{T \L \J}$ possibilities for $\phi$. In summary, the number of templates is
    \begin{equation*}
        \leq (2\K+1)^{T} \cdot \L^{T} \cdot (2^{\L})^{T} \cdot (T \L \J)^{T \L \J} \cdot 2^{T \L \J} \cdot (3\J\K)^{T \L \J} = e^{O(\J\K)}.
    \end{equation*}
    This proves \cref{eq:partrank4} and concludes the proof of \cref{lem:excludedprimes}.
\end{proof}

The next technical lemma relates the remainder term of the combinatorial sieve to triangular systems of constraints. It relies on \cref{lem:prohibitedanalysis} and is the key ingredient in the proof of part \cref{item:rank3} of \cref{lem:existencerank}.

\begin{lemma}
    \label{lem:remainderterm}
    Let $\s$, $\lit$, $\unlit$ be sets such that $\intK\times\intJ = \s\sqcup\lit\sqcup\unlit$ and $|\unlit| \leq\K^{2\epsone}$. Let $1\leq r\leq \K$.

    Let $\d \in \DD_{\s,\lit,r}$. Let $T \geq 10 \K^{2\epsone}$ and let $k_1\leq k_2\leq \cdots \leq k_{T}$ be elements of $\interval{0}{2\K}$. Let $\d^{[1]}, \ldots, \d^{[T]}$ be primitive prohibited sequences, and let $Q_1, \ldots, Q_{T}\in \prohibprog$ be the associated prohibited progressions. Suppose that, for each $t\in \sinterval{T}$, the modulus $\modulus_{Q_t}$ does not divide $\modulus_{\progression}\prod_{s\in \sinterval{T}\setminus \{t\}}\modulus_{Q_s}$, and that
    \begin{equation*}
        \bigcap_{t\in \sinterval{T}} \big(Q_t - \b_{k_t}\big) \cap \progression \neq \emptyset.
    \end{equation*}
    Let $\bm{v}$ be the sequence obtained by concatenation of $\d, \d^{[1]}, \ldots, \d^{[T]}, -\d^{[1]}, \ldots, -\d^{[T]}$. Then, $\bm{v}$ satisfies a triangular system of $\gg T$ constraints with shift $\kappa = 0$.
\end{lemma}

\begin{proof}
    Fix some $n\in \bigcap_{t=1}^{T} \big(Q_t - \b_{k_t}\big) \cap \progression$.

    For every $t\in \sinterval{T}$, let $\ell_t$ be the length of $\d^{[t]}$, and let $\Gamma_t$ be the set of all constraints with shift $\kappa = 0$ satisfied by $\d^{[t]}$.

    Suppose first that there is a set $I\subset \sinterval{T}$ of size $\geq \tfrac{1}{10} T$ such that, for every $t\in I$, there is a constraint $C_t\in \Gamma_t$ and a prime $p_t$ which is involved in $C_t$ and does not divide $\prod_{s<t}\modulus_{Q_s}$. Thus, $p_t$ is absent from $C_s$, for every $s\in I$ with $s<t$. This means that the constraints $(C_t)_{t\in I}$ form a triangular system, so we are done. The same conclusion holds if there is a set $I\subset \sinterval{T}$ of size $\geq \tfrac{1}{10} T$ such that, for every $t\in I$, there is a constraint $C_t\in \Gamma_t$ and a prime $p_t$ which is involved in $C_t$ and does not divide $\prod_{s>t}\modulus_{Q_s}$. We may thus assume that, for $\geq \tfrac{8}{10} T$ values of $t\in \sinterval{T}$, every prime involved in some constraint of $\Gamma_t$ divides both $\prod_{s<t}\modulus_{Q_s}$ and $\prod_{s>t}\modulus_{Q_s}$.

    For every $t\in \sinterval{T}$, fix a prime $p_t$ dividing $\modulus_{Q_t}$ but not dividing $\modulus_{\progression}\prod_{s\in \sinterval{T}\setminus \{t\}}\modulus_{Q_s}$. This is possible by the assumption in the statement. We apply \cref{lem:prohibitedanalysis} with this prime $p_t$. Note that the first case of \cref{lem:prohibitedanalysis} (i.e.~$p_t$ being involved in some constraint of $\Gamma_t$) can only occur for $<\tfrac{2}{10}T$ values of $t\in \sinterval{T}$ by definition of $p_t$ and the previous paragraph. Let $I_1$ be the set of $t\in \sinterval{T}$ such that the second case holds, i.e.~for $t\in I_1$ there is a prime $q_t$ involved in a constraint $C_t\in \Gamma_t$ such that
    \begin{equation}
        \label{eq:propertyofpt}
        \sum_{\substack{1\leq i<i_t\\ p_t\mid d^{[t]}_i}} d^{[t]}_i \not\equiv 0 \pmod{q_t},
    \end{equation}
    where $i_t\in \sinterval{\ell_t}$ is such that $q_t\mid d^{[t]}_{i_t}$. Thus $\abs{I_1}\geq \tfrac{8}{10} T$.

    By our earlier arguments, there is a subset $I_2\subset I_1$ of size $\abs{I_2} \geq \tfrac{6}{10} T$ such that, for all $t\in I_2$, there are $1\leq s_1(t)<t<s_2(t)\leq T$ with $q_t\mid \modulus_{Q_{s_1(t)}}$ and $q_t\mid \modulus_{Q_{s_2(t)}}$.

    By definition of $p_t$, we know that $p_t \nmid \modulus_{\progression}$, i.e.~$p_t\nmid \d_{\lit}$. Moreover, there are at most $\tfrac{1}{10} T$ values of $t\in I_2$ such that $p_t\mid \d_{\unlit}$, since $\abs{\unlit}\leq \K^{2\epsone}\leq \tfrac{1}{10} T$ and all $p_t$ are distinct. We may thus find a subset $I_3\subset I_2$ of size $\abs{I_3} \geq\tfrac12 T$ such that $p_t\nmid\d_{\lit \sqcup \unlit}$ for every $t\in I_3$.

    Let $I_4$ be the set of all $t\in I_3$ such that $p_t \nmid \d_{\interval{1}{k_t}\times \intJ}$. Let $I_5$ be the set of all $t\in I_3$ such that $p_t \nmid \d_{\interval{k_t+1}{2\K}\times \intJ}$. By definition of $I_3$, we know that for every $t\in I_3$ there is at most one pair $(i,j)\in \intK\times\intJ$ (which lies in $\s$) such that $p_t = \dsub{i}{j}$. In particular, $I_4 \cup I_5 = I_3$, so one of $I_4$ and $I_5$ has size $\geq \tfrac{1}{4} T$. We will only treat the case where $\abs{I_4} \geq  \tfrac{1}{4} T$; the proof for the case $\abs{I_5} \geq  \tfrac{1}{4} T$ is the same up to symmetry.

    Let $t\in I_4$. By definition of $n$, we have $n+\b_{k_t} \in Q_t$ and thus, by definition of $Q_t$ being the prohibited progression associated to $\d^{[t]}$,
    \begin{equation*}
        q_t \bigmid\  n + \b_{k_t} + \sum_{\substack{1\leq i<i_t}} d_i^{[t]},
    \end{equation*}
    with $i_t$ as defined earlier. Since $q_t\mid \modulus_{Q_{s_1(t)}}$, the same reasoning shows that
    \begin{equation*}
        q_t \bigmid\  n + \b_{k_{s_1(t)}} + \sum_{\substack{1\leq i<i_t'}} d_i^{[s_1(t)]},
    \end{equation*}
    where $i_t'\in \sinterval{\ell_{s_1(t)}}$ is any index such that $q_t\mid d^{[s_1(t)]}_{i_t'}$. Subtracting the two divisibility relations, we obtain
    \begin{equation*}
        q_t \bigmid\ \sum_{k_{s_1(t)} < i \leq k_t} d_i + \sum_{\substack{1\leq i<i_t}} d_i^{[t]} - \sum_{\substack{1\leq i<i_t'}} d_i^{[s_1(t)]}.
    \end{equation*}
    This is now a genuine constraint on $\bm{v}$ (with shift $\kappa = 0$), which we call $C_t'$. By \cref{eq:propertyofpt}, and since $p_t\nmid \modulus_{Q_{s_1(t)}}$ (by definition of $p_t$) and $p_t \nmid \d_{\interval{1}{k_t}\times \intJ}$ (by definition of $I_4$), we see that $p_t$ is involved in this constraint $C_t'$. In addition, for $t_1,t_2\in I_4$ with $t_1<t_2$, the prime $p_{t_2}$ is absent from $C_{t_1}'$ since none of the integers $\d_{\interval{1}{k_{t_1}}\times \intJ}$, $ \modulus_{Q_{t_1}}$ and $\modulus_{Q_{s_1(t_1)}}$ are divisible by $p_{t_2}$. Therefore, the family $(C_t')_{t\in I_4}$ forms a triangular system of $\gg T$ constraints, and we are done.

    The case $\abs{I_5} \geq  \tfrac{1}{4} T$ is analogous, where this time $s_2(t)$ takes the role of $s_1(t)$.
\end{proof}

We can finally prove part \cref{item:rank3} of \cref{lem:existencerank}.

\begin{proof}[Proof of part \cref{item:rank3} of \cref{lem:existencerank}]
    Let $T := \lceil \K^{5\epsone} \rceil$. Let $\d \in \DD_{\s,\lit,r}$ and let $R\in \partial \Xd$ be such that ${R\cap \progression \neq \emptyset}$. Observe that $R$ satisfies $T \leq \rank(R) <+\infty$ by definition of $\partial \Xd$. Thus, by \cref{def:rank}, we can find progressions $Q_1, \ldots, Q_{T}\in \shiftedprog$ containing $R$ such that, for each $t\in \sinterval{T}$, the modulus $\modulus_{Q_t}$ does not divide $\modulus_{\progression}\prod_{s\in \sinterval{T}\setminus \{t\}}\modulus_{Q_s}$. We will first sum over all possibilities for ${R_1 := \bigcap_{t\in \sinterval{T}} Q_t \cap \progression}$.

    Let $\mathcal{X}$ be the set of all $R_1\in \interprog$ which are of the form
    \begin{equation}
        \label{eq:reprR1bis}
        R_1 = \bigcap_{t\in \sinterval{T}} Q_t \cap \progression \neq \emptyset
    \end{equation}
    for some ${Q_1, \ldots, Q_{T}\in \shiftedprog}$ such that $\modulus_{Q_t}$ does not divide $\modulus_{\progression}\prod_{s\in \sinterval{T}\setminus \{t\}}\modulus_{Q_s}$ for all $t\in \sinterval{T}$. Summing over $\mathcal{X}$ first, we have
    \begin{equation}
        \label{eq:forrank3eq1}
        \sum_{\substack{\d\in \DD_{\s,\lit,r}}} \sum_{\substack{R\in \partial\Xd \\ R\cap \progression \neq \emptyset}}\, \prod_{\substack{p\mid \modulus_{R} \pall}}\frac{1}{p} \leq  \sum_{\substack{\d\in \DD_{\s,\lit,r}}} \sum_{R_1\in \mathcal{X}}  \prod_{\substack{p\mid \modulus_{R_1}\! \pall}}\frac{1}{p} \sum_{\substack{R\in \partial\Xd \\ R\cap \progression \neq \emptyset}}\, \prod_{\substack{p\mid \modulus_{R}\\ p\nmid \modulus_{R_1}\! \pall}}\frac{1}{p}.
    \end{equation}

    For fixed $R_1$, we bound the innermost sum using \cref{lem:excludedprimes} with ${\primes' := \{p:p\mid \modulus_{R_1} \pall\}}$. The assumptions on $\primes'$ are satisfied since $\primes'$ contains the prime divisors of $\modulus_{\progression}$ and
    \begin{equation*}
        \abs{\primes'} \leq \omega(\pall) + \abs{\{p : p\mid \modulus_{R_1}, \, p\nmid \modulus_{\progression}\}} \leq 2\J\K +  T \J\L \leq 3\J\K
    \end{equation*}
    in view of the representation \cref{eq:reprR1bis} of elements of $\mathcal{X}$. Thus, applying \cref{lem:excludedprimes}, we obtain
    \begin{equation}
        \label{eq:forrank3eq2}
        \sum_{\substack{R\in \partial\Xd \\ R\cap \progression \neq \emptyset}}\, \prod_{\substack{p\mid \modulus_{R}\\ p\nmid \modulus_{R_1}\! \pall}}\frac{1}{p} = e^{O(\J\K)}.
    \end{equation}

    It remains to bound the expression
    \begin{equation}
        \label{eq:forrank3eq3}
        \sum_{\substack{\d\in \DD_{\s,\lit,r}}} \sum_{R_1\in \mathcal{X}}  \prod_{\substack{p\mid \modulus_{R_1}\! \pall}}\frac{1}{p}.
    \end{equation}
    For every non-decreasing sequence $\bm{k} = (k_t)_{t\in \sinterval{T}}$ of elements of $\interval{0}{2\K}$, let $\XX(\bm{k})$ be the set of all pairs $(\d, R_1)\in \DD_{\s,\lit,r}\times \mathcal{X}$ such that
    \begin{equation*}
        \emptyset \neq R_1 =  \bigcap_{t\in \sinterval{T}} \big(Q_t - \b_{k_t}\big) \cap \progression
    \end{equation*}
    for some prohibited progressions $Q_t\in \prohibprog$ with $\modulus_{Q_t} \nmid \modulus_{\progression}\prod_{s\in \sinterval{T}\setminus \{t\}}\modulus_{Q_s}$ for all $t$.

    By \cref{lem:remainderterm}, for any pair $(\d, R_1)\in \XX(\bm{k})$ and any choice $\d^{[1]}, \ldots, \d^{[T]}$ of prohibited sequences used in the definition of $R_1$, the concatenation of $\d, \d^{[1]}, \ldots, \d^{[T]}, -\d^{[1]}, \ldots, -\d^{[T]}$ satisfies a triangular system of $\gg T$ constraints with shift $\kappa = 0$. This concatenation has length $\leq 2\K + 2T \L \leq 3\K$. Applying \cref{lem:constraints}, we get
    \begin{equation*}
        \sum_{(\d, R_1)\in \XX(\bm{k})} \, \prod_{\substack{p\mid \modulus_{R_1}\! \pall}}\frac{1}{p} \ll \K^{O(\J\K)} H_0^{-cT}
    \end{equation*}
    for some absolute constant $c>0$. Summing over all $\K^{O(T)}$ choices for $\bm{k} = (k_t)_{t\in \sinterval{T}}$, we obtain that the expression \cref{eq:forrank3eq3} is bounded by $\K^{O(\J\K+T)} H_0^{-cT}$.

    By \cref{eq:forrank3eq1}, and recalling our bound \cref{eq:forrank3eq2} for the inner sum, we conclude that
    \begin{equation*}
        \sum_{\substack{\d\in \DD_{\s,\lit,r}}} \sum_{\substack{R\in \partial\Xd \\ R\cap \progression \neq \emptyset}} \prod_{\substack{p\mid \modulus_{R} \pall}}\frac{1}{p}  \ll e^{O(\J\K)} \K^{O(\J\K+T)} H_0^{-cT},
    \end{equation*}
    which is $\ll 1$ since $\J \ll \log \K$, ${\log H_0 \gg\K^{1-\epsone}}$ and $T \asymp \K^{5\epsone}$.
\end{proof}

\subsection{Proof of the high trace bound}
\label{sec:proofofhightrace}
Having assembled all the necessary ingredients over the course of the preceding sections, we now reach the final step: proving the high-trace estimate for the non-backtracking operator $M_{\Y}$.

We restate \cref{prop:hightracebound} here for convenience.

\hightracebound*

\begin{proof}[Proof of \cref{prop:hightracebound}]
    We prove \cref{prop:hightracebound} with $Y = \Y$, the set defined in \cref{def:prohibitedprogression}. This choice is acceptable as $\abs{\N\setminus \Y} \ll H_0^{-1/3}N$ by \cref{lem:Ysmall}.

    We turn to the high trace bound. By \cref{lem:threeindicestypes}, we have
    \begin{equation*}
        \Tr\big[\big((M_{Y})^{\K}\big)^*(M_{Y})^{\K}\big] \ll e^{O(\J\K)}\max_{\substack{\s \sqcup \lit \sqcup \unlit = \intK\times \intJ\\ 1\leq r\leq \K}} E_{\s,\lit,\unlit, r} + N
    \end{equation*}
    with $E_{\s,\lit,\unlit, r}$ as in \cref{def:Esum}. Hence, it suffices to prove that
    \begin{equation}
        \label{eq:boundEtoprove}
        E_{\s,\lit,\unlit, r} \ll e^{O(\J\K)} \V^{\J\K} N
    \end{equation}
    for any choice of $\s$, $\lit$, $\unlit$ and $r$.

    By \cref{lem:unlit}, we have $E_{\s,\lit,\unlit, r} \ll N$ when $|\unlit| \geq \K^{2\epsone}$, which is stronger than \cref{eq:boundEtoprove}. We now assume that $|\unlit| < \K^{2\epsone}$.

    By \cref{prop:cancellationoverY}, we have $E_{\s,\lit,\unlit, r} \ll e^{O(\J\K)}N$ when $\abs{\s}\geq \K^{1-\epsone}$, which is also stronger than~\cref{eq:boundEtoprove}. We may thus assume that $\abs{\s} < \K^{1-\epsone}$.

    In the remaining case, we apply first the trivial bound given by \cref{lem:trivboundE} to obtain
    \begin{equation*}
        E_{\s,\lit,\unlit, r} \ll e^{O(\J\K)} N \sum_{\d \in \DD_{\s,\lit,r}}\, \prod_{p\mid \pall} \frac{1}{p}.
    \end{equation*}
    We then split the sum over $\d$ into two parts: the predictable part $\pred$ and the unpredictable part $\unpred$ (see \cref{def:predunpred}).

    By \cref{prop:predictablevector}, the predictable contribution is
    \begin{equation*}
        \sum_{\substack{\d \in \pred}} \, \prod_{p\mid \pall} \frac{1}{p} \ll e^{O(\J\K)}\V^{|\s|+(\abs{\lit}+\abs{\unlit})/2}.
    \end{equation*}
    Since $\abs{\s} < \K^{1-\epsone}$ and $\V \ll \K$, we have $\V^{|\s|} \ll e^{O(\K)}$, which can be absorbed in the $e^{O(\J\K)}$ factor. Therefore, since $\abs{\lit} + \abs{\unlit} \leq 2\J\K$, we obtain
    \begin{equation}
        \label{eq:finalpredcontrib}
        \sum_{\substack{\d \in \pred}} \, \prod_{p\mid \pall} \frac{1}{p} \ll e^{O(\J\K)} \V^{\J\K}.
    \end{equation}

    Finally, by \cref{prop:unpredictable}, the unpredictable contribution satisfies
    \begin{equation}
        \label{eq:finalunpredcontrib}
        \sum_{\d\in \unpred}\,\prod_{p\mid \pall}\, \frac{1}{p} \ll 1.
    \end{equation}

    Combining \cref{eq:finalpredcontrib} and \cref{eq:finalunpredcontrib}, we obtain the bound \cref{eq:boundEtoprove} in the remaining case where $|\s| < \K^{1-\epsone}$ and $|\unlit| < \K^{2\epsone}$. This finishes the proof of \cref{prop:hightracebound}.
\end{proof}

\appendix
\begin{appendices}
    \section{The non-backtracking operator}
    \label{appendix:nonbacktracking}

    We begin by showing how the results in Deitmar~\cite{deitmar} combine to yield the weighted Ihara-Bass formula stated in \cref{prop:iharabass}. Note that Deitmar's setting is slightly different, as it encompasses directed weighted graphs that may be infinite but have summable weights. Alternatively, one could derive the same identity from \cite[Lemma~5.2]{fan} or \cite[Theorem~2]{watanabe}.

    \begin{proof}[Proof of \cref{prop:iharabass}]
        Let $u\in \C$ be sufficiently small. The weighted \emph{Ihara zeta function} $Z(u)$ of $G$ is defined, in analogy to the Riemann zeta function, as an Euler product over regular prime cycles in $G$ (see \cite[Definition~1.4]{deitmar}). As stated in \cite[Theorem~1.6]{deitmar}, it is related to the non-backtracking operator~$M$ by the identity
        \begin{equation}
            \label{eq:Zeta1}
            \det(I-uM) = Z(u)^{-1}.
        \end{equation}

        Furthermore, \cite[Theorem~2.10]{deitmar} provides an alternative expression for $Z(u)^{-1}$:
        \begin{equation}
            \label{eq:Zeta2}
            Z(u)^{-1} = \det(B(-u))  \prod_{\{x,y\}\in E}\big(1-u^2w(x,y)^2\big),
        \end{equation}
        where $B(-u)$ is an operator on $L(V)$ defined in \cite[Lemma~2.2]{deitmar}. Expanding this definition, we get
        \begin{equation}
            \label{eq:expansionBu}
            B(-u) = I - u B_1 + \sum_{n\geq 1} u^{2n}B_{2n} - \sum_{n\geq 1} u^{2n+1} B_{2n+1}
        \end{equation}
        where the linear maps $B_j : L(V) \to L(V)$ are defined on the canonical basis $(\texttt{x})_{x\in V}$ by
        \begin{equation*}
            B_{2n-1}(\texttt{x}) = \sum_{y\in V} w(x,y)^{2n-1} \texttt{y} \quad\text{and}\quad B_{2n}(\texttt{x}) = \sum_{y\in V} w(x,y)^{2n} \texttt{x}
        \end{equation*}
        for $n\geq 1$. In particular, $B_1$ is the adjacency operator $A$, and the two sums in \cref{eq:expansionBu} evaluate to $u^2D_u$ and $u^2O_u$ respectively, by the geometric series formula.

        Comparing \cref{eq:Zeta1,eq:Zeta2} yields the Ihara-Bass formula \cref{eq:IharaBass} for small $u$. A straightforward analytic continuation argument extends this equality for all $u\in \C$ for which $D_u$ and $O_u$ are defined, and also shows that any positivity assumption on the weights in \cite{deitmar} can be disregarded for this result.
    \end{proof}

    \begin{definition}
        \label{def:nbwalks}
        Let $G = (V, E, w)$ be a weighted graph. A \emph{walk} in $G$ is a finite sequence $(v_1, v_2, \ldots, v_k)$ of vertices of $G$ such that $\{v_i, v_{i+1}\} \in E$ for all $1\leq i < k$. A walk $(v_1, v_2, \ldots, v_k)$ is \emph{non-backtracking} if $v_{i-1} \neq v_{i+1}$ for all $1< i < k$. We write $\NB{G}$ for the set of non-backtracking walks in $G$ (of any length).
    \end{definition}

    \begin{lemma}
        \label{lem:traceexpansion}
        Let $G = (V, E, w)$ be a weighted graph with non-backtracking operator $M$. Then, for any $k\geq 1$,
        \begin{equation*}
            \Tr\big[(M^k)^*M^k\big] = \sum_{\substack{x_{-1}, x_0, x_1, \ldots, x_{2k}\in V\\ (x_{-1},x_0, \ldots, x_k) \in \NB{G}\\ (x_{k}, x_{k+1}, \ldots, x_{2k}, x_{-1})\in \NB{G}\\ (x_{0}, x_{k-1})=(x_{2k}, x_{k+1})}} \prod_{i=1}^{2k} w(x_{i-1}, x_i).
        \end{equation*}
    \end{lemma}

    \begin{proof}
        By definition of the non-backtracking operator $M$, for any basis vector $\overrightarrow{\texttt{x}_{-1}\texttt{x}_0}$ of $L(\oE)$,
        \begin{equation*}
            M^k (\overrightarrow{\texttt{x}_{-1}\texttt{x}_0}) = \sum_{\substack{x_1, \ldots, x_k\in V \\ (x_{-1},x_0, \ldots, x_k)\in \NB{G} }} \prod_{i=1}^k w(x_{i-1}, x_i) \ \overrightarrow{\texttt{x}_{k-1}\texttt{x}_k}.
        \end{equation*}
        Similarly, for any basis vector $\overrightarrow{\texttt{y}_{-1}\texttt{y}_{0}}$ of $L(\oE)$,
        \begin{equation*}
            (M^k)^* (\overrightarrow{\texttt{y}_{-1}\texttt{y}_{0}}) = \sum_{\substack{y_{-2}, \ldots, y_{-k-1}\in V\\ (y_{0},y_{-1}, \ldots, y_{-k-1})\in \NB{G}}} \prod_{i=0}^{k-1} w(y_{-i}, y_{-i-1})\ \overrightarrow{\texttt{y}_{-k-1}\texttt{y}_{-k}}.
        \end{equation*}
        Therefore, the diagonal entry corresponding to $\overrightarrow{\texttt{x}_{-1}\texttt{x}_0}$ in the standard basis matrix of $(M^k)^*M^k$ is
        \begin{equation*}
            \big[(M^k)^*M^k\big]_{\overrightarrow{\texttt{x}_{-1}\texttt{x}_0},\overrightarrow{\texttt{x}_{-1}\texttt{x}_0}} = \sum_{\substack{x_1, \ldots, x_k\in V \\ (x_{-1}, x_0, \ldots, x_k)\in \NB{G}}} \sum_{\substack{y_{0}, y_{-1},\ldots, y_{-k-1}\in V \\ (y_{0},y_{-1}, \ldots, y_{-k-1})\in \NB{G}\\ (y_{-1}, y_{0}) = (x_{k-1}, x_k)\\ (y_{-k-1}, y_{-k})=(x_{-1},x_0)}}\prod_{i=1}^k w(x_{i-1}, x_i)  \prod_{i=0}^{k-1} w(y_{-i}, y_{-i-1}).
        \end{equation*}
        Performing the change of variables $x_{k+i} := y_{-i}$ for $1\leq i\leq k$, this can be rewritten as
        \begin{equation*}
            \big[(M^k)^*M^k\big]_{\overrightarrow{\texttt{x}_{-1}\texttt{x}_0},\overrightarrow{\texttt{x}_{-1}\texttt{x}_0}} = \sum_{\substack{x_1, \ldots, x_{2k}\in V\\ (x_{-1},x_0, \ldots, x_k) \in \NB{G}\\ (x_{k}, x_{k+1}, \ldots, x_{2k}, x_{-1})\in \NB{G}\\ (x_{0}, x_{k-1})=(x_{2k}, x_{k+1})}} \prod_{i=1}^{2k} w(x_{i-1}, x_i).
        \end{equation*}
        Summing over all diagonal entries gives the formula for the trace.
    \end{proof}

    \section{Combinatorial sieve for composite moduli}
    \label{appendix:sieve}

    Let $\prohibprog$ be a finite set of arithmetic progressions in $\Z$. By the inclusion-exclusion principle, we can write
    \begin{equation*}
        \ind{n\not\in P\  \forall P\in \prohibprog} = 1 - \sum_{P_1\in \prohibprog} \ind{n\in P_1} + \sum_{\substack{P_1, P_2\in \prohibprog\\ \text{distinct}}} \ind{n\in P_1 \cap P_2} - \sum_{\substack{P_1, P_2, P_3\in \prohibprog\\ \text{distinct}}} \ind{n\in P_1 \cap P_2\cap P_3} + \cdots = \sum_{S\subset \prohibprog} (-1)^{|S|} \ind{n \in \cap S}.
    \end{equation*}
    For $S = \emptyset$, we used the convention $\cap \emptyset := \Z$. The last sum contains $2^{|\prohibprog|}$ terms. We wish to replace this exact identity with an approximate version having far fewer terms. To do so, we truncate the above sum and restrict $S$ to a smaller collection $\calX$ of subsets of $\prohibprog$.

    \begin{lemma}
        \label{lem:sieveidentity}
        Let $\prohibprog$ be a finite set of arithmetic progressions in $\Z$. Let $\calX$ be a non-empty collection of subsets of $\prohibprog$ which is closed under containment, i.e.~if $S \in \calX$ and $S' \subset S$ then $S' \in \calX$.
        \begin{enumerate}
            \item If $n\not\in P$ for all $P\in \prohibprog$, then
                  \begin{equation*}
                      \ind{n\not\in P\  \forall P\in \prohibprog} = 1 = \sum_{S\in \calX} (-1)^{|S|} \ind{n \in \cap S}.
                  \end{equation*}
            \item If $n\in P_0$ for some progression $P_0\in \prohibprog$, then
                  \begin{equation}
                      \label{eq:sieveidentity}
                      \ind{n\not\in P\  \forall P\in \prohibprog}  = 0  = \sum_{S\in \calX} (-1)^{|S|} \ind{n \in \cap S}  + \sum_{\substack{P_0\in S \subset \prohibprog\\ S \not\in \calX,\, S\setminus \{P_0\} \in \calX}} (-1)^{|S|} \ind{n\in \cap S}.
                  \end{equation}
        \end{enumerate}
    \end{lemma}

    \begin{proof}
        \begin{enumerate}
            \item If $n$ does not belong to any $P\in \prohibprog$, all the terms in the sum are zero except for $S = \emptyset$.
            \item Suppose $n\in P_0 \in \prohibprog$. By inclusion-exclusion, we know that
                  \begin{equation*}
                      0 = \ind{n\not\in P\  \forall P\in \prohibprog}  = \sum_{S\subset \prohibprog} (-1)^{|S|} \ind{n \in \cap S} = \Bigg(\sum_{S\in \calX}+ \sum_{\substack{S\not\in \calX \\ P_0\not \in S}} + \sum_{\substack{S\not\in \calX \\ P_0\in S\\ S\setminus \{P_0\}\in \calX}}+ \sum_{\substack{S\not\in \calX \\ P_0\in S\\ S\setminus \{P_0\}\not\in \calX}}\Bigg) (-1)^{|S|} \ind{n \in \cap S}.
                  \end{equation*}
                  To obtain the conclusion, note that the second and fourth sums on the right-hand side cancel each other out, since
                  \begin{equation*}
                      \sum_{\substack{P_0\in S \subset \prohibprog\\ S\not\in \calX \\  S\setminus \{P_0\}\not\in \calX}} (-1)^{|S|} \ind{n \in \cap S} = \sum_{\substack{P_0\not\in \T \subset \prohibprog\\ \T\not\in \calX }} (-1)^{|\T \cup \{P_0\}|} \ind{n \in \cap \T} \ind{n\in P_0} = - \sum_{\substack{P_0\not\in S \subset \prohibprog\\ S\not\in \calX }} (-1)^{|S|} \ind{n \in \cap S},
                  \end{equation*}
                  using that $\calX$ is closed under containment in the first equality. \qedhere
        \end{enumerate}
    \end{proof}

    The next lemma shows some cancellation for combinatorial sums having up to $2^{2^{|\Omega|}}$ terms. The short proof below is due to Helfgott and Radziwiłł \cite{HR}.
    \begin{lemma}
        \label{lem:rota}
        Let $\mathcal{A}$ be any collection of subsets of a finite set $\Omega$. Then
        \begin{equation*}
            \abs{\sum_{\substack{\mathcal{B} \subset \mathcal{A}\\ \cup \mathcal{B} = \Omega}} (-1)^{|\mathcal{B}|} } \leq 2^{|\Omega|}.
        \end{equation*}
    \end{lemma}

    \begin{proof}
        Observe that, given two finite sets $\Omega_1\subset \Omega_2$, we have
        \begin{equation}
            \label{eq:trickforcancellation}
            (-1)^{|\Omega_1|} \sum_{\Omega_1 \subset W \subset \Omega_2} (-1)^{|W|} = \ind{\Omega_1=\Omega_2}.
        \end{equation}
        Indeed, this is obvious if $\Omega_1 = \Omega_2$, and if $\Omega_1\neq \Omega_2$ the left-hand side is the expanded form of $(1-1)^{|\Omega_2\setminus \Omega_1|}$.

        This allows us to write
        \begin{equation*}
            \abs{\sum_{\substack{\mathcal{B} \subset \mathcal{A}\\ \cup \mathcal{B} = \Omega}} (-1)^{|\mathcal{B}|} } = \abs{\sum_{\mathcal{B} \subset \mathcal{A} } (-1)^{|\mathcal{B}|} \sum_{\cup \mathcal{B} \subset W \subset \Omega} (-1)^{|W|} } = \abs{\sum_{W\subset \Omega } (-1)^{|W|} \sum_{\substack{\mathcal{B} \subset \mathcal{A}\\ \cup \mathcal{B} \subset W}} (-1)^{|\mathcal{B}|} }.
        \end{equation*}
        The inner sum has the shape of \cref{eq:trickforcancellation}, with $\Omega_1 = \emptyset$ and $\Omega_2 = \{A\in \mathcal{A} \mid A\subset W\}$, so is at most $1$ in absolute value. Since the outer sum has $\leq 2^{|\Omega|}$ terms, the claim follows.
    \end{proof}

    Assuming that the progressions in $\prohibprog$ have square-free moduli, and with an additional hypothesis on the shape of $\calX$, we can use \cref{lem:rota} to show that the two sums in \cref{eq:sieveidentity} exhibit some cancellation.

    \begin{proposition}
        \label{prop:sievesquare-free}
        Let $\prohibprog$ be a finite set of arithmetic progressions in $\Z$ with square-free moduli. Let
        \begin{equation*}
            \prohibprog^{\cap} := \{\cap S : S \subset \prohibprog\}.
        \end{equation*}

        Fix a non-empty subset $X \subset \prohibprog^{\cap}$ that is closed under containment, i.e.~if a progression $P\in \prohibprog^{\cap}$ is an element of $X$, then so are all $P'\in \prohibprog^{\cap}$ with $P' \supset P$. Let $\calX$ be the collection of subsets of $\prohibprog$ defined by\footnote{Note that $\emptyset \in \calX$ since $\cap \emptyset = \Z$ by convention.}
        \begin{equation*}
            \calX = \{S \subset \prohibprog: \cap S \in X\}.
        \end{equation*}
        Then
        \begin{equation}
            \label{eq:bonferronilike}
            \ind{n\not\in P\  \forall P\in \prohibprog}  = \sum_{S\in \calX} (-1)^{|S|} \ind{n \in \cap S}  + O \left(\sum_{R\in \partial X }  3^{\omega(\modulus_{R})} \ind{n\in R}\right)
        \end{equation}
        where
        \begin{equation*}
            \partial X := \{R\in \prohibprog^{\cap}: R\not\in X \text{ and }R = P\cap P'\text{ for some } P\in X \text{ and }P'\in \prohibprog \}.
        \end{equation*}
        Moreover, the first sum can be rewritten as
        \begin{equation*}
            \sum_{S\in \calX} (-1)^{|S|} \ind{n \in \cap S} = \sum_{P\in X } c_{P} \ind{n\in P}
        \end{equation*}
        for some coefficients $c_{P}\in \Z$ satisfying $|c_{P}| \leq 2^{\omega(\modulus_{P})}$.
    \end{proposition}

    \begin{proof}
        If the left-hand side of \cref{eq:bonferronilike} is $1$, then the equality \cref{eq:bonferronilike} is true by \cref{lem:sieveidentity}. On the other hand, if the left-hand side is $0$, then by \cref{lem:sieveidentity} we have
        \begin{equation*}
            \ind{n\not\in P\  \forall P\in \prohibprog}  = \sum_{S\in \calX} (-1)^{|S|} \ind{n \in \cap S}  +  \sum_{\substack{P_0\in S \subset \prohibprog\\ S \not\in \calX,\, S\setminus \{P_0\} \in \calX}} (-1)^{|S|} \ind{n\in \cap S},
        \end{equation*}
        where, in the last sum, $P_0 \in \prohibprog$ is an arbitrary progression containing $n$. We will bound the second sum at the end of this proof.

        Let us analyse the first sum. We have
        \begin{equation*}
            \sum_{S\in \calX} (-1)^{|S|} \ind{n \in \cap S} = \sum_{P\in X } c_{P} \ind{n\in P}
        \end{equation*}
        where, for $P\in X $,
        \begin{equation*}
            c_{P} := \sum_{\substack{S\in \calX\\ \cap S = P}} (-1)^{|S|}  = \sum_{\substack{S\subset \{P'\in \prohibprog : P'\supset P\} \\ \cap S = P}} (-1)^{|S|}.
        \end{equation*}

        Fix some $P\in X$. If $S$ is a set of progressions containing $P$, the condition $\cap S = P$ is equivalent to
        \begin{equation*}
            \mathrm{lcm} \{\modulus_{P'} : P'\in S\} = \modulus_{P}.
        \end{equation*}
        Since all progressions in $\prohibprog$ have square-free moduli, this is in turn equivalent to
        \begin{equation*}
            \bigcup_{P'\in S} \{p : p\mid \modulus_{P'}\} = \{p : p\mid \modulus_{P}\}.
        \end{equation*}
        Let $\Omega_P := \{p : p\mid \modulus_{P}\}$,
        \begin{equation*}
            \mathcal{A}_P := \Big\{ \{p : p\mid \modulus_{P'}\} : P'\in \prohibprog, P'\supset P\Big\}
        \end{equation*}
        and, for every set $S$ of progressions containing $P$, let
        \begin{equation*}
            \mathcal{B}_{P}(S) := \Big\{ \{p : p\mid \modulus_{P'}\} : P'\in S\Big\}.
        \end{equation*}
        Note that $\mathcal{B}_P(S)$ determines $S$, since a progression $P'\in \prohibprog$ with $P'\supset P$ is uniquely determined by its modulus $\modulus_{P'}$, which in turn is uniquely determined by its set of prime factors. Therefore,
        \begin{equation*}
            c_P =  \sum_{\substack{S\subset \{P'\in \prohibprog : P'\supset P\} \\ \cap S = P}} (-1)^{|S|} = \sum_{\substack{\mathcal{B} \subset \mathcal{A}_P \\ \cup \mathcal{B} = \Omega_P}} (-1)^{|\mathcal{B}|}.
        \end{equation*}
        By \cref{lem:rota}, we obtain $|c_P| \leq 2^{|\Omega_P|} = 2^{\omega(\modulus_{P})}$.

        We now turn to the remainder term. We suppose that $n\in P_0$ for some $P_0\in \prohibprog$. We operate a change of variables and write ${S' = S \setminus \{P_0\}}$, $P = \cap S'$ and ${R = \cap S}$. The conditions $S \not\in \calX$ and $S\setminus \{P_0\} \in \calX$ become $R\not\in X$ and $P\in X$, respectively. Hence, we have
        \begin{equation*}
            \sum_{\substack{P_0\in S \subset \prohibprog\\ S \not\in \calX,\, S\setminus \{P_0\} \in \calX}} (-1)^{|S|} \ind{n\in \cap S}
            = \sum_{\substack{R\in \prohibprog^{\cap} \\ R\not \in X }}  \ind{n\in R}  \sum_{\substack{P\in X \\ R = P\cap P_0}}\sum_{\substack{S'\in \calX\\ \cap S'= P}} (-1)^{|S'|+1}.
        \end{equation*}
        The inner sum is exactly $-c_P$, which is $O\big(2^{\omega(\modulus_{P})}\big)$. Recalling that, for fixed $R$, a progression $P\supset R$ is uniquely determined by its modulus $\modulus_{P}$, which divides $\modulus_{R}$, we have
        \begin{equation*}
            \sum_{\substack{P_0\in S \subset \prohibprog\\ S \not\in \calX,\, S\setminus \{P_0\} \in \calX}} (-1)^{|S|} \ind{n\in \cap S}
            = O \left(\sum_{R\in \partial X}  \ind{n\in R}  \sum_{d\mid \modulus_{R}} 2^{\omega(d)}\right).
        \end{equation*}
        The observation that $\sum_{d\mid m} 2^{\omega(d)} = 3^{\omega(m)}$ for all square-free $m\geq 1$ concludes the proof.
    \end{proof}

    \section{Sum without divisibility conditions}
    \label{appendix:probmodel}

    In this section we prove \cref{prop:differencesum}, which quickly follows from the next proposition.

    \begin{proposition}
        \label{prop:probmodel}
        Let $N \geq \frac{1}{H}\exp\!\big((\log H)^{2}\big)$ and let $\N := \mathbb{N} \cap (N, 2N]$.

        For $\indexset \subset \intJ$, define $\D_\indexset$ to be the set of all products $\prod_{i\in \indexset}p_i$ with $p_i\in \primes_i$ for all~$i$.
        Then, for all non-empty $\indexset \subset \intJ$, we have
        \begin{equation*}
            \sum_{n\in \N}  \sum_{\substack{d\in \D_\indexset}}\frac{1}{d} \lambda(n) \lambda(n+d) \ll \frac{\V^{\J} N}{(\log H)^{1/2000}}.
        \end{equation*}
    \end{proposition}

    \begin{proof}[Proof of \cref{prop:differencesum} assuming \cref{prop:probmodel}]
        We can expand the difference $S_2 - S_1$ as
        \begin{equation*}
            \sum_{\substack{\emptyset \neq \indexset \subset \intJ}} S(\indexset),
        \end{equation*}
        where
        \begin{equation*}
            S(\indexset) := \sum_{n\in \N} \sum_{(p_1,\ldots, p_{\J})\in \primes_1\times\cdots\times \primes_{\J}} \Bigg( \prod_{i\in \indexset} \frac{1}{p_i}\Bigg) \Bigg( \prod_{i\in \intJ\setminus \indexset} \ind{p_i \mid n}\Bigg) \lambda(n) \lambda(n+p_1\cdots p_{\J}).
        \end{equation*}
        Changing variables $n = md$ with $d = \prod_{i\in \intJ\setminus \indexset}p_i$ gives
        \begin{equation*}
            S(\indexset) = \sum_{d\in \D_{\intJ\setminus \indexset}}\lambda(d)^2 \sum_{\frac{N}{d}<m\leq \frac{2N}{d}}\sum_{d'\in \D_\indexset} \frac{1}{d'} \lambda(m) \lambda(m+d').
        \end{equation*}
        By \cref{prop:probmodel}, the double sum over $m$ and $d'$ is
        \begin{equation*}
            \ll \frac{\V^{\J} N/d}{(\log H)^{1/2000}}.
        \end{equation*}
        Hence,
        \begin{equation*}
            S(\indexset) \ll  \frac{\V^{\J} N}{(\log H)^{1/2000}}\sum_{d\in \D_{\intJ\setminus \indexset}}\frac{1}{d} \ll \frac{\Vtopower{2\J} N}{(\log H)^{1/2000}}
        \end{equation*}
        for every non-empty $\indexset \subset \intJ$.
        Therefore
        \begin{equation*}
            |S_2-S_1| \ll 2^{\J} \Vtopower{2\J} \frac{N}{(\log H)^{1/2000}}.
        \end{equation*}
        Note that $2^J\Vtopower{2\J} \ll (\log H)^{\epsone}$ by \cref{lem:primesets}, so $|S_2-S_1| \ll N (\log H)^{-1/2500}$ if $\epsone$ is sufficiently small.
    \end{proof}

    \begin{lemma}
        \label{lem:fourthmoment}
        Fix a non-empty $\indexset\subset \intJ$ and let $\D_\indexset$ be as in \cref{prop:probmodel}. Let $\Mvar \in [H_0, H]$. Define
        \begin{equation*}
            Q(\alpha) := \sum_{\substack{d\in \D_\indexset\\ d\in (\Mvar/2, \Mvar]}} \frac{e( \alpha d)}{d},
        \end{equation*}
        where, as usual, $e(x):=\exp(2\pi i x)$. Then,
        \begin{equation*}
            \int_{0}^1 |Q(\alpha)|^4 \, d\alpha \ll \frac{\Vtopower{4\J}}{\Mvar (\log \Mvar)^{4}}.
        \end{equation*}
    \end{lemma}

    \begin{proof}
        By Parseval's identity, we can expand
        \begin{equation*}
            \int_{0}^1 |Q(\alpha)|^4 d\alpha = \int_{0}^1 \Bigg\lvert \sum_{|m|\leq \Mvar} \Bigg(\sum_{\substack{d_1,d_2\in \D_{\indexset}\\ d_1, d_2\in (\Mvar/2, \Mvar] \\d_1-d_2=m}}\frac{1}{d_1d_2}\Bigg) e(m\alpha)\Bigg\rvert^2 d\alpha = \sum_{|m|\leq \Mvar}  \Bigg\lvert \sum_{\substack{d_1,d_2\in \D_{\indexset}\\ d_1, d_2\in (\Mvar/2, \Mvar] \\d_1-d_2=m}}\frac{1}{d_1d_2} \Bigg\rvert^2.
        \end{equation*}
        For $m=0$, the inner sum is trivially $\ll 1/{\Mvar}$.

        Fix $m> 0$. Let $N(m, b, \indexset, \Mvar)$ denote the number of pairs $(d_1, d_2)\in \D_{\indexset}\times \D_{\indexset}$ such that $d_1-d_2 = m$, $d_1\in (\Mvar/2,\Mvar]$ and ${\gcd(d_1, d_2, m) = b}$. Observe that $N(m, b, \indexset, \Mvar) = 0$ unless $b\mid m$ and $b\in \D_{\otherindexset}$ for some $\otherindexset\subset \indexset$, in which case we have
        \begin{equation*}
            N(m, b, \indexset, \Mvar) = N\left(\tfrac{m}{b}, 1, \indexset\setminus \otherindexset, \tfrac{\Mvar}{b}\right).
        \end{equation*}
        We thus are led to bound the number of coprime solutions $(e_1, e_2)\in \D_{\indexset\setminus \otherindexset}\times \D_{\indexset\setminus \otherindexset}$ to the equation $e_1-e_2 = m/b$ with $e_1\in (\Mvar/2b, \Mvar/b]$. Let $i_+$ be the largest element of $\indexset\setminus \otherindexset$. We can rewrite $e_i = n_i p_i$ where $n_i\in \D_{(\indexset\setminus \otherindexset) \setminus \{i_+\}}$ and $p_i\in \primes_{i_+}$ for $i\in \{1,2\}$. For fixed $n_1, n_2$, the number of solutions $(p_1, p_2)\in \primes_{i_+}\times \primes_{i_+}$ to the linear equation
        \begin{equation*}
            n_1p_1 - n_2p_2 = \frac{m}{b}
        \end{equation*}
        with $n_1p_1\in (\Mvar/2b, \Mvar/b]$ is
        \begin{equation*}
            \ll \frac{\Mvar/b}{\varphi(n_1)\varphi(n_2)(\log \Mvar/b)^2} \cdot \frac{m/b}{\varphi(m/b)}
        \end{equation*}
        by classical sieve theoretic methods, such as \cite[Proposition~6.22]{opera}. To apply this particular result, we used the fact that $\max(n_1, n_2) \leq (\Mvar/b)^{1/10}$, which holds by property \cref{item:primesets2} of \cref{lem:primesets}.

        Note that $\varphi(n) \gg n$ if $n$ is a product of $\leq \J$ primes, each $\geq H_0$. This is the case for $n_1$ and $n_2$. Therefore,
        \begin{equation*}
            N(m, b, \indexset, \Mvar) \ll  \sum_{n_1, n_2\in \D_{(\indexset\setminus \otherindexset)\setminus \{i_+\}}}\frac{\Mvar/b}{n_1n_2(\log \Mvar/b)^2} \cdot \frac{m/b}{\varphi(m/b)} \ll \Vtopower{2\J}\frac{\Mvar/b}{(\log \Mvar/b)^2} \cdot \frac{m}{\varphi(m)}.
        \end{equation*}
        We conclude that the total number of solutions $(d_1, d_2)$ to $d_1-d_2 = m$ with $d_1, d_2\in (\Mvar/2, \Mvar]$ is
        \begin{equation*}
            \ll \sum_{\substack{b\mid m\\ b< m}}\Vtopower{2\J}\frac{\Mvar/b}{(\log \Mvar/b)^2} \cdot \frac{m}{\varphi(m)} \ll \frac{\Vtopower{2\J} \Mvar}{(\log \Mvar)^2}\cdot \frac{\sigma_1(m)}{\varphi(m)}.
        \end{equation*}

        We thus obtain
        \begin{equation*}
            \int_{0}^1 |Q(\alpha)|^4 d\alpha \ll \frac{1}{\Mvar^2} + \left(\frac{\Vtopower{2\J}}{\Mvar(\log \Mvar)^2}\right)^2 \sum_{m=1}^{\Mvar} \left(\frac{\sigma_1(m)}{\varphi(m)}\right)^2 \ll \frac{\Vtopower{4\J}}{\Mvar (\log \Mvar)^4},
        \end{equation*}
        where we used the elementary estimate \cite[Corollary~3.6]{tenenbaum} in the last inequality.
    \end{proof}

    \begin{proof}[Proof of \cref{prop:probmodel}]
        Let $V_{[\Mvar]} :=\sum_{\substack{d\in \D_{\indexset}\cap (\Mvar/2, \Mvar]}}1/d$. It suffices to show that
        \begin{equation}
            \label{eq:dyadicsufficient}
            T_{\Mvar} := \sum_{n\in \N}  \sum_{\substack{d\in \D_{\indexset}\\ d\in (\Mvar/2, \Mvar]}}\frac{1}{d} \lambda(n) \lambda(n+d) \ll  \left(\frac{\V^{\J}}{\log \Mvar}+V_{[\Mvar]}\right) \frac{N}{(\log H)^{1/1750}}
        \end{equation}
        holds for all $\Mvar\in [H_0, H]$. Indeed, summing this inequality for $\Mvar\in  \{H2^{-j} : j\geq 0\}\cap [H_0, H]$ gives the desired upper bound
        \begin{equation*}
            \sum_{n\in \N}  \sum_{\substack{d\in \D_{\indexset}}}\frac{1}{d} \lambda(n) \lambda(n+d) \ll \Big(\V^{\J}(\log \log H)+\V^{\J} \Big) \frac{N}{(\log H)^{1/1750}} \ll \frac{\V^{\J} N}{(\log H)^{1/2000}}.
        \end{equation*}

        To prove \cref{eq:dyadicsufficient}, we start by introducing a new average over shifts $m\leq \Mvar$ and use the circle method:
        \begin{align*}
            T_{\Mvar} & = \frac{1}{\Mvar} \sum_{m\leq \Mvar} \sum_{n\in \N}   \sum_{\substack{d\in \D_{\indexset}              \\ d\in (\Mvar/2, \Mvar]}}\frac{1}{d} \lambda(n+m) \lambda(n+m+d) + O(\Mvar\V^{\J})\\
                      & = \frac{1}{\Mvar} \sum_{n\in \N} \int_0^1 Q(\alpha) F_n(\alpha) G_n(\alpha) d\alpha + O(\Mvar\V^{\J}),
        \end{align*}
        with $F_n(\alpha):=\sum_{m\leq \Mvar}\lambda(n+m)e(\alpha m)$, $G_n(\alpha):=\sum_{k\leq 2\Mvar}\lambda(n+k)e(-\alpha k)$ and $Q(\alpha)$ as in \cref{lem:fourthmoment}.
        The error term $O(\Mvar\V^{\J})$ is clearly negligible.

        Let $\eps>0$ be a parameter that will be fixed later, and let $E_\eps := \{\alpha\in [0, 1] : |Q(\alpha)| > \eps\}$. Outside of $E_{\eps}$, the function $|Q|$ is small and we have
        \begin{equation*}
            \sum_{n\in \N} \int_{[0, 1]\setminus E_{\eps}} |Q||F_n| |G_n| \leq \eps N \norm{F_n}_2 \norm{G_n}_2 \ll \eps \Mvar N.
        \end{equation*}
        On the other hand, the Lebesgue measure of $E_{\eps}$ is $\ll \Vtopower{4\J}/(\eps^4 \Mvar (\log \Mvar)^4)$ by \cref{lem:fourthmoment} and Markov's inequality. Hence
        \begin{equation*}
            \sum_{n\in \N} \int_{E_{\eps}} |Q||F_n| |G_n| \ll \frac{\Vtopower{4\J} \norm{Q}_{\infty} }{\eps^4 \Mvar (\log \Mvar)^4} \norm{\sum_{n\in \N} |F_n||G_n|}_{\infty}\!\! \ll \frac{\Vtopower{4\J}  \, V_{[\Mvar]}}{\eps^4 \Mvar (\log \Mvar)^4}\cdot \Mvar \norm{\sum_{n\in \N} |F_n|}_{\infty}.
        \end{equation*}
        We now make crucial use of \cite[Theorem~1.3]{MRT} to obtain
        \begin{equation*}
            \norm{\sum_{n\in \N} |F_n|}_{\infty} \!\! = \sup_{\alpha\in \R}\, \sum_{n\in \N} \abs{\sum_{n\leq n'\leq n+\Mvar} \lambda(n') e(\alpha n')} \ll \left((\log \Mvar)^{-1/2} + (\log N)^{-1/700} \right) \Mvar N.
        \end{equation*}
        Since $\Mvar \geq H_0$ and $\log N \gg (\log H)^2$, this upper bound is $\ll (\log H)^{-c'} \Mvar N$ where $c' = 1/350$.

        Putting everything together, we conclude that
        \begin{equation*}
            T_{\Mvar}  \ll \frac{1}{\Mvar}\left(\eps +\frac{\Vtopower{4\J}\, V_{[\Mvar]}}{\eps^4 (\log \Mvar )^4 (\log H)^{c'}} \right) \Mvar N.
        \end{equation*}
        Choosing $\eps = \V^{\J} (\log \Mvar )^{-1} (\log H)^{-c'/5}$, we deduce the claimed bound \cref{eq:dyadicsufficient}.
    \end{proof}
\end{appendices}

\nocite{mrtsigns}
\bibliography{Sections/article_chowla}
\bibliographystyle{amsplain}

\end{document}